\documentclass[11pt]{amsart}

\newtheorem{thm}{Theorem}[section]
\newtheorem*{thm*}{Theorem}

\newtheorem{claim}[thm]{Claim}

\newtheorem{cor}[thm]{Corollary}

\newtheorem{lem}[thm]{Lemma}
\newtheorem*{lem*}{Lemma}
\newtheorem{mainthm}{Theorem}
\newtheorem*{mainthm*}{Theorem}
\newtheorem{maincor}[mainthm]{Corollary}

\newtheorem{prop}[thm]{Proposition}

\theoremstyle{definition}

\newtheorem*{case*}{Case}

\newtheorem{defn}[thm]{Definition}
\newtheorem*{defn*}{Definition}

\newtheorem*{exmp*}{Example}

\newtheorem{hyp}[thm]{Hypothesis}

\renewcommand{\thestep}{}

\theoremstyle{remark}

\renewcommand{\thecase}{}

\newtheorem{rmk}[thm]{Remark}
\newtheorem*{rmk*}{Remark}

\makeatletter
\def\alphenumi{
  \def\theenumi{\alph{enumi}}
  \def\p@enumi{\theenumi}
  \def\labelenumi{(\@alph\c@enumi)}}
\makeatother

\makeatletter
\def\thecase{\@arabic\c@case}
\makeatother

\makeatletter
\def\thestep{\@arabic\c@step}
\makeatother

\makeatletter
\newcommand{\tpitchfork}{%
  \vbox{
    \baselineskip\z@skip
    \lineskip-.52ex
    \lineskiplimit\maxdimen
    \m@th
    \ialign{##\crcr\hidewidth\smash{$-$}\hidewidth\crcr$\pitchfork$\crcr}
  }%
}
\makeatother

\DeclareFontFamily{U}{mathx}{\hyphenchar\font45}
\DeclareFontShape{U}{mathx}{m}{n}{
      <5> <6> <7> <8> <9> <10>
      <10.95> <12> <14.4> <17.28> <20.74> <24.88>
      mathx10
      }{}
\DeclareSymbolFont{mathx}{U}{mathx}{m}{n}
\DeclareFontSubstitution{U}{mathx}{m}{n}
\DeclareMathAccent{\widecheck}{0}{mathx}{"71}
\DeclareMathAccent{\wideparen}{0}{mathx}{"75}

\newcount\hh
\newcount\mm
\mm=\time
\hh=\time
\divide\hh by 60
\divide\mm by 60
\multiply\mm by 60
\mm=-\mm
\advance\mm by \time
\def\hhmm{\number\hh:\ifnum\mm<10{}0\fi\number\mm}

\let\oldmarginpar\marginpar
\renewcommand\marginpar[1]{\-\oldmarginpar[\raggedleft\footnotesize #1]%
{\raggedright\footnotesize #1}}

\newcommand\CC{\mathbb{C}}

\newcommand\HH{\mathbb{H}}

\newcommand\KK{\mathbb{K}}

\newcommand\NN{\mathbb{N}}

\newcommand\PP{\mathbb{P}}

\newcommand\RR{\mathbb{R}}

\newcommand\cA{{\mathcal{A}}}

\newcommand\cF{{\mathcal{F}}}

\newcommand\calV{{\mathcal{V}}}
\newcommand\cW{{\mathcal{W}}}
\newcommand\cX{{\mathcal{X}}}

\newcommand\sE{{\mathscr{E}}}
\newcommand\sF{{\mathscr{F}}}
\newcommand\sG{{\mathscr{G}}}
\newcommand\sH{{\mathscr{H}}}

\newcommand\sK{{\mathscr{K}}}
\newcommand\sL{{\mathscr{L}}}
\newcommand\sM{{\mathscr{M}}}
\newcommand\sN{{\mathscr{N}}}
\newcommand\sO{{\mathscr{O}}}

\newcommand\sQ{{\mathscr{Q}}}
\newcommand\sR{{\mathscr{R}}}

\newcommand\sU{{\mathscr{U}}}
\newcommand\sV{{\mathscr{V}}}
\newcommand\sW{{\mathscr{W}}}
\newcommand\sX{{\mathscr{X}}}
\newcommand\sY{{\mathscr{Y}}}
\newcommand\sZ{{\mathscr{Z}}}

\newcommand\eps{\varepsilon}

\newcommand\GL{\operatorname{GL}}

\DeclareMathOperator{\PSL}{PSL}

\newcommand\SU{\operatorname{SU}}

\newcommand\less{\setminus}

\newcommand\Coker{\operatorname{Coker}}

\DeclareMathOperator{\Crit}{Crit}

\DeclareMathOperator{\Emb}{Emb}

\newcommand\grad{\operatorname{grad}}

\DeclareMathOperator{\Imm}{Imm}

\newcommand\Ind{\operatorname{Index}}

\DeclareMathOperator{\Iso}{Iso}

\newcommand\Ker{\operatorname{Ker}}

\newcommand\Map{\operatorname{Map}}

\newcommand\Ran{\operatorname{Ran}}

\newcommand\tr{\operatorname{tr}}

\newcommand\vol{\operatorname{vol}}
\newcommand\Vol{\operatorname{Vol}}

\newcommand\apriori{{\emph{a priori }}}
\newcommand\Apriori{{\emph{A priori }}}

\newcommand\Harm{{\mathrm{Harm}}}

\newcommand\id{{\mathrm{id}}}

\newcommand\round{{\mathrm{round}}}

\newcommand\sym{{\mathrm{sym}}}

\numberwithin{equation}{section}
\numberwithin{figure}{section}

\usepackage{footmisc}
\usepackage{graphicx}
\usepackage{hyperref}
\usepackage{mathrsfs}
\usepackage[usenames]{color}
\usepackage{paralist}
\usepackage{slashed}
\usepackage{txfonts}
\usepackage{url}
\usepackage{verbatim}
\usepackage{xr}

\usepackage{amsmath}
\usepackage{amscd}
\usepackage{amssymb}
\usepackage{lmodern}

\hypersetup{pdftitle={On the Morse--Bott Property of Analytic Functions on Banach Spaces with Lojasiewicz Exponent One Half}}
\hypersetup{pdfauthor={Paul M. N. Feehan}}

\begin{document}

\title[Morse--Bott Property for Analytic Functions with {\L}ojasiewicz Exponent one half]{On the Morse--Bott Property of Analytic Functions on Banach Spaces with {\L}ojasiewicz Exponent One Half}

\author[Paul M. N. Feehan]{Paul M. N. Feehan}
\address{Department of Mathematics, Rutgers, The State University of New Jersey, 110 Frelinghuysen Road, Piscataway, NJ 08854-8019, United States of America}
\email{feehan@math.rutgers.edu}

\date{This version: April 12, 2020, incorporating final galley proof corrections. \emph{Calculus of Variations and Partial Differential Equations} (2020), \url{https://doi.org/10.1007/s00526-020-01734-4}.}

\begin{abstract}
  It is a consequence of the Morse--Bott Lemma (see Theorems \ref{thm:Morse-Bott_Lemma_Banach} and \ref{thm:Morse-Bott_Lemma_Banach_refined}) that a $C^2$ Morse--Bott function on an open neighborhood of a critical point in a Banach space obeys a {\L}ojasiewicz gradient inequality with the optimal exponent one half. In this article we prove converses (Theorems \ref{mainthm:Analytic_function_Lojasiewicz_exponent_one-half_Morse-Bott_Banach}, \ref{mainthm:Analytic_function_Lojasiewicz_exponent_one-half_Morse-Bott_Banach_refined}, and Corollary \ref{maincor:Analytic_function_Lojasiewicz_exponent_one-half_Morse-Bott}) for analytic functions on Banach spaces: If the {\L}ojasiewicz exponent of an analytic function is equal to one half at a critical point, then the function is Morse--Bott and thus its critical set nearby is an analytic submanifold. The main ingredients in our proofs are the {\L}ojasiewicz gradient inequality for an analytic function on a finite-dimensional vector space \cite{Lojasiewicz_1965} and the Morse Lemma (Theorems \ref{mainthm:Hormander_C-6-1_Banach} and \ref{mainthm:Hormander_C-6-1_Banach_refined}) for functions on Banach spaces with degenerate critical points that generalize previous versions in the literature, and which we also use to give streamlined proofs of the {\L}ojasiewicz--Simon gradient inequalities for analytic functions on Banach spaces (Theorems \ref{mainthm:Lojasiewicz-Simon_gradient_inequality} and \ref{mainthm:Lojasiewicz-Simon_gradient_inequality_refined}).
\end{abstract}

\subjclass[2010]{Primary 32B20, 32C05, 32C18, 32C25, 58E05; secondary 14E15, 32S45, 57R45, 58A07, 58A35}

\keywords{Analytic varieties, {\L}ojasiewicz inequalities, gradient flow, Morse theory, Morse--Bott functions, resolution of singularities, semianalytic sets and subanalytic sets, singularities}

\thanks{The author was partially supported by National Science Foundation grant DMS-1510064 and the Dublin Institute for Advanced Studies.}

\maketitle
\tableofcontents

\section{Introduction}
\label{sec:Introduction}
Let $\KK$ be the field of real numbers $\RR$ or complex numbers $\CC$, and $d\geq 1$ be an integer, and let $\KK^{d*} = (\KK^d)^*$ denote the dual space. In order to better motivate our main results (Theorems \ref{mainthm:Analytic_function_Lojasiewicz_exponent_one-half_Morse-Bott_Banach}, \ref{mainthm:Analytic_function_Lojasiewicz_exponent_one-half_Morse-Bott_Banach_refined}, and Corollary \ref{maincor:Analytic_function_Lojasiewicz_exponent_one-half_Morse-Bott}), we begin by recalling the well-known

\begin{thm}[{\L}ojasiewicz gradient inequality for an analytic function]
\label{thm:Lojasiewicz_gradient_inequality}
Let $U \subset \KK^d$ be an open neighborhood of the origin and $f:U\to\KK$ be an analytic function. If $f(0) = 0$ and $f'(0) = 0$ then, after possibly shrinking $U$, there are constants $C \in (0, \infty)$ and $\theta \in [1/2,1)$ such that
\begin{equation}
\label{eq:Lojasiewicz_gradient_inequality}
\|f'(x)\|_{\KK^{d*}} \geq C|f(x)|^\theta, \quad\text{for all } x \in U.
\end{equation}
\end{thm}

{\L}ojasiewicz used the theory of semianalytic sets to prove Theorem \ref{thm:Lojasiewicz_gradient_inequality} as\footnote{The first page number refers to the version of {\L}ojasiewicz's original manuscript mimeographed by IHES while the page number in parentheses refers to the cited LaTeX version of his manuscript prepared by M. Coste and available on the Internet.}
\cite[Proposition 1, p. 92 (67)]{Lojasiewicz_1965} when $\KK=\RR$ and gave the range for $\theta$ as the interval $(0,1)$. His article remained unpublished, but Bierstone and Milman gave a simplified and streamlined exposition of {\L}ojasiewicz's method in \cite{BierstoneMilman} for $\KK=\RR$ and later gave an elegant and entirely new proof in \cite{Bierstone_Milman_1997} of \eqref{eq:Lojasiewicz_gradient_inequality} for $\KK=\RR$ using resolution of singularities for analytic varieties \cite{Hironaka_1964-I-II} and for which they also gave a new and significantly simplified proof. In \cite{BierstoneMilman, Bierstone_Milman_1997}, Bierstone and Milman state the range as for $\theta$ as the interval $(0,1)$.

In \cite{Feehan_lojasiewicz_inequality_all_dimensions}, we proved Theorem \ref{thm:Lojasiewicz_gradient_inequality} as \cite[Theorem 1]{Feehan_lojasiewicz_inequality_all_dimensions} by also appealing to resolution of singularities for analytic sets but in a different way from that of Bierstone and Milman \cite{Bierstone_Milman_1997}. Our approach is valid for both $\KK=\RR$ or $\CC$ and it allowed us to give the forthcoming partial identification \eqref{eq:Lojasiewicz exponent_monomial_function} of the {\L}ojasiewicz exponent, $\theta$, and show that it is restricted to the interval $[1/2,1)$, sharpening the range $(0,1)$ provided in \cite{BierstoneMilman, Bierstone_Milman_1997, Lojasiewicz_1965}. Resolution of singularities for analytic varieties yields the following special case of \cite[Theorem 4.5]{Feehan_lojasiewicz_inequality_all_dimensions} (see \cite[Sections 4.3 and 4.4]{Feehan_lojasiewicz_inequality_all_dimensions} for details of references to statements and proofs):

\begin{thm}[Monomialization of an analytic function]
\label{thm:Monomialization_analytic_function}
Let $U \subset \KK^d$ be an open neighborhood of the origin and $f:U\to\KK$ be an analytic function. If $f(0) = 0$ then, after possibly shrinking $U$, there are an open neighborhood $V \subset \KK^d$ of the origin and an analytic map,
\begin{equation}
\label{eq:Resolution_morphism}
\pi:V \ni y \mapsto x \in U,
\end{equation}
such that $\pi(0)=0$ and $\pi$ restricts to an analytic diffeomorphism on the complement of the zero set, $Z := f^{-1}(0)$,
\[
\pi:V \less \pi^{-1}(Z) \cong U \less Z,
\]
and $\pi^*f$ is a simple normal crossing function, that is,
\begin{equation}
\label{eq:Pullback_analytic_function_simple_normal_crossing}
\pi^*f(y) = y_1^{n_1}y_2^{n_2}\cdots y_d^{n_d}, \quad\text{for all } y \in V,
\end{equation}
where the $n_i$ are non-negative integers for $i=1,\ldots,d$.
\end{thm}

The {\L}ojasiewicz exponent of a monomial function can easily be computed exactly using the Generalized Young Inequality (see \cite[Remark 3.1]{Feehan_lojasiewicz_inequality_all_dimensions}) to give the

\begin{lem}[{\L}ojasiewicz exponent of a monomial function]
\label{lem:Lojasiewicz exponent_monomial_function}
(See Feehan \cite[Theorem 5]{Feehan_lojasiewicz_inequality_all_dimensions} or Haraux \cite[Theorem 3.1]{Haraux_2005}.)
Let $g:\KK^d\to\KK$ be an analytic function given by $g(y)=y_1^{n_1}y_2^{n_2}\cdots y_d^{n_d}$ for $y\in \KK^d$, where the $n_i$ are non-negative integers for $i=1,\ldots,d$. If $g(0)=0$ and $g'(0)=0$, then $g$ obeys the {\L}ojasiewicz gradient inequality \eqref{eq:Lojasiewicz_gradient_inequality} on $U=\KK^d$ for a constant $C\in(0,\infty)$ and exponent
\begin{equation}
\label{eq:Lojasiewicz exponent_monomial_function}
\theta = 1-\frac{1}{N} \in [1/2,1), \quad\text{where } N := \sum_{i=1}^{d}n_i \geq 2.
\end{equation}
\end{lem}

We apply Theorem \ref{thm:Monomialization_analytic_function} and Lemma \ref{lem:Lojasiewicz exponent_monomial_function} to prove Theorem \ref{thm:Lojasiewicz_gradient_inequality} using the elementary

\begin{lem}[{\L}ojasiewicz exponents and maps]
\label{lem:Pushforwards_preserve_Lojasiewicz_exponents}
Let $d$ and $e$ be positive integers, $V\subset \KK^e$ and $U\subset \KK^d$ be open neighborhoods of the origins and $\phi:V \to U$ be an open $C^1$ map such that $\phi(0) = 0$. If $f:U\to\KK$ is a $C^1$ function such that $\phi^*f$ obeys the {\L}ojasiewicz gradient inequality \eqref{eq:Lojasiewicz_gradient_inequality} at the origin with exponent $\theta \geq 0$ then, after possibly shrinking $U$, the function $f$ obeys the {\L}ojasiewicz gradient inequality \eqref{eq:Lojasiewicz_gradient_inequality} with the same exponent $\theta$ and a possibly smaller constant $C\in(0,\infty)$.
\end{lem}

Theorem \ref{thm:Lojasiewicz_gradient_inequality} now follows as an immediate corollary of Theorem \ref{thm:Monomialization_analytic_function} and Lemmas \ref{lem:Lojasiewicz exponent_monomial_function} and \ref{lem:Pushforwards_preserve_Lojasiewicz_exponents}. The exponent $\theta=1/2$ is optimal in the sense that if a solution $x(t)$, for $t\in[0,\infty)$, to the negative gradient flow defined by $f$,
\[
\frac{dx}{dt} = -\grad f(x(t)), \quad x(0) = x_0 \in U,
\]
converges to a point $x_\infty \in \Crit f$ as $t\to\infty$, then the norm of the difference, $\|x(t)-x_\infty\|_{\KK^d}$, converges to zero as $t\to\infty$ like $\exp(-ct)$ for some $c>0$ when $\theta=1/2$ but only like $t^{-\gamma}$ for some $\gamma>0$ when $\theta \in (1/2,1)$: see Appendix \ref{sec:Convergence_rate} for a discussion and references.

The optimal exponent, $\theta=1/2$, is achieved when $f:U\to\KK$ is a $C^2$ function that is Morse--Bott at the origin, that is, when the critical set $\Crit f := \{x\in U:f'(x) = 0\}$ is a connected, smooth submanifold of $U$ (after possibly shrinking $U$) of dimension equal to $\dim\Ker f''(0)$. This is readily seen by applying the Morse--Bott Lemma (see Theorems \ref{thm:Morse-Bott_Lemma_Banach} or \ref{thm:Morse-Bott_Lemma_Banach_refined}) to produce (after possibly shrinking $U$) a $C^2$ diffeomorphism $\Phi:V \to U$ from an open neighborhood $V\subset \KK^d$ of the origin onto $U$ such that $\Phi(0)=0$ and
\[
\Phi^*f(y) = \sum_{i=1}^{d-c}a_iy_i^2, \quad\text{for all } y \in V,
\]
where $c = \dim\Ker f''(0)$ and $a_i \in \KK\less\{0\}$ for $i=1,\ldots,d-c$. The {\L}ojasiewicz gradient inequality \eqref{eq:Lojasiewicz_gradient_inequality} for $f$ with exponent $\theta=1/2$ follows immediately by direct calculation for $\Phi^*f$ and invariance of the {\L}ojasiewicz exponent under diffeomorphisms. Proofs of the optimal {\L}ojasiewicz gradient inequality for Morse--Bott functions on $\KK^d$ or Banach spaces over $\KK$ were provided by the author in \cite[Theorem 3]{Feehan_lojasiewicz_inequality_all_dimensions} and \cite[Theorem 3]{Feehan_lojasiewicz_inequality_ground_state} and by the author and Maridakis \cite[Theorems 3 and 4]{Feehan_Maridakis_Lojasiewicz-Simon_Banach} without relying on the Morse--Bott Lemma.

The main goal of this article is explore whether the converse is true:
\begin{quote}
\emph{If $f:U\to\KK$ is a $C^2$ function that obeys the {\L}ojasiewicz gradient inequality \eqref{eq:Lojasiewicz_gradient_inequality} with exponent $\theta = 1/2$, then is $f$ Morse--Bott at the origin?}
\end{quote}
As we shall see in Theorems \ref{mainthm:Analytic_function_Lojasiewicz_exponent_one-half_Morse-Bott_Banach}, \ref{mainthm:Analytic_function_Lojasiewicz_exponent_one-half_Morse-Bott_Banach_refined} and Corollary \ref{maincor:Analytic_function_Lojasiewicz_exponent_one-half_Morse-Bott}, this converse is indeed true in great generality --- for a broad class of analytic functions, $f:\sX\supset\sU\to\KK$, on Banach spaces $\sX$ over $\KK$ and for any analytic function $f$ when $\sX=\KK^d$. It is apparent from examples that $\theta \in [1/2,1)$ provides a measure of complexity of the singularity of the critical set of $f$, at the origin. Theorems \ref{mainthm:Analytic_function_Lojasiewicz_exponent_one-half_Morse-Bott_Banach}, \ref{mainthm:Analytic_function_Lojasiewicz_exponent_one-half_Morse-Bott_Banach_refined} and Corollary \ref{maincor:Analytic_function_Lojasiewicz_exponent_one-half_Morse-Bott} make this informal measure of complexity of the singularity precise for analytic functions on open neighborhoods of the origin in $\KK^d$ with arbitrary $d\geq 1$ and even Banach spaces over $\KK$: the critical set is an analytic submanifold of the expected dimension when $\theta=1/2$.

Intuition supporting the preceding conclusion can be obtained by examining the structure of the function $\pi^*f$ in \eqref{eq:Pullback_analytic_function_simple_normal_crossing} when $N=2$, the lowest possible total degree of the monomial. Indeed, if $N=2$, then (after relabeling coordinates) either $n_1=2$ and $n_i=0$ for all $i\geq 2$ or $n_1=n_2=1$ and $n_i=0$ for all $i \geq 3$ and thus\footnote{We omit the pair of possible signs, $\pm$, when $\KK=\CC$.}
\begin{equation}
\label{eq:Pullback_analytic_function_simple_normal_crossing_Lojasiewicz_one_half}
\pi^*f(y) = \pm y_1^2 \quad\text{or}\quad y_1y_2, \quad\text{for all } y \in V,
\end{equation}
together with
\[
\pi^{-1}(f^{-1}(0)) = \{y\in V: y_1 = 0\} \quad\text{or}\quad \{y\in V: y_1 = 0 \text{ or } y_2 = 0\}.
\]
In particular, the critical set of $\pi^*f$ is either the codimension-one submanifold $\{y\in V: y_1 = 0\}$ or the codimension-two submanifold $\{y\in V: y_1 = y_2 = 0\}$. These observations tell us that the condition $\theta=1/2$ imposes strong constraints on resolution morphism $\pi$ and the structure of the analytic function $f$ itself since resolution of singularities tends to `increase degrees'.

Because the identification \eqref{eq:Pullback_analytic_function_simple_normal_crossing} of $\pi^*f$ as a simple normal crossing function comes from resolution of singularities for analytic sets, one might expect that methods of algebraic geometry could be used to compute $\theta$ directly in terms of $f$ and also lead to the conclusion that $f$ must be Morse--Bott, at least when $f$ is a polynomial and possibly even when $f$ is analytic. However, while the {\L}ojasiewicz exponent has been estimated for certain classes of polynomials (see \cite[Section 1]{Feehan_lojasiewicz_inequality_all_dimensions} for a survey), it appears difficult to estimate the exponent in any generality, even for polynomial functions. Using the fact that $\pi^*f(y) = \pm y_1^2$ or $y_1y_2$ when $f$ has {\L}ojasiewicz exponent $1/2$ to directly constrain the structure of $f$ and the resolution morphism $\pi$ in the proof of resolution of singularities appears challenging, although this may provide one route to a proof of Corollary \ref{maincor:Analytic_function_Lojasiewicz_exponent_one-half_Morse-Bott} using methods of algebraic geometry.

Our approach to proving Theorems \ref{mainthm:Analytic_function_Lojasiewicz_exponent_one-half_Morse-Bott_Banach}, \ref{mainthm:Analytic_function_Lojasiewicz_exponent_one-half_Morse-Bott_Banach_refined} in this article is analytic and relies on a version\footnote{I am indebted to Michael Greenblatt and Andr{\'a}s N{\'e}methi for pointing out to me that this should be a key analytical tool.} (see Section \ref{subsec:Morse_lemma_functions_Banach_space_degenerate_critical_points}) of the Morse Lemma for \emph{analytic} functions $f$ with \emph{degenerate} critical points, together with our identification \eqref{eq:Pullback_analytic_function_simple_normal_crossing_Lojasiewicz_one_half} of $\pi^*f$ when $f$ has {\L}ojasiewicz exponent $1/2$ and $\pi$ is a resolution of singularities \eqref{eq:Resolution_morphism} for the zero set $f^{-1}(0)$.

The concept of a Morse--Bott function was introduced by Bott in \cite[Definition, p. 248]{Bott_1954} and used by him in his first proof of the Bott Periodicity Theorem  \cite{Bott_1959}. Morse--Bott functions were employed by Austin and Braam \cite[Section 3]{Austin_Braam_1995} in their approach to developing a Morse theoretic approach to equivariant cohomology.

\subsection{Morse--Bott property of analytic functions with {\L}ojasiewicz exponent one half}
\label{subsec:Main_results_Morse-Bott_property_functions_Lojasiewicz_exponent_half}
Let $\sX,\sY$ be Banach spaces over $\KK$, and $\sL(\sX,\sY)$ denote the Banach space of bounded linear operators from $\sX$ to $\sY$, and $\Ker A$ and $\Ran A$ denote the kernel and range of $A \in \sL(\sX,\sY)$, and $\sX^*$ denote the continuous dual space of $\sX$. Let $\sL_\sym(\sX, \sX^*) \subset \sL(\sX, \sX^*)$ denote the closed subspace of operators $A$ that are symmetric in the sense that $\langle x, Ay\rangle_{\sX\times\sX^*} = \langle y, Ax\rangle_{\sX\times\sX^*}$ for all $x, y \in \sX$, where $\langle \cdot, \cdot\rangle_{\sX\times\sX^*}$ denotes the canonical pairing, $\sX\times\sX^* \ni (x,\alpha) \mapsto \alpha(x) \in \KK$. We recall the canonical identifications,
\[
\sL_\sym(\sX, \sX^*) = \sL_\sym^2(\sX,\KK) = \sL_\sym(\sX\otimes\sX,\KK),
\]
where $\sL^n(\sX,\KK)$ (respectively, $\sL(\otimes^n\sX,\KK)$) is the Banach space of continuous $n$-linear (respectively, linear) functions, $A:\times^n\sX\to\KK$ (respectively, $A:\otimes^n\sX\to\KK$), for integers $n\geq 1$.

Let $\cX$ be a $C^p$ Banach manifold ($p\geq 0$) modeled on a Banach space $\sX$ (see \cite[Section II.1]{Lang_fundamentals_differential_geometry} or \cite[Definitions 3.1.1 and 3.1.3]{AMR}). A subset $\cW \subset \cX$ is a $C^p$ Banach \emph{submanifold} \cite[Section II.2]{Lang_fundamentals_differential_geometry}, \cite[Definition 3.2.1]{AMR} modeled on a Banach space $\sW$ if there is a Banach space $\sN$ such that $\sX=\sW\oplus\sN$ as a direct sum of Banach spaces and, for each point $w \in \cW$, a chart $\psi:\cX \supset \calV \to \sX$ such that $\psi(w)=0$ and
\[
  \psi(\calV\cap\cX) = \psi(\calV)\cap(\sW\oplus\{0\}).
\]
Observe that $T_w\cX \cong \sX$ and $T_w\cW \cong \sW$ with $T_w\cX \cong T_w\cW \oplus \sN$; in particular, $T_w\cW$ is a closed subspace with closed complement in $T_w\cX$.

If $\sU\subset\sX$ is an open subset, $f:\sU\to\KK$ is a $C^2$ function, and $\Crit f = \{x\in \sU:f'(x) = 0\}$ is a $C^2$, connected submanifold of $\sU$, then\footnote{For example, see the discussion prior to Theorem \ref{thm:Morse-Bott_Lemma_Banach}.} the tangent space $T_x\Crit f$ is contained in $\Ker f''(x)$, for each $x \in \Crit f$, where $f'(x)\in\sX^*$ and $f''(x)\in\sL_\sym(\sX,\sX^*)$.

\begin{defn}[Morse--Bott properties]
\label{defn:Morse-Bott_function}
Let $\sX$ be a Banach space over $\KK$, and $\sU \subset \sX$ be an open neighborhood of the origin, and $f:\sU\to\KK$ be a $C^2$ function such that $\Crit f$ is a $C^2$, connected submanifold.
\begin{enumerate}
  \item
  \label{item:Morse-Bott_point}
  If $x_0 \in \Crit f$ and $\Ker f''(x_0) \subset \sX$ has a closed complement $\sX_0$ and $\Ran f''(x_0) = \sX_0^*$, and $T_{x_0}\Crit f = \Ker f''(x_0)$, then $f$ is \emph{Morse--Bott at the point} $x_0$;
  \item
  \label{item:Morse-Bott_critical_set}
  If $f$ is \emph{Morse--Bott at each point} $x \in \Crit f$, then $f$ is \emph{Morse--Bott along $\Crit f$} or a \emph{Morse--Bott function}.
\end{enumerate}
\end{defn}

\begin{rmk}[On the assumption that $\Ker f''(x_0)$ has a closed complement]
\label{rmk:Existence_closed_complement}  
Because $\Crit f \subset \sX$ is a submanifold in Definition \ref{defn:Morse-Bott_function}, then $T_{x_0}\Crit f$ automatically has a closed complement in $\sX$ and because $T_{x_0}\Crit f = \Ker f''(x_0)$ in Definition \ref{defn:Morse-Bott_function} \eqref{item:Morse-Bott_point}, then $\Ker f''(x_0)$ automatically has a closed complement in $\sX$. However, we include the closed complement assumption for $\Ker f''(x_0)$ in our definition for the sake of emphasis.
\end{rmk}  

The Morse--Bott Lemma (see Theorem \ref{thm:Morse-Bott_Lemma_Banach}) implies that if $f$ is Morse--Bott at a point, as in Definition \ref{defn:Morse-Bott_function} \eqref{item:Morse-Bott_point}, then $f$ is Morse--Bott along an open neighborhood of that point in $\Crit f$, as in Definition \ref{defn:Morse-Bott_function} \eqref{item:Morse-Bott_critical_set}.
If $\Crit f$ in Definition \ref{defn:Morse-Bott_function} consists of isolated points and $f''(x_0) \in \sL(\sX,\sX^*)$ is invertible for each $x_0 \in \Crit f$, then $f$ is a \emph{Morse function}. The finite-dimensional analogue of Definition \ref{defn:Morse-Bott_function} \eqref{item:Morse-Bott_critical_set} is well-known.

\begin{rmk}[Morse--Bott functions on Euclidean space]
\label{rmk:Morse-Bott_function_finite_dimension}
When $\sX$ is finite-dimensional, the definition of a Morse--Bott function was given by Bott \cite[Definition, p. 248]{Bott_1954}, \cite{Bott_1959}. See Nicolaescu \cite[Definition 2.41]{Nicolaescu_morse_theory} for a modern exposition.
\end{rmk}

When $\sX$ is infinite-dimensional, then one must impose hypotheses on $f$ in addition to those of Bott in the finite-dimensional case in order to obtain a tractable version, such as Theorem \ref{thm:Morse-Bott_Lemma_Banach}, of the classical Morse--Bott Lemma (for example, Nicolaescu \cite[Proposition 2.42]{Nicolaescu_morse_theory}). Because the operator $f''(x_0) \in \sL(\sX,\sX^*)$ is symmetric, the forthcoming Lemma \ref{lem:Range_symmetric_operator_kernel_has_closed_complement} \eqref{item:Range_symmetric_operator_kernel_has_closed_complement} implies that one always has $\Ran f''(x_0) \subset \sX_0^*$ if $\sX_0$ is a closed complement of $\Ker f''(x_0)$. Thus Item \eqref{item:Morse-Bott_point} in Definition \ref{defn:Morse-Bott_function} imposes the non-degeneracy condition $\Ran f''(x_0) = \sX_0^*$, like in the finite-dimensional case, given the property that $\Ker f''(x_0)$ has a closed complement.

When the operator $f''(x_0) \in \sL(\sX,\sX^*)$ is Fredholm with index zero, Lemma \ref{lem:Range_symmetric_operator_kernel_has_closed_complement} \eqref{item:Range_Fredholm_operator_index_zero} yields the non-degeneracy condition, $\Ran f''(x_0) = \sX_0^*$. When the operator $f''(x_0) \in \sL(\sX,\sX^*)$ is not Fredholm but $\sX$ is reflexive and $\Ker f''(x_0)$ has a closed complement, Lemma \ref{lem:Isomorphism_properties_symmetric_operator} implies that the condition that $\Ran f''(x_0)=\sX_0^*$ in Item \eqref{item:Morse-Bott_point} of Definition \ref{defn:Morse-Bott_function} is equivalent to the condition that $\Ran f''(x_0)\subset\sX^*$ be a closed subspace. We now state the main results of this article.

\begin{mainthm}[Morse--Bott property of an analytic function with {\L}ojasiewicz exponent one half]
\label{mainthm:Analytic_function_Lojasiewicz_exponent_one-half_Morse-Bott_Banach}
Let $\sX$ be a Banach space over $\KK$, and $\sU \subset \sX$ be an open neighborhood of the origin, and $f:\sU\to\KK$ be a
non-constant analytic function such that $f(0) = 0$ and $f'(0) = 0$ and $f''(0) \in \sL(\sX,\sX^*)$ is a Fredholm operator with index zero. If there is a constant $C \in (0,\infty)$ such that, after possibly shrinking $\sU$,
\begin{equation}
\label{eq:Lojasiewicz_gradient_inequality_dual_space}
\|f'(x)\|_{\sX^*} \geq C|f(x)|^{1/2}, \quad\text{for all } x \in \sU,
\end{equation}
then $f$ is a Morse--Bott function in the sense of Definition \ref{defn:Morse-Bott_function}.
\end{mainthm}

Hence, Theorem \ref{mainthm:Analytic_function_Lojasiewicz_exponent_one-half_Morse-Bott_Banach} is a converse to the simpler Theorem \ref{mainthm:Lojasiewicz_gradient_inequality_Morse-Bott} when $f$ is \emph{analytic}. The conclusion of Theorem \ref{mainthm:Analytic_function_Lojasiewicz_exponent_one-half_Morse-Bott_Banach} implies that (using the version of the Morse--Bott Lemma provided by Theorem \ref{thm:Morse-Bott_Lemma_Banach}), after possibly shrinking $\sU$, there are an open neighborhood $\sV\subset\sX$ of the origin and an analytic diffeomorphism $\Phi:\sV\to\sU$ such that $\Phi(0)=0$ and
\[
f(\Phi(y)) = \frac{1}{2}\langle y,Ay\rangle_{\sX\times\sX^*}, \quad\text{for all } y \in \sV,
\]
where $A = (f\circ\Phi)''(0) \in \sL_\sym(\sX,\sX^*)$ and, letting $\sK := \Ker A \subset \sX$ denote the finite-dimensional kernel of $A$ with closed complement $\sX_0 \subset \sX$, and $\Ran A = \sX_0^* \subset \sX^*$ (see Lemma \ref{lem:Range_symmetric_operator_kernel_has_closed_complement} \eqref{item:Range_Fredholm_operator_index_zero}) denote the closed range of $A$ with finite-dimensional complement $\sK^*$ in $\sX^*=\sX_0^*\oplus\sK^*$ (see Lemma \ref{lem:Dual_direct_sum_Banach_spaces_is_direct_sum_dual_spaces}), we have
\[
A = \begin{pmatrix} A_0 & 0 \\ 0 & 0 \end{pmatrix}:\sX_0\oplus\sK \to \sX_0^*\oplus\sK^*,
\]
where $A_0 \in \sL_\sym(\sX_0,\sX_0^*)$ is an isomorphism of Banach spaces. Thus, $(f\circ\Phi)'(y) = Ay$, for all $y\in\sV$, and $\Crit f\circ\Phi = \sV\cap\Ker A$, an analytic submanifold of $\sV$ of dimension equal to $\dim\Ker A$.

As explained in \cite[Section 1.1]{Feehan_Maridakis_Lojasiewicz-Simon_Banach}, the hypotheses of Theorem \ref{mainthm:Analytic_function_Lojasiewicz_exponent_one-half_Morse-Bott_Banach} are restrictive since they imply that $\sX$ is isomorphic to its continuous dual space, $\sX^*$. For example, that would exclude familiar choices such as the Banach space of $C^{2,\alpha}$ sections of a finite-rank Riemannian vector bundle over a closed finite-dimensional Riemannian manifold, as in \cite{Adams_Simon_1988, Simon_1983, Simon_1985}.

(There are examples of Banach spaces that are isomorphic to their dual spaces but are not isomorphic to Hilbert spaces.
However, even the implication that $\sX$ is a reflexive Banach space is already restrictive for some applications of infinite-dimensional Morse Theory to geometric analysis. A classical theorem of Lindenstrauss and Tzafriri \cite{Lindenstrauss_Tzafriri_1971} asserts that a real Banach space in which every closed subspace is complemented (that is, is the range of a bounded linear projection) is isomorphic to a Hilbert space.)

As in \cite{Simon_1983}, one can relax the implicit restriction that $\sX\cong\sX^*$ by introducing an extrinsic gradient operator $\sM(x)$ to represent the derivative $f'(x)$ for each $x\in\sU$.

\begin{defn}[Gradient map]
\label{defn:Huang_2-1-1}
(See Berger \cite[Section 2.5]{Berger_1977} or Huang \cite[Definition 2.1.1]{Huang_2006}.)
Let $\sX$ and $\tilde{\sX}$ be Banach spaces over $\KK$, and $\tilde{\sX}\subset\sX^*$ be a continuous embedding, and $\sU \subset \sX$ be an open subset. A continuous map $\sM:\sU\to \tilde{\sX}$ is a \emph{gradient map} if there is a $C^1$ \emph{potential function} $f:\sU\to\KK$ such that
\begin{equation}
\label{eq:Differential_and_gradient_maps}
f'(x)v = \langle v,\sM(x)\rangle_{\sX\times\sX^*}, \quad \text{for all } x \in \sU \text{ and } v \in \sX.
\end{equation}
\end{defn}

A continuous embedding of Banach spaces $\tilde{\sX} \subset \sX^*$ induces a continuous embedding of Banach spaces of bounded linear operators,
\[
\sL(\sX,\tilde{\sX}) \subset \sL(\sX,\sX^*),
\]
since if $T \in \sL(\sX,\tilde{\sX})$, then we obtain $T \in \sL(\sX,\sX^*)$ by composing $T$ with the continuous embedding $\tilde{\sX} \subset \sX^*$. We can therefore define
\[
\sL_\sym(\sX,\tilde{\sX}) := \sL(\sX,\tilde{\sX}) \cap \sL_\sym(\sX,\sX^*).
\]
Some basic properties of gradient maps are listed in Proposition \ref{prop:Huang_2-1-2}, including the fact that $\sM(x) \in \sL_\sym(\sX,\tilde{\sX})$ for all $x\in\sU$. When $\tilde{\sX} = \sX^*$ in Definition \ref{defn:Huang_2-1-1}, then the derivative and gradient maps coincide. If we are given a $C^1$ function $f:\sU\to\KK$ such that $f'(x)=\langle\cdot,\sM(x)\rangle_{\sX\times\sX^*}$, for all $x\in\sU$, for a $C^0$ map $\sM:\sU\to \tilde{\sX}$, then we simply write $f'(x)=\sM(x)$, for all $x\in\sU$. These observations motivate the following generalization of Definition \ref{defn:Morse-Bott_function}.

\begin{defn}[Generalized Morse--Bott properties]
\label{defn:Morse-Bott_function_refined}
Let $\sX$ and $\tilde{\sX}$ be Banach spaces over $\KK$, and $\tilde{\sX}\subset\sX^*$ be a continuous embedding, and $\sU \subset \sX$ be an open subset, and $f:\sU\to\KK$ be a $C^2$ function such that $\Crit f$ is a $C^2$, connected submanifold and $f'(x) \in \tilde{\sX}$ for all $x \in \sU$.
\begin{enumerate}
  \item
  \label{item:Morse-Bott_point_refined}
  If $x_0 \in \Crit f$ and $\Ker f''(x_0)$ has a closed complement\footnote{As explained in Remark \ref{rmk:Existence_closed_complement}, this assumption is automatically obeyed for the same reasons as in Definition \ref{defn:Morse-Bott_function}.} and $\Ran f''(x_0) = \tilde{\sX}$ and $T_{x_0}\Crit f = \Ker f''(x_0)$, then $f$ is \emph{Morse--Bott at the point} $x_0$;
  \item
  \label{item:Morse-Bott_critical_set_refined}
  If $f$ is Morse--Bott at each point $x \in \Crit f$, then $f$ is \emph{Morse--Bott along $\Crit f$} or a \emph{Morse--Bott function}.
\end{enumerate}
\end{defn}

Because $f''(x_0) \in \sL(\sX,\sX^*)$, then $\sK := \Ker f''(x_0) \subset \sX$ is a closed subspace, the quotient $\sX/\sK$ is a Banach space, and the induced operator $f''(x_0) \in \sL(\sX/\sK,\tilde{\sX})$ is an isomorphism by the Open Mapping Theorem. However, in our proof of the Morse--Bott Lemma (see Theorem \ref{thm:Morse-Bott_Lemma_Banach_refined}) for functions that are Morse--Bott at a point in the sense of Definition \ref{defn:Morse-Bott_function_refined} \eqref{item:Morse-Bott_point_refined}, we shall exploit the existence of a splitting $\sX = \sX_0\oplus\sK$, where $\sX_0 \subset \sX$ is a closed subspace.

Again, the Morse--Bott Lemma (see Theorem \ref{thm:Morse-Bott_Lemma_Banach_refined}) implies that if $f$ is Morse--Bott at a point, as in Definition \ref{defn:Morse-Bott_function_refined} \eqref{item:Morse-Bott_point_refined}, then $f$ is Morse--Bott along an open neighborhood of that point in $\Crit f$, as in Definition \ref{defn:Morse-Bott_function_refined} \eqref{item:Morse-Bott_critical_set_refined}.

\begin{mainthm}[Generalized Morse--Bott property of an analytic function with {\L}ojasiewicz exponent one half]
\label{mainthm:Analytic_function_Lojasiewicz_exponent_one-half_Morse-Bott_Banach_refined}
Let $\sX$ and $\tilde{\sX}$ be Banach spaces over $\KK$, and $\tilde{\sX}\subset\sX^*$ be a continuous embedding, and $\sU \subset \sX$ be an open neighborhood of the origin, and $f:\sU\to\KK$ be a non-constant analytic function such that $f(0)=0$ and $f'(0) = 0$ and $f'(x) \in \tilde{\sX}$ for all $x \in \sU$ and $f''(0) \in \sL(\sX, \tilde{\sX})$ is a Fredholm operator with index zero. If there is a constant $C \in (0,\infty)$ such that, after possibly shrinking $\sU$,
\begin{equation}
\label{eq:Lojasiewicz_gradient_inequality_dual_space_refined}
\|f'(x)\|_{\tilde{\sX}} \geq C|f(x)|^{1/2}, \quad\text{for all } x \in \sU,
\end{equation}
then $f$ is a Morse--Bott function in the sense of Definition \ref{defn:Morse-Bott_function_refined}.
\end{mainthm}

Hence, Theorem \ref{mainthm:Analytic_function_Lojasiewicz_exponent_one-half_Morse-Bott_Banach_refined} is a converse to the simpler Theorem \ref{mainthm:Lojasiewicz_gradient_inequality_Morse-Bott_refined} when $f$ is \emph{analytic} and, moreover, immediately yields Theorem \ref{mainthm:Analytic_function_Lojasiewicz_exponent_one-half_Morse-Bott_Banach} upon choosing $\tilde{\sX}=\sX^*$.

The conclusion of Theorem \ref{mainthm:Analytic_function_Lojasiewicz_exponent_one-half_Morse-Bott_Banach_refined} has an interpretation similar to that of Theorem \ref{mainthm:Analytic_function_Lojasiewicz_exponent_one-half_Morse-Bott_Banach}. Using the more general version of the Morse--Bott Lemma provided by Theorem \ref{thm:Morse-Bott_Lemma_Banach_refined}, after possibly shrinking $\sU$, there are an open neighborhood $\sV\subset\sX$ of the origin and an analytic diffeomorphism $\Phi:\sV\to\sU$ such that $\Phi(0)=0$ and
\[
f(\Phi(y)) = \frac{1}{2}\langle y,Ay\rangle_{\sX\times\sX^*}, \quad\text{for all } y \in \sV,
\]
where $A \in \sL_\sym(\sX,\tilde{\sX})$ and, letting $\sK := \Ker A \subset \sX$ denote the finite-dimensional kernel of $A$ with closed complement $\sX_0 \subset \sX$ and $\tilde{\sX}_0 := \Ran A \subset \tilde{\sX}$ denote the closed range of $A$ with finite-dimensional complement $\tilde{\sK} \cong \sK = \Ker A$, and $\tilde{\sX}_0 \cong \sX_0$ (see Lemma \ref{lem:Isomorphism_properties_Fredholm_operator}), we have
\[
A = \begin{pmatrix} A_0 & 0 \\ 0 & 0 \end{pmatrix}:\sX_0\oplus\sK \to \tilde{\sX}_0\oplus\tilde{\sK},
\]
where $A_0 \in \sL(\sX_0,\tilde{\sX}_0)$ is an isomorphism of Banach spaces that is symmetric with respect to the continuous embedding $\tilde{\sX}_0 \subset \sX_0^*$ and canonical pairing $\sX_0\times\sX_0^*\to\KK$. Indeed, because $\tilde{\sX} \subset \sX^*$ is a continuous embedding and the splitting $\sX = \sX_0\oplus\sK$ yields a splitting $\sX^* = \sX_0^*\oplus\sK^*$ by Lemma \ref{lem:Dual_direct_sum_Banach_spaces_is_direct_sum_dual_spaces} while $\tilde{\sX} = \tilde{\sX}_0\oplus\tilde{\sK} \cong  \tilde{\sX}_0\oplus\sK$, we obtain a continuous embedding $\tilde{\sX}_0 \subset \sX_0^*$, as implied above. Thus, $(f\circ \Phi)'(y) = Ay$, for all $y\in\sV$, and $\Crit f\circ\Phi = \sV\cap\Ker A$, an analytic submanifold of $\sV$ of dimension equal to $\dim\Ker A$.

When we specialize Theorem \ref{mainthm:Analytic_function_Lojasiewicz_exponent_one-half_Morse-Bott_Banach} to $\sX=\KK^d$, we obtain the desired characterization of the optimal {\L}ojasiewicz exponent for analytic functions on finite-dimensional vector spaces.

\begin{maincor}[Morse--Bott property of an analytic function with {\L}ojasiewicz exponent one half]
\label{maincor:Analytic_function_Lojasiewicz_exponent_one-half_Morse-Bott}
Let $d \geq 1$ be an integer, $U \subset \KK^d$ be an open neighborhood of the origin, and $f:U\to\KK$ be a non-constant analytic function such that $f(0) = 0$ and $f'(0) = 0$. If there is a constant $C \in (0, \infty)$ such that, after possibly shrinking $U$, the function $f$ obeys the {\L}ojasiewicz gradient inequality \eqref{eq:Lojasiewicz_gradient_inequality} with exponent $\theta=1/2$, then $f$ is a Morse--Bott function.
\end{maincor}

The interpretation of Corollary \ref{maincor:Analytic_function_Lojasiewicz_exponent_one-half_Morse-Bott} is simpler than that of Theorem \ref{mainthm:Analytic_function_Lojasiewicz_exponent_one-half_Morse-Bott_Banach}. By the Morse--Bott Lemma (Theorem \ref{thm:Morse-Bott_Lemma_Banach} with $\sX=\KK^d$ and diagonalization \cite[p. 278]{Horn_Johnson_topics_matrix_analysis_1994} of the symmetric matrix $A$ over $\KK$), after possibly shrinking $U$, there are an open neighborhood $V\subset\KK^d$ of the origin and an analytic diffeomorphism $\Phi:V\to U$ with $\Phi(0)=0$ and an integer $c$ obeying $0 \leq c \leq d-1$ such that
\[
f(\Phi(y)) = \frac{1}{2}\sum_{i=1}^{d-c}a_iy_i^2, \quad\text{for all } y \in V,
\]
where $a_i\in\KK\less\{0\}$, for $i=1,\ldots,d-c$. Thus, $(f\circ\Phi)'(y) = (a_1y_1,\ldots,a_{d-c}y_{d-c},0,\ldots,0) \in \KK^d$, for all $y\in V$, and $\Crit f\circ\Phi = \{y\in V:y_1=\cdots =y_{d-c}=0\}$, an analytic submanifold of $V$ of dimension $c$.

\subsection{Morse Lemma for functions on Banach spaces with degenerate critical points}
\label{subsec:Morse_lemma_functions_Banach_space_degenerate_critical_points}
It is important to carefully distinguish between the \emph{Morse--Bott Lemma} and the more general \emph{Morse Lemma for functions on Banach spaces with degenerate critical points} (also known as the \emph{Morse Lemma with parameters} or \emph{Splitting Lemma}): the latter makes no assumption on whether the critical set is a submanifold or, even if it is a submanifold, whether its tangent space at each critical point is equal to the kernel of the Hessian operator at that point. We begin with the

\begin{mainthm}[Morse Lemma for functions on Banach spaces with degenerate critical points]
\label{mainthm:Hormander_C-6-1_Banach}
Let $\sX$ and $\sY$ be Banach spaces over $\KK$, and $\sU\subset \sX$ and $\sV \subset \sY$ be open neighborhoods of the origin, and
\[
  f:\sX\times\sY \supset \sU\times \sV \ni (x,y) \mapsto f(x,y) \in \KK
\]
be a $C^{p+2}$ function ($p\geq 1$) such that $f(0,0) = 0$ and $D_1f(0,0)=0 $. If $D_1^2f(0,0) \in \sL(\sX, \sX^*)$ is invertible then, after possibly shrinking $\sU$ and $\sV$, there are an open neighborhood $\sU'$ of the origin and a $C^p$ diffeomorphism
\[
  \Phi: \sU'\times \sV \ni (z,y) \mapsto (x,y) \in \sU\times \sV
\]
with $\Phi(0,0)=(0,0)$ and
\begin{equation}
\label{eq:Derivative_Phi}
D\Phi(0,0) = \begin{pmatrix} \id_\sX & \star \\ 0 & \id_\sY \end{pmatrix} \in\sL(\sX\oplus\sY)
\end{equation}
such that
\begin{equation}
\label{eq:Hormander_C-6-1_Banach_dualspace}
f(\Phi(z,y)) = f(\Phi(0,y)) + \frac{1}{2}\langle z, Az\rangle_{\sX\times\sX^*}, \quad\text{for all } (z,y) \in \sU'\times \sV,
\end{equation}
where
\[
  A := D_1^2f(0,0) = D_1^2(f\circ\Phi)(0,0) \in \sL_\sym(\sX, \sX^*).
\]
If $f$ is analytic, then $\Phi$ is analytic.
\end{mainthm}

\begin{rmk}[Previous versions of the Morse Lemma for functions with degenerate critical points]
\label{rmk:Hormander_C-6-1_Banach_older_versions}
Theorem \ref{mainthm:Hormander_C-6-1_Banach} is a generalization of \cite[Lemma C.6.1]{Hormander_v3} from the case where $f$ is $C^\infty$ on $\sH = \sX\times\sY$ and $\sH=\RR^d$ with its standard inner product; and \cite[Lemma 1]{Gromoll_Meyer_1969top}, due to Gromoll and Meyer, from the case where $f$ is $C^\infty$ and $\sH$ is a real separable Hilbert space and the Hessian is Fredholm at the critical point; and \cite[Theorem 1]{Mawhin_Willem_1985}, due to Mawhin and Willem, from the case where $f$ is $C^2$ and $\sH$ is a real Hilbert space and the Hessian is Fredholm at the critical point. H\"ormander's proof is similar to that of Palais \cite[p. 307]{Palais_1963}, who allows $\sH$ to be a real Hilbert space but assumes that the critical point is non-degenerate. In \cite{Palais_1969}, Palais uses the Moser Path Method from \cite{Moser_1965}; see \cite{Weinstein_1969} for another early application of \cite{Moser_1965} in the  setting of Banach manifolds. Theorem \ref{mainthm:Hormander_C-6-1_Banach} is also known as the Morse Lemma with parameters (see \cite{Postnikov_Rudyak_morse_lemma}) or Splitting Lemma --- see Br{\"o}cker \cite[Lemma 14.12]{Brocker_differentiable_germs_catastrophes} or Poston and Stewart \cite[Theorems 4.5 and 6.1]{Poston_Stewart_catastrophe_theory_applications} --- and attributed to Thom.
\end{rmk}

\begin{rmk}[Previous versions of the Morse Lemma for functions on Banach spaces]
\label{rmk:Morse_lemmas_Banach_spaces}
Palais proved the Morse Lemma for smooth functions on Hilbert spaces in \cite[p. 307]{Palais_1963} and later extended and simplified his proof to give the Morse Lemma for smooth functions on Banach spaces in \cite[p. 968]{Palais_1969} (see also Guillemin and Sternberg \cite[Chapter 1, Appendix 1]{Guillemin_Sternberg_1977} for another exposition of Palais's proof in \cite{Palais_1969}).

The version of the Morse Lemma (for functions with degenerate critical points) in Theorem \ref{mainthm:Hormander_C-6-1_Banach}, is not the most general possible extension of H\"ormander's \cite[Lemma C.6.1]{Hormander_v3} from Euclidean space to a Banach space. Rather, as in Lang's exposition \cite[Section 7.5]{Lang_fundamentals_differential_geometry} of the proof of Palais's version of the Morse Lemma on Hilbert spaces \cite[p. 307]{Palais_1963}, our hypotheses in Theorem \ref{mainthm:Hormander_C-6-1_Banach} are strong enough that replacement of Euclidean space by a Banach space over $\KK$ involves no new complication.

A shorter proof of Palais's version of the Morse Lemma on Banach spaces \cite[p. 968]{Palais_1969} is given by Ang and Tuan \cite{Ang_Tuan_1973}; see also their article \cite{Tuan_Ang_1979}. Kuo \cite[Theorem, p. 364]{Kuo_HH_1974} and Tromba \cite{Tromba_1972} consider $C^{p+2}$ functions on Banach spaces with isolated critical points obeying a more general notion of non-degeneracy than that of Palais \cite[p. 968]{Palais_1969}, avoiding the implication in \cite{Palais_1969} that the Banach space is isomorphic to its dual space via the isomorphism provided by the Hessian operator. A related extension was proposed earlier by Uhlenbeck \cite{Uhlenbeck_1970}, inspired by a question of Smale \cite{Smale_1964}. Antoine \cite[Theorem 1]{Antoine_1979} considers $C^{p+2}$ functions on Banach spaces with isolated critical points and invertible Hessian operators.

Gromoll and Meyer \cite[Lemma 1]{Gromoll_Meyer_1969top} consider $C^\infty$ functions on Hilbert spaces with Fredholm Hessian operators while Mawhin and Willem \cite[Theorem 1]{Mawhin_Willem_1985} relax the regularity requirement of Gromoll and Meyer in \cite{Gromoll_Meyer_1969top} from $C^\infty$ to $C^2$.
\end{rmk}

\begin{rmk}[Regularity of the function $f$]
\label{rmk:Degenerate_Morse_function_regularity}
It is likely that one can adapt the arguments of Cambini \cite{Cambini_1973}, Kuiper \cite{Kuiper_1972}, and Mawhin and Willem \cite{Mawhin_Willem_1985} to reduce the $C^{p+2}$ (with $p\geq 1$) regularity requirement on $f$ in Theorem \ref{mainthm:Hormander_C-6-1_Banach} to $C^2$, as those authors allow for their versions of the Morse Lemma on Banach spaces, but the resulting proof would be lengthier and less elegant. See Remark \ref{rmk:Proof_Lojasiewicz-Simon_gradient_inequality_by_Morse-Bott_lemma} for further discussion.
\end{rmk}

\begin{rmk}[Morse Lemma for functions with degenerate critical points and Lyapunov--Schmidt reduction]
\label{rmk:Hormander_C-6-1_Banach_as_Lyapunov--Schmidt_reduction}
Theorem \ref{mainthm:Hormander_C-6-1_Banach} may be regarded as a more refined version of the technique of Lyapunov--Schmidt reduction (see Guo and Wu \cite[Section 5.1]{Guo_Wu_bifurcation_theory_functional_differential_equations}, Huang \cite[Proposition 2.4.1]{Huang_2006}, or Nirenberg \cite[Section 2.7.6]{Nirenberg_topics_nonlinear_functional_analysis}) when $\sY$ has finite dimension. Indeed, as we see from Lemma~\ref{lem:Invariance_Lojasiewicz_exponent_addition_quadratic_form_refined}, Theorem \ref{mainthm:Hormander_C-6-1_Banach} immediately reduces the proof of the {\L}ojasiewicz gradient inequality for analytic functions on Banach spaces to the well-known {\L}ojasiewicz gradient inequality for analytic functions on Euclidean space (see Feehan \cite{Feehan_lojasiewicz_inequality_all_dimensions} for a detailed survey and references).
\end{rmk}

As noted in Section \ref{subsec:Main_results_Morse-Bott_property_functions_Lojasiewicz_exponent_half} (see the discussion prior to Theorem \ref{mainthm:Analytic_function_Lojasiewicz_exponent_one-half_Morse-Bott_Banach_refined}), the hypothesis in Theorem \ref{mainthm:Hormander_C-6-1_Banach} that $D_1^2f(0,0) \in \sL(\sX,\sX^*)$ is an isomorphism is strong but is relaxed in the following generalization which immediately yields Theorem \ref{mainthm:Hormander_C-6-1_Banach} upon specializing to $\tilde{\sX} = \sX^*$. See Kuo \cite[Theorem, p. 364]{Kuo_HH_1974}, Tromba \cite{Tromba_1972}, \cite{Uhlenbeck_1970}, and Mawhin and Willem \cite[Theorem 1]{Mawhin_Willem_1985} for related refinements, though none provide the generality of Theorem \ref{mainthm:Hormander_C-6-1_Banach_refined}.

\begin{mainthm}[Generalized Morse Lemma for functions on Banach spaces with degenerate critical points]
\label{mainthm:Hormander_C-6-1_Banach_refined}
Let $\sX$, $\tilde{\sX}$, and $\sY$ be Banach spaces over $\KK$, and $\tilde{\sX} \subset\sX^*$ be a continuous embedding, and $\sU\subset \sX$ and $\sV \subset \sY$ be open neighborhoods of the origin, and
\[
  f:\sX\times\sY \supset \sU\times \sV \ni (x,y) \mapsto f(x,y) \in \KK
\]
be a $C^{p+2}$ function ($p\geq 1$) such that $f(0,0) = 0$ and $D_1f(0,0)=0$ and $D_1f(x,y) \in \tilde{\sX}$ for all $(x,y)\in \sU\times \sV$. If $D_1^2f(0,0) \in \sL(\sX, \tilde{\sX})$ is invertible then, after possibly shrinking $\sU$ and $\sV$, there are an open neighborhood of the origin $\sU'\subset \sX$ and a $C^p$ diffeomorphism,
\[
  \Phi: \sU'\times \sV \ni (z,y) \mapsto (x,y) = \Phi(z,y) \in \sU\times \sV
\]
with $\Phi(0,0)=(0,0)$ and
\begin{equation}
\label{eq:Derivative_Phi_refined}
D\Phi(0,0) = \begin{pmatrix} \id_\sX & \star \\ 0 & \id_\sY \end{pmatrix} \in\sL(\sX\oplus\sY)
\end{equation}
such that
\begin{equation}
\label{eq:Hormander_C-6-1_Banach_refined}
f(\Phi(z,y)) = f(\Phi(0,y)) + \frac{1}{2}\langle z, Az\rangle_{\sX\times\sX^*}, \quad\text{for all } (z,y) \in \sU'\times \sV,
\end{equation}
where
\[
  A := D_1^2f(0,0) = D_1^2(f\circ\Phi)(0,0) \in \sL_\sym(\sX, \tilde{\sX}).
\]
If $f$ is analytic, then $\Phi$ is analytic.
\end{mainthm}


\subsection{Morse and Morse--Bott Lemmas for functions on Banach spaces}
\label{subsec:Morse-Bott_lemma_functions_Banach_space_degenerate_critical_points}
Theorems \ref{mainthm:Hormander_C-6-1_Banach} and \ref{mainthm:Hormander_C-6-1_Banach_refined} easily yield versions of the Morse Lemma and, more broadly, the Morse--Bott Lemma in varying degrees of generality, namely Theorems \ref{thm:Morse_Lemma_Banach}, \ref{thm:Morse-Bott_Lemma_Banach}, \ref{thm:Morse-Bott_Lemma_Banach_refined}, and \ref{cor:Morse-Bott_Lemma_holomorphic}. We refer to Section \ref{subsec:Applications_Morse_lemma_functions_degenerate_critical_points} for their statements and short proofs.

\subsection{{\L}ojasiewicz--Simon gradient inequalities for smooth Morse--Bott functions on Banach spaces}
\label{subsec:Lojasiewicz-Simon inequalities_Morse-Bott_functions_Banach_spaces}
The Morse--Bott Lemma (Theorems \ref{thm:Morse-Bott_Lemma_Banach} and \ref{thm:Morse-Bott_Lemma_Banach_refined}) readily leads to {\L}ojasiewicz--Simon gradient inequalities with exponent one half for $C^{p+2}$ (with $p\geq 1$) Morse--Bott functions on Banach spaces, giving alternative proofs to those that do not rely on the Morse--Bott Lemma provided by the author in \cite[Theorem 3]{Feehan_lojasiewicz_inequality_all_dimensions} (when $f$ is $C^2$ and $\sX$ is finite-dimensional) and \cite[Theorem 3 and Corollaries 4 and 5]{Feehan_lojasiewicz_inequality_ground_state} (when $f$ is $C^2$ and $\sX$ is a Banach space) and by the author and Maridakis \cite[Theorem 4]{Feehan_Maridakis_Lojasiewicz-Simon_Banach} (when $f$ is $C^2$ and $\sX$ is a Banach space and $f''(0)$ is Fredholm with index zero). We begin with the following analogue of Theorem \ref{mainthm:Analytic_function_Lojasiewicz_exponent_one-half_Morse-Bott_Banach} and which is proved by appealing to the Morse--Bott Lemma provided by Theorem \ref{thm:Morse-Bott_Lemma_Banach}; the following Theorem \ref{mainthm:Analytic_function_Lojasiewicz_exponent_one-half_Morse-Bott_Banach} is similar to our \cite[Corollary 5]{Feehan_lojasiewicz_inequality_ground_state}, except that here we assume that $f$ is $C^{p+2}$ for some $p\geq 1$.

\begin{mainthm}[{\L}ojasiewicz gradient inequality for $C^{p+2}$ Morse--Bott functions on Banach spaces]
\label{mainthm:Lojasiewicz_gradient_inequality_Morse-Bott}
(Compare Feehan \cite[Corollary 5]{Feehan_lojasiewicz_inequality_ground_state}.)
Let $\sX$ be a Banach space over $\KK$, and $\sU$ be an open neighborhood of the origin, and $f:\sU \to \KK$ be a $C^{p+2}$ function ($p\geq 1$) such that $f(0) = 0$ and $f'(0)=0$. If $f$ is Morse--Bott at the origin in the sense of Definition \ref{defn:Morse-Bott_function} (\ref{item:Morse-Bott_point}) then, after possibly shrinking $\sU$, there is a constant $C \in (0, \infty)$ such that
\begin{equation}
\label{eq:Lojasiewicz-Simon_gradient_inequality_Morse-Bott}
\|f'(x)\|_{\sX^*} \geq C|f(x)|^{1/2}, \quad\text{for all } x \in \sU.
\end{equation}
\end{mainthm}

\begin{rmk}[On the definition of a Morse--Bott function on a Banach space]
\label{rmk:Definition_Morse-Bott_function_Banach_space}
In \cite[Definition 1.5]{Feehan_lojasiewicz_inequality_ground_state} and \cite[Definition 1.10]{Feehan_Maridakis_Lojasiewicz-Simon_Banach}, we said that a $C^2$ function $f:\sU\to\KK$ is Morse--Bott at a point $x_0\in\sU$ if $\Crit f$ is a $C^2$ (connected) submanifold and $T_{x_0}\Crit f = \Ker f''(x_0)$, but omitted the requirement\footnote{The property that $\Ker f''(x_0)\subset\sX$ has a closed complement $\sX_0 \subset \sX$ is automatically obeyed for the reasons explained in Remark \ref{rmk:Existence_closed_complement}.\label{fn:Existence_closed_complement}} that $\Ran f''(x_0)=\sX_0^*$. In our {\L}ojasiewicz gradient inequality \cite[Corollary 5]{Feehan_lojasiewicz_inequality_ground_state} for $C^2$ Morse--Bott functions analogous to Theorem \ref{mainthm:Lojasiewicz_gradient_inequality_Morse-Bott}, we required\footref{fn:Existence_closed_complement} that $\Ran f''(x_0)\subset \sX^*$ be a closed subspace (equivalent to $\Ran f''(x_0)=\sX_0^*$ when $\sX$ is reflexive by Lemma \ref{lem:Isomorphism_properties_symmetric_operator} \eqref{item:Kernel_complemented}. In the hypotheses for our \cite[Theorem 4]{Feehan_Maridakis_Lojasiewicz-Simon_Banach}, we imposed the stronger requirement that $f''(x_0) \in \sL(\sX,\sX^*)$ be Fredholm, so the additional condition that $\Ran f''(x_0)=\sX_0^*$ was obeyed automatically --- see Lemma \ref{lem:Range_symmetric_operator_kernel_has_closed_complement} \eqref{item:Range_Fredholm_operator_index_zero}.
\end{rmk}

The forthcoming Theorem \ref{mainthm:Lojasiewicz_gradient_inequality_Morse-Bott_refined} is an analogue of Theorem \ref{mainthm:Analytic_function_Lojasiewicz_exponent_one-half_Morse-Bott_Banach_refined}. While Theorem \ref{mainthm:Lojasiewicz_gradient_inequality_Morse-Bott_refined} is similar to our \cite[Corollary 4]{Feehan_lojasiewicz_inequality_ground_state}, we proved the latter result directly for functions $f$ that are only $C^2$, without appealing to the Morse--Bott Lemma (provided here by Theorem \ref{thm:Morse-Bott_Lemma_Banach_refined}); by contrast, we assume in Theorem \ref{mainthm:Lojasiewicz_gradient_inequality_Morse-Bott_refined} that $f$ is $C^{p+2}$ for some $p\geq 1$.

\begin{mainthm}[Generalized {\L}ojasiewicz gradient inequality for $C^{p+2}$ Morse--Bott functions on Banach spaces]
\label{mainthm:Lojasiewicz_gradient_inequality_Morse-Bott_refined}
(Compare Feehan \cite[Corollary 4]{Feehan_lojasiewicz_inequality_ground_state}.)
Let $\sX$ and $\tilde{\sX}$ be Banach spaces over $\KK$, and $\tilde{\sX}\subset\sX^*$ be a continuous embedding, $\sU$ be an open neighborhood of the origin, and $f:\sU \to \KK$ be a $C^{p+2}$ function ($p\geq 1$) such that $f(0) = 0$ and $f'(0)=0$ and $f'(x) \in \tilde{\sX}$ for all $x\in\sU$. If $f$ is Morse--Bott at the origin in the sense of Definition \ref{defn:Morse-Bott_function_refined} (\ref{item:Morse-Bott_point_refined}) then, after possibly shrinking $\sU$, there is a constant $C \in (0, \infty)$ such that
\begin{equation}
\label{eq:Lojasiewicz-Simon_gradient_inequality_Morse-Bott_refined}
\|f'(x)\|_{\tilde{\sX}} \geq C|f(x)|^{1/2}, \quad\text{for all } x \in \sU.
\end{equation}
\end{mainthm}

Simon \cite[Lemma 3.13.1]{Simon_1996} proved an analogue of Theorem \ref{mainthm:Lojasiewicz_gradient_inequality_Morse-Bott_refined} for a certain class of smooth functions on an open neighborhood of the origin in the Banach space of $C^3$ sections of a finite-rank, smooth Riemannian vector bundle over a closed, finite-dimensional, smooth Riemannian manifold; he describes the construction of the class of smooth functions in \cite[Sections 3.11, 3.12, and 3.13]{Simon_1996}. Simon's \cite[Lemma 3.13.1]{Simon_1996} can be recovered from our results with Maridakis \cite[Theorems 3 and 4]{Feehan_Maridakis_Lojasiewicz-Simon_Banach}, as we note in \cite[Remark 1.8]{Feehan_Maridakis_Lojasiewicz-Simon_Banach}. See also Haraux and Jendoubi \cite{Haraux_Jendoubi_2011} for related results. 

\begin{rmk}[On the definition of a generalized Morse--Bott function on a Banach space]
\label{rmk:Definition_Morse-Bott_function_Banach_space_refined}
In Feehan \cite[Definition 1.5]{Feehan_lojasiewicz_inequality_ground_state} and Feehan and Maridakis \cite[Definition 1.10]{Feehan_Maridakis_Lojasiewicz-Simon_Banach}, we said that a $C^2$ function $f:\sU\to\KK$ is Morse--Bott at a point $x_0\in\sU$ if $\Crit f$ is a $C^2$ (connected) submanifold and $T_{x_0}\Crit f = \Ker f''(x_0)$, but omitted the requirement\footref{fn:Existence_closed_complement} that $\Ran f''(x_0)=\tilde{\sX}$. In the hypotheses for our {\L}ojasiewicz gradient inequality \cite[Corollary 4]{Feehan_lojasiewicz_inequality_ground_state} for $C^2$ Morse--Bott functions analogous to Theorem \ref{mainthm:Lojasiewicz_gradient_inequality_Morse-Bott_refined}, we required that\footref{fn:Existence_closed_complement} $\Ran f''(x_0)\subset\tilde{\sX}$ be a closed subspace. In the hypotheses for our \cite[Theorem 4]{Feehan_Maridakis_Lojasiewicz-Simon_Banach}, we imposed the stronger requirement that $f''(x_0) \in \sL(\sX,\tilde{\sX})$ be Fredholm, so this additional condition was obeyed automatically --- see Lemma \ref{lem:Isomorphism_properties_Fredholm_operator} \eqref{item:Fredholm_index_zero}.
\end{rmk}

Theorem \ref{mainthm:Lojasiewicz_gradient_inequality_Morse-Bott_refined} immediately yields Theorem \ref{mainthm:Lojasiewicz_gradient_inequality_Morse-Bott} upon choosing $\tilde{\sX}=\sX^*$.

\begin{rmk}[On proofs of the {\L}ojasiewicz--Simon gradient inequalities via Morse--Bott Lemmas]
\label{rmk:Proof_Lojasiewicz-Simon_gradient_inequality_by_Morse-Bott_lemma}
As we shall see in Section \ref{subsec:Proof_Lojasiewicz_gradient_inequality_Morse-Bott}, the {\L}ojasiewicz--Simon gradient inequalities with exponent one half for $C^{p+2}$ Morse--Bott functions ($p\geq 1$) on Banach spaces are indeed easy consequences of the Morse--Bott Lemma (Theorems \ref{thm:Morse-Bott_Lemma_Banach} and \ref{thm:Morse-Bott_Lemma_Banach_refined}). However, the most useful version of such a {\L}ojasiewicz--Simon gradient inequality (namely \cite[Theorem 3]{Feehan_lojasiewicz_inequality_ground_state}), matching the generality of our forthcoming Theorem \ref{mainthm:Lojasiewicz-Simon_gradient_inequality2}, does not appear to be an obvious consequence of a Morse--Bott Lemma. Second, because our primary focus in this article is on the Morse--Bott property of analytic functions with {\L}ojasiewicz exponent one half, we have not striven to reduce the regularity requirements on $f$ from $C^{p+2}$ (with $p\geq 1$) to $C^2$. Mawhin and Willem \cite[Theorem 1]{Mawhin_Willem_1985} do establish a Morse--Bott Lemma for functions that are only $C^2$, but impose additional hypotheses on $f$ that we do not require in our Theorem \ref{thm:Morse-Bott_Lemma_Banach_refined}, namely that $\sX$ be a Hilbert space and (after identifying $\sX^*=\sX$) that $f''(0) \in \sL(\sX)$ be Fredholm; Mawhin and Willem generalize earlier Morse--Bott Lemmas for functions that are only $C^2$ due to Cambini \cite{Cambini_1973}, Hofer \cite{Hofer_1984, Hofer_1986}, and Kuiper \cite{Kuiper_1972}.
\end{rmk}

\subsection{Relationship between Morse--Bott and integrability conditions}
\label{subsec:Relationship_Adams-Simon_integrability_Morse-Bott_point_conditions}
We begin with the

\begin{defn}[Jacobi vectors and integrability]
\label{defn:Jacobi_vectors_integrability}
Let $\sX$ and $\tilde{\sX}$ be Banach spaces over $\KK$, and $\tilde{\sX}\subset\sX^*$ be a continuous embedding, and $\sU \subset \sX$ be an open subset, and $f:\sU\to\KK$ be a $C^2$ function such that $f'(x) \in \tilde{\sX}$ for all $x \in \sU$ and $f'(x_0)=0$ for some point $x_0\in\sU$. We call $\Ker f''(x_0) \subset \sX$ the subspace of \emph{Jacobi vectors} for $f$ at the critical point $x_0$. We say that $v\in\Ker f''(x_0)$ is an \emph{integrable Jacobi vector} if there exists an open neighborhood $J\subset\RR$ of the origin and a $C^1$ map $u:J \to \sX$ such that $u(0)=x_0$ and $u'(0)=v$ and $u(J)\subset\Crit f$. Finally, we say that $x_0$ is an \emph{integrable critical point} of $f$ if every Jacobi vector for $f$ at $x_0$ is integrable.
\end{defn}

Our Definition \ref{defn:Jacobi_vectors_integrability} is partly based on the integrability condition\footnote{I am grateful to Otis Chodosh for reminding me of this condition.} $(\star)$ described by  Adams and Simon in \cite[pp. 229--230]{Adams_Simon_1988} and inspired by an earlier definition due to Allard and Almgren \cite{Allard_Almgren_1981}: According to \cite{Adams_Simon_1988}, a critical point $x_0$ is \emph{integrable} if
\begin{multline}
\label{eq:Star}
\tag{$\star$}
\text{For all } v \in \Ker f''(x_0), \text{there exists } u \in C^0((0,1);\tilde{\sX}) \text{ such that } O(u) \subset \Crit f
\\
\text{and } \lim_{t\downarrow 0} u(t) = 0 \text{ (in $\tilde{\sX}$) } \text{ and } \lim_{t\downarrow 0} u(t)/t = v \text{ (in $\tilde{\sG}$)},
\end{multline}
where $O(u) := \{u(t): t \in (0,1)\}$ and $\tilde\sG$ is a Banach space with continuous embeddings $\tilde{\sX} \subset \tilde{\sG} \subset \sX^*$ as in the hypotheses of Theorem \ref{mainthm:Lojasiewicz-Simon_gradient_inequality2}. (Adams and Simon choose $\tilde{\sG}$ to be a certain Hilbert space but do not otherwise precisely specify the regularity properties of the path $u$ in their definition.)

Clearly, the property that a $C^2$ function be Morse--Bott at a critical point is closely related to integrability of that critical point. Indeed, by comparing Definitions \ref{defn:Morse-Bott_function_refined} and \ref{defn:Jacobi_vectors_integrability}, we obtain the 

\begin{lem}[Morse--Bott property implies integrability for critical points of $C^2$ functions on Banach spaces]
\label{lem:Morse-Bott_property_implies_integrability}  
Let $\sX$ and $\tilde{\sX}$ be Banach spaces over $\KK$, and $\tilde{\sX}\subset\sX^*$ be a continuous embedding, and $\sU \subset \sX$ be an open subset, and $f:\sU\to\KK$ be a $C^2$ function such that $f'(x) \in \tilde{\sX}$ for all $x \in \sU$. If $f$ is Morse--Bott at a point $x_0\in\Crit f$ in the sense of Definition \ref{defn:Morse-Bott_function_refined} \eqref{item:Morse-Bott_point_refined}, then $x_0$ is an integrable critical point in the sense of Definition \ref{defn:Jacobi_vectors_integrability}.
\end{lem}

The converse to Lemma \ref{lem:Morse-Bott_property_implies_integrability} is far more subtle since the integrability condition for a critical point is weaker than the Morse--Bott condition. The following result is proved by Simon for a specific class of analytic functions \cite{Adams_Simon_1988, Simon_1983, Simon_1985} on certain Banach spaces (given by $C^{2,\alpha}$ sections of a Riemannian vector bundle over a closed Riemannian manifold), but his method of proof\footnote{I am grateful to Leon Simon for explaining the relevant key ideas from \cite{Adams_Simon_1988, Simon_1983, Simon_1985}.} extends with little change to give the slightly more general

\begin{mainthm}[Integrability implies Morse--Bott property for critical points of analytic functions on Banach spaces]
\label{mainthm:Integrability_implies_Morse-Bott_property}
Let $\sX$ and $\tilde{\sX}$ be Banach spaces over $\KK$, and $\tilde{\sX}\subset\sX^*$ be a continuous embedding, $\sU$ be an open neighborhood of a point $x_0$, and $f:\sU \to \KK$ be an analytic function such that $f'(x) \in \tilde{\sX}$ for all $x\in\sU$. If $f'(x_0)=0$ and $f''(x_0) \in \sL(\sX,\tilde{\sX})$ is a Fredholm operator with index zero and the critical point $x_0$ is integrable in the sense of Definition \ref{defn:Jacobi_vectors_integrability}, then $f$ is Morse--Bott at $x_0$ in the sense of Definition \ref{defn:Morse-Bott_function_refined} \eqref{item:Morse-Bott_point_refined}.
\end{mainthm}

We refer to Appendix \ref{sec:Integrability_and_Morse-Bott_properties} for references to the literature where versions of Theorem \ref{mainthm:Integrability_implies_Morse-Bott_property} are stated, an outline of the proof based on those references, and a discussion of integrability and Morse--Bott conditions for the harmonic map energy and area functions, together with examples. We shall give a detailed proof of a more general version of Theorem \ref{mainthm:Integrability_implies_Morse-Bott_property} elsewhere \cite{Feehan_jacobi_vectors_analytic_potential_functions}.

\subsection{{\L}ojasiewicz--Simon gradient inequalities for analytic functions on Banach spaces}
\label{subsec:Lojasiewicz-Simon inequalities_analytic_functions_Banach_spaces}
Theorems \ref{mainthm:Hormander_C-6-1_Banach} and \ref{mainthm:Hormander_C-6-1_Banach_refined} can be used to give new proofs of the {\L}ojasiewicz--Simon gradient inequalities for analytic functions on Banach spaces proved earlier by the author and Maridakis in \cite{Feehan_Maridakis_Lojasiewicz-Simon_Banach}; moreover, our new proofs allow us to slightly weaken the hypotheses\footnote{By relaxing the hypotheses that the continuous embedding $\sX\subset\sX^*$ be definite, that one has a continuous embedding $\sX \subset \tilde{\sX}$, and that $\KK=\RR$.} that we assumed in \cite{Feehan_Maridakis_Lojasiewicz-Simon_Banach}.

\begin{mainthm}[{\L}ojasiewicz--Simon gradient inequality for analytic functions on Banach spaces]
\label{mainthm:Lojasiewicz-Simon_gradient_inequality}
(Compare Feehan and Maridakis \cite[Theorem 1]{Feehan_Maridakis_Lojasiewicz-Simon_Banach}.)
Let $\sX$ be a Banach space over $\KK$, and $\sU \subset \sX$ be an open neighborhood of the origin, and $f:\sU\to\KK$ be an analytic function such that $f(0) = 0$ and $f'(0) = 0$. If $f''(0) \in \sL(\sX,\sX^*)$ is a Fredholm operator with index zero then, after possibly shrinking $\sU$, there are constants $C \in (0, \infty)$ and $\theta \in [1/2,1)$ such that
\begin{equation}
\label{eq:Lojasiewicz-Simon_gradient_inequality_analytic_functional}
\|f'(x)\|_{\sX^*} \geq C|f(x)|^\theta, \quad\text{for all } x \in \sU.
\end{equation}
\end{mainthm}

The following generalization of Theorem \ref{mainthm:Lojasiewicz-Simon_gradient_inequality} relaxes the strong hypothesis that $f''(0) \in \sL(\sX,\sX^*)$ be Fredholm and immediately yields Theorem \ref{mainthm:Lojasiewicz-Simon_gradient_inequality} upon specializing to $\tilde{\sX} = \sX^*$.

\begin{mainthm}[Generalized {\L}ojasiewicz--Simon gradient inequality for analytic functions on Banach spaces]
\label{mainthm:Lojasiewicz-Simon_gradient_inequality_refined}
(Compare Feehan and Maridakis \cite[Theorem 2]{Feehan_Maridakis_Lojasiewicz-Simon_Banach}.)
Let $\sX$ and $\tilde{\sX}$ be Banach spaces over $\KK$ with a continuous embedding $\tilde{\sX} \subset \sX^*$, and $\sU \subset \sX$ be an open neighborhood of the origin, $f:\sU\to\KK$ be an analytic function with $f(0)=0$ and $f'(0)=0$ and $f'(x)\in\tilde{\sX}$ for all $x\in\sU$. If $f''(0) \in \sL(\sX,\tilde{\sX})$ is Fredholm with index zero then, after possibly shrinking $\sU$, there are constants $C \in (0,\infty)$ and $\theta \in [1/2, 1)$ such that
\begin{equation}
\label{eq:Lojasiewicz-Simon_gradient_inequality_analytic_functional_general}
\|f'(x)\|_{\tilde{\sX}} \geq C|f(x)|^\theta, \quad\text{for all } x \in \sU.
\end{equation}
\end{mainthm}

Theorem \ref{mainthm:Lojasiewicz-Simon_gradient_inequality_refined} is deduced from Theorem \ref{mainthm:Hormander_C-6-1_Banach_refined} in Section \ref{sec:Morse_lemma_degenerate_critical_points_Lojasiewicz-Simon_inequality}. While Theorem \ref{mainthm:Lojasiewicz-Simon_gradient_inequality_refined} is sufficient for many applications in geometric analysis, it also excludes some important examples (see \cite[Section 1.2]{Feehan_Maridakis_Lojasiewicz-Simon_Banach} for a discussion of such examples), including Simon's \cite[Theorem 3]{Simon_1983}, so we shall recall a useful generalization of Theorem \ref{mainthm:Lojasiewicz-Simon_gradient_inequality_refined}. We say that a bilinear form $b:\sX\times\sX \to \KK$ is \emph{definite} if $b(x,x) \neq 0$ for all $x \in \sX\less\{0\}$. We say that a continuous \emph{embedding} $\jmath:\sX\to\sX^*$ of a Banach space into its continuous dual space is \emph{definite} if the pullback of the canonical pairing, $\sX\times\sX \ni (x,y) \mapsto \langle x,\jmath(y)\rangle_{\sX\times\sX^*} \in \KK$, is a definite bilinear form. The following generalization of Theorem \ref{mainthm:Lojasiewicz-Simon_gradient_inequality_refined} does not appear to be a simple consequence of a Morse Lemma for degenerate critical points like Theorem \ref{mainthm:Hormander_C-6-1_Banach_refined}.

\begin{mainthm}[Generalized {\L}ojasiewicz--Simon gradient inequality for analytic functions on Banach spaces]
\label{mainthm:Lojasiewicz-Simon_gradient_inequality2}
(See Feehan and Maridakis \cite[Theorem 3]{Feehan_Maridakis_Lojasiewicz-Simon_Banach} for the case $\KK=\RR$.)
Let $\sX$ and $\tilde{\sX}$ be Banach spaces over $\KK$ with continuous embeddings $\sX \subset \tilde{\sX} \subset \sX^*$ and such that the embedding $\sX \subset \sX^*$ is definite. Let $\sU \subset \sX$ be an open subset and $f:\sU\to\KK$ be an analytic function such that $f(0)=0$ and $f'(0) = 0$. Let
\[
\sX\subset \sG \subset \tilde\sG \quad \text{and} \quad \tilde{\sX} \subset \tilde\sG \subset \sX^*
\]
be continuous embeddings of Banach spaces such that the compositions
\[
\sX\subset \sG\subset \tilde\sG \quad \text{and}\quad \sX\subset \tilde{\sX}\subset \tilde\sG
\]
induce the same embedding $\sX \subset \tilde\sG$. Let $\sM:\sU\to\tilde{\sX}$ be a gradient map for $f$ in the sense of Definition \ref{defn:Huang_2-1-1}. Suppose that for each $x \in \sU$, the bounded, linear operator
\[
\sM'(x): \sX \to \tilde \sX
\]
has an extension
\[
\sM_1(x): \sG \to \tilde\sG
\]
such that the map
\[
\sU \ni x \mapsto \sM_1(x) \in \sL(\sG, \tilde\sG) \quad\hbox{is continuous}.
\]
If $\sM'(0):\sX\to \tilde{\sX}$ and $\sM_1(0):\sG\to \tilde\sG$ are Fredholm operators with index zero then, after possibly shrinking $\sU$, there are constants $C \in (0,\infty)$ and $\theta \in [1/2, 1)$ such that
\begin{equation}
\label{eq:Lojasiewicz-Simon_gradient_inequality_analytic_functional_general2}
\|\sM(x)\|_{\tilde\sG} \geq C|f(x)|^\theta, \quad\text{for all } x\in \sU.
\end{equation}
\end{mainthm}

Suppose now that $\tilde\sG = \sH$, a Hilbert space, so that the embedding $\sG\subset \sH$  in Theorem \ref{mainthm:Lojasiewicz-Simon_gradient_inequality2}, factors through $\sG\subset \sH\simeq \sH^* $ and therefore
\[
f'(x)v = \langle v, \sM(x) \rangle_{\sX\times\sX^*} = (v, \sM(x))_\sH, \quad\text{for all } x \in \sU \text{ and } v \in \sX,
\]
using the continuous embeddings $\tilde{\sX} \subset \sH \subset \sX^*$. As we note in Remark \ref{rmk:Embedding_hypothesis_Huang_theorem_2-4-5}, the hypothesis in Theorem \ref{mainthm:Lojasiewicz-Simon_gradient_inequality2} that the embedding $\sX \subset \sX^*$ is definite is implied by the assumption that $\sX \subset \sH$ is a continuous embedding into a Hilbert space. Theorem \ref{mainthm:Lojasiewicz-Simon_gradient_inequality2} then yields
\begin{equation}
\label{eq:Lojasiewicz-Simon_gradient_inequality_analytic_functional_Hilbert_space}
\|\sM(x)\|_{\sH} \geq C|f(x)|^\theta, \quad\text{for all } x\in \sU,
\end{equation}
as desired.

\begin{rmk}[Comments on the embedding hypothesis in Theorem \ref{mainthm:Lojasiewicz-Simon_gradient_inequality2}]
\label{rmk:Embedding_hypothesis_Huang_theorem_2-4-5}
(See Feehan and Maridakis \cite[Remark 1.1]{Feehan_Maridakis_Lojasiewicz-Simon_Banach}.)
The hypothesis in Theorem
\ref{mainthm:Lojasiewicz-Simon_gradient_inequality2} on the
continuous embedding $\sX \subset \sX^*$ is easily achieved given a
continuous embedding $\eps$ of $\sX$ into a Hilbert space
$\sH$. Indeed, because $\langle y,\jmath(x)\rangle_{\sX\times\sX^*} =
(\eps(y),\eps(x))_\sH$ for all $x,y\in\sX$, then $\langle
x,\jmath(x)\rangle_{\sX\times\sX^*} = 0$ implies $x = 0$; see
\cite[Remark 3, page 136]{Brezis} or \cite[Lemma
D.1]{Feehan_Maridakis_Lojasiewicz-Simon_harmonic_maps_v5} for
details.
\end{rmk}

\subsection{Applications}
\label{subsec:Applications}
Due to the difficulty in computing the {\L}ojasiewicz exponent, one should not in general expect Theorems \ref{mainthm:Analytic_function_Lojasiewicz_exponent_one-half_Morse-Bott_Banach} or \ref{mainthm:Analytic_function_Lojasiewicz_exponent_one-half_Morse-Bott_Banach_refined} to provide a useful way to prove the Morse--Bott property of an analytic function with Hessian operator that is Fredholm of index zero. Nonetheless, they provide insight to applications in geometric analysis and we survey a few such applications here.

\subsubsection{Yamabe energy function, integrability conditions, {\L}ojasiewicz exponents, and  Morse--Bott properties}
\label{subsubsec:Integrability_conditions_Lojasiewicz_exponents_Morse-Bott_properties}
Carlotto, Chodosh, and Rubinstein \cite{Carlotto_Chodosh_Rubinstein_2015} study the existence of `slowly-converging' (volume-normalized) gradient flows for the Yamabe energy function on Riemannian metrics over a closed manifold of dimension three or more with the aid of results due to
\begin{inparaenum}[\itshape a\upshape)]
\item Adams and Simon \cite{Adams_Simon_1988} on the relationship between integrability and certain types of non-integrability and rates of convergence of geometric flows, and
\item Chill \cite{Chill_2003} on the {\L}ojasiewicz--Simon gradient inequality for functions on Banach spaces.
\end{inparaenum}
In particular, for a certain class of geometric flows, Adams and Simon show that the integrability condition \eqref{eq:Star} implies an exponential rate of convergence \cite[Theorem 1 (i)]{Adams_Simon_1988} and in a certain subcase where integrability fails \cite[Theorem 1 (ii)]{Adams_Simon_1988}, the flow converges according to a negative power law and thus is slowly converging in the terminology of \cite{Carlotto_Chodosh_Rubinstein_2015}. We refer the reader to Appendix \ref{sec:Convergence_rate} for an exposition of our general result \cite[Theorem 3]{Feehan_yang_mills_gradient_flow_v4} on the relationship between the rate of convergence for the gradient flow of a function obeying a {\L}ojasiewicz--Simon gradient inequality near a critical point and the value of the {\L}ojasiewicz exponent.

When $\sE$ is the Yamabe (or Einstein--Hilbert) energy function, Carlotto, Chodosh, and Rubinstein show that the Adams--Simon integrability condition \eqref{eq:Star} implies that the {\L}ojasiewicz exponent for $\sE$ at a critical point is equal to one half \cite[Proposition 13]{Carlotto_Chodosh_Rubinstein_2015}, for a suitable choice of Banach spaces, and that in turn indicates (by the main results of this article) that $\sE$ should be Morse--Bott at the critical point. More generally, when $\sE$ is an analytic function on a Banach space obeying hypotheses similar to those of Theorems \ref{mainthm:Analytic_function_Lojasiewicz_exponent_one-half_Morse-Bott_Banach}, \ref{mainthm:Analytic_function_Lojasiewicz_exponent_one-half_Morse-Bott_Banach_refined}, or perhaps even Theorem \ref{mainthm:Lojasiewicz-Simon_gradient_inequality2}, we would expect the Adams--Simon integrability condition \eqref{eq:Star} for a critical point to imply that $\sE$ is Morse--Bott at that point by generalizing the proof due to Kwon \cite{KwonThesis} of her Theorem \ref{thm:Kwon_4-1_harmonic_map}.

\subsubsection{Harmonic map energy function for maps from a Riemann surface into a closed Riemannian manifold}
\label{subsubsec:Harmonic_map_energy_function_maps_Riemann_surface_into_closed_Riemannian_manifold}
For background on harmonic maps, we refer to H{\'e}lein \cite{Helein_harmonic_maps}, Jost \cite{Jost_riemannian_geometry_geometric_analysis}, Simon \cite{Simon_1996}, and Struwe \cite{Struwe_variational_methods}. Let $(M,g)$ and $(N,h)$ be a pair of closed, smooth Riemannian manifolds. One defines the \emph{harmonic map energy function} by
\begin{equation}
\label{eq:Harmonic_map_energy_functional}
\sE_{g,h}(f)
:=
\frac{1}{2} \int_M |df|_{g,h}^2 \,d\vol_g,
\end{equation}
for smooth maps $f:M\to N$, where $df:TM \to TN$ is the differential map.

For the harmonic map energy function, a {\L}ojasiewicz gradient inequality with exponent one half,
\[
\|\sE'(f)\|_{L^p(S^2)} \geq Z|\sE(f) - \sE(f_\infty)|^{1/2},
\]
has been obtained by Kwon \cite[Theorem 4.2]{KwonThesis} for maps $f:S^2\to N$, where $N$ is a closed Riemannian manifold and $f$ is close to a harmonic map $f_\infty$ in the sense that
\[
\|f - f_\infty\|_{W^{2,p}(S^2)} < \sigma,
\]
where $p$ is restricted to the range $1 < p \leq 2$, and $f_\infty$ is assumed to be \emph{integrable} in the sense of \cite[Definitions 4.3 or 4.4 and Proposition 4.1]{KwonThesis}. Her proof of \cite[Proposition 4.1]{KwonThesis} quotes results of Simon \cite[pp. 270--272]{Simon_1985} and Adams and Simon \cite[Lemma 1]{Adams_Simon_1988}. The result \cite[Lemma 3.3]{Liu_Yang_2010} due to Liu and Yang is another example of a {\L}ojasiewicz gradient inequality with exponent one half for the harmonic map energy function, but restricted to the setting of maps $f:S^2\to N$, where $N$ is a K{\"a}hler manifold of complex dimension $n \geq 1$ and nonnegative bisectional curvature, and the energy $\sE(f)$ is sufficiently small. The result of Liu and Yang generalizes that of Topping \cite[Lemma 1]{Topping_1997}, who assumes that $N = S^2$.

Milnor observes \cite[Footnote to Problem 3-c]{Milnor_dynamics_one_complex_variable} that the space of holomorphic maps of degree $d$ from $\CC\PP^1$ to $\CC\PP^1$ is a non-compact complex manifold of dimension $2d+1$. However, he notes \cite[Footnote to Problem 3-c]{Milnor_dynamics_one_complex_variable} that there is an example (due to J. Harris) of a Riemann surface $\Sigma$ of genus $5$ such that the space of holomorphic maps from $\Sigma$ into $\CC\PP^1$ has singularities. In general, the space of harmonic maps of degree $d$ from $S^2$ into $S^{2n}$ (with $n\geq 1$) will not be a smooth manifold \cite{Fernandez_2012}. We survey some positive results for spaces of harmonic maps in Appendix \ref{subsec:Integrability_and_Morse-Bott_harmonic_map_energy_functional}. The version of the `Bumpy Metric Theorem' proved by Moore as \cite[Theorem 5.1.1]{Moore_introduction_global_analysis} states that if $M$ is a compact manifold of dimension at least three and the Riemannian metric is generic, then all minimal two-spheres in $M$ are as nondegenerate as allowed by the group of conformal automorphisms of $S^2$,  that is, they lie on nondegenerate critical submanifolds of $\Map(S^2,M)$, each such submanifold being an orbit for the symmetry group $\PSL(2,\CC)$.

\subsubsection{Moduli spaces of flat connections and representation varieties}
\label{subsubsec:Moduli_spaces_flat_connections_representation_varieties}
When a base manifold $X$ is compact and K\"ahler\footnote{I am grateful to Graeme Wilkin for drawing my attention to the results of Simpson and Goldman--Millson described here.}, and $G$ is a complex reductive Lie group, Simpson proved that the singularities in the moduli space of flat connections are at worst quadratic at any reductive representation of the fundamental group \cite[Corollary 2.4]{Simpson_1992}; when $G$ is a compact Lie group, this result is due to Goldman and Millson \cite[Theorem 1]{Goldman_Millson_1988}. When the base manifold $X$ is not compact or K\"ahler then the singularities in the moduli space of flat connections may be worse. Indeed, this can occur for representation varieties for fundamental groups of certain closed, smooth three-dimensional manifolds. Goldman and Millson \cite[Section 9.1]{Goldman_Millson_1988} choose $X = H/\Gamma$, where $H$ is the three-dimensional real Heisenberg group and $\Gamma \subset H$ is a lattice, so that $X$ is the total space of an oriented circle bundle over a two-torus with non-zero Euler class. If $G$ is an algebraic Lie group that is not two-step nilpotent and $\rho:\Gamma \to G$ is the trivial representation, then the representation variety $\sR(\Gamma,G)$ is not quadratic at $\rho$: the analytic germ of $\sR(\Gamma,G)$ is isomorphic to a cubic cone. A compact Lie group with a simple Lie algebra, such as $\SU(n)$ for $n \geq 2$, is not nilpotent and so we may choose $G = \SU(n)$ with $n\geq 2$ in the Goldman--Millson counterexample. Recall \cite[Proposition 2.2.3]{DK} that the gauge-equivalence classes of flat connections on a principal $G$-bundle over a connected manifold $X$ are in one-to-one correspondence with the conjugacy classes of representations $\pi_1(X)\to G$.

In our articles \cite{Feehan_lojasiewicz_inequality_ground_state, Feehan_nonlinear_uhlenbeck_estimate}, we explore the Morse--Bott properties of the Yang--Mills energy function,
\begin{equation}
\label{eq:Yang-Mills_energy_function}
\sE(A)  := \frac{1}{2}\int_X |F_A|^2\,d\vol_g,
\end{equation}
on the affine space of $W^{1,q}$ connections $A$ on a principal $G$-bundle (for a compact Lie group $G$) over a closed Riemannian manifold $(X,g)$ of dimension $d\geq 2$ (and $q\in[2,\infty)$ obeying $q>d/2$). In particular, we explore the Morse--Bott properties of $\sE$ in \eqref{eq:Yang-Mills_energy_function} near the subspace of flat connections.

\subsubsection{$F$-function on the space of hypersurfaces in Euclidean space}
\label{subsubsec:F-function_hypersurfaces_Euclidean_space}
Colding and Minicozzi \cite{Colding_Minicozzi_2014sdg, Colding_Minicozzi_2015} have given proofs of {\L}ojasiewicz--Simon gradient and distance inequalities \cite[Equations (5.9) and (5.10)]{Colding_Minicozzi_Pedersen_2015bams} that do not involve Lyapunov--Schmidt reduction to a finite-dimensional gradient inequality, as in the original paradigm due to Simon \cite{Simon_1983}. Their gradient inequality applies to the $F$ function \cite[Section 2.4]{Colding_Minicozzi_Pedersen_2015bams} on the space of hypersurfaces $\Sigma \subset \RR^{d+1}$ and is analogous to \eqref{eq:Lojasiewicz_gradient_inequality} with $\theta=2/3$. Their cited articles contain detailed technical statements of their inequalities while their article with Pedersen \cite{Colding_Minicozzi_Pedersen_2015bams} contains a less technical summary of some of their main results.

\subsection{Structure of this article}
\label{subsec:Outline}
In Section \ref{sec:Generalized_Morse_Lemmas_functions_Banach_spaces}, we prove the Morse Lemma for functions on Banach spaces with degenerate critical points (Theorem \ref{mainthm:Hormander_C-6-1_Banach_refined}) and then deduce some corollaries, including the Morse Lemma for functions on Banach spaces with non-degenerate critical points (Theorem \ref{thm:Morse_Lemma_Banach}), and the Morse--Bott Lemma for functions on Banach spaces (Theorems \ref{thm:Morse-Bott_Lemma_Banach} and \ref{thm:Morse-Bott_Lemma_Banach_refined}). In Section \ref{sec:Morse_lemma_degenerate_critical_points_Lojasiewicz-Simon_inequality},  we apply Theorem \ref{thm:Morse-Bott_Lemma_Banach_refined} to prove the {\L}ojasiewicz gradient inequality for $C^{p+2}$ Morse--Bott functions on Banach spaces (Theorem \ref{mainthm:Lojasiewicz_gradient_inequality_Morse-Bott_refined}) and apply Theorem \ref{mainthm:Hormander_C-6-1_Banach_refined} to prove the {\L}ojasiewicz gradient inequality for analytic functions on Banach spaces (Theorem \ref{mainthm:Lojasiewicz-Simon_gradient_inequality_refined}). Finally, in Section \ref{sec:Characterization_optimal_exponent_Morse-Bott_condition} we complete the proof of the Morse--Bott property of an analytic function with {\L}ojasiewicz exponent one half (Theorem \ref{mainthm:Analytic_function_Lojasiewicz_exponent_one-half_Morse-Bott_Banach}). In Appendix \ref{sec:Convergence_rate}, we discuss our general result \cite[Theorem 3]{Feehan_yang_mills_gradient_flow_v4} on the rate of convergence of a gradient flow for a function obeying a {\L}ojasiewicz gradient inequality. In Appendix \ref{sec:Blow-up_map_exceptional_divisor_polar_coordinates}, we describe the relationship between Morse--Bott functions and quadratic simple normal crossing functions. In Appendix~\ref{sec:Integrability_and_Morse-Bott_properties}, we outline the proof of Theorem \ref{mainthm:Integrability_implies_Morse-Bott_property} and survey results on integrability conditions and the Morse--Bott property for critical points of the harmonic map energy and area functions.

\subsection{Acknowledgments}
\label{subsec:Acknowledgments}
I am indebted to Michael Greenblatt and Andr{\'a}s N{\'e}methi for independently pointing out to me that, for functions on Euclidean space, the Morse Lemma for functions with degenerate critical points (also known as the Morse Lemma with parameters or Splitting Lemma) should be the key ingredient needed to prove the main result of this article in the finite-dimensional case (Corollary~\ref{maincor:Analytic_function_Lojasiewicz_exponent_one-half_Morse-Bott}). I am extremely grateful to Brian White for explaining results of his \cite{White_1987, White_1991, White_2017} and others on minimal surfaces and integrability of Jacobi fields and to Leon Simon for explaining his results and results with Adams on integrability of Jacobi fields in \cite{Adams_Simon_1988, Simon_1983, Simon_1985}. I also thank Carles Bivi{\`a}-Ausina, Otis Chodosh, Tristan Collins, Santiago Encinas, Luis Fernandez, Antonella Grassi, David Hurtubise, Johan de Jong, Daniel Ketover, Qingyue Liu, Doug Moore, Yanir Rubinstein, Siddhartha Sahi, Ovidiu Savin, Peter Topping, Graeme Wilkin, and Jarek W{\l}odarczyk for helpful communications, discussions, or questions during the preparation of this article. I am grateful to the National Science Foundation for their support and the Dublin Institute for Advanced Studies and Yi-Jen Lee and the Institute of Mathematical Sciences at the Chinese University of Hong Kong for their hospitality and support. Lastly, I am most grateful to the anonymous referee for numerous comments and suggestions that helped improve this article.

\section{Generalized Morse Lemmas for functions on Banach spaces}
\label{sec:Generalized_Morse_Lemmas_functions_Banach_spaces}
In Sections \ref{subsec:Linear_functional_analysis_preliminaries} and \ref{subsec:Nonlinear_functional_analysis_preliminaries}, respectively, we collect some basic observations from linear and nonlinear functional analysis that we require in this article. In Section \ref{subsec:Morse_lemma_functions_degenerate_critical_points}, we prove the Morse Lemma for functions on Banach spaces with degenerate critical points (Theorem \ref{mainthm:Hormander_C-6-1_Banach_refined}) and in Section \ref{subsec:Applications_Morse_lemma_functions_degenerate_critical_points} we deduce some corollaries, including the Morse Lemma for functions on Banach spaces with non-degenerate critical points (Theorem \ref{thm:Morse_Lemma_Banach}), and the Morse--Bott Lemma for functions on Banach spaces (Theorems \ref{thm:Morse-Bott_Lemma_Banach} and \ref{thm:Morse-Bott_Lemma_Banach_refined}).

\subsection{Preliminaries on linear functional analysis}
\label{subsec:Linear_functional_analysis_preliminaries}
In this subsection, we gather a few elementary observations from linear functional analysis. We begin with the following useful

\begin{lem}[Dual space of a direct sum of Banach spaces]
\label{lem:Dual_direct_sum_Banach_spaces_is_direct_sum_dual_spaces}
(See \cite{MathStackExchange_Dual_direct_sum_Banach_spaces_is_direct_sum_dual_spaces}.)
If $\sX, \sY$ are Banach spaces over $\KK$, then $(\sX\oplus\sY)^* = \sX^*\oplus\sY^*$.
\end{lem}

\begin{proof}
Let $\sZ := \sX\oplus\sY$, with product norm $\|(x,y)\|_{\sX\oplus\sY} := \|x\|_\sX + \|y\|_\sY$, continuous projection operators $\pi_\sX:\sZ\to\sX$ and $\pi_\sY:\sZ\to\sY$, continuous injection operators $\iota_\sX:\sX\to\sZ$ and $\iota_\sY:\sY\to\sZ$, and define $T:\sZ^* \to \sX^*\oplus\sY^*$ by $Tz^* := (z^*\iota_\sX, z^*\iota_\sY)$. We observe that $T$ is bounded because
\[
\|Tz^*\|_{\sX^*\oplus\sY^*} = \|z^*\iota_\sX\|_{\sX^*} + \|z^*\iota_\sY\|_{\sY^*} \leq 2\|z^*\|_{\sZ^*},
\]
noting that
\[
\|z^*\iota_\sX\|_{\sX^*} = \sup_{x\in\sX\less\{0\}} \frac{|z^*(\iota_\sX(x))|}{\|x\|_\sX} \leq \sup_{z\in\sZ\less\{0\}} \frac{|z^*(z)|}{\|z\|_\sZ} = \|z^*\|_{\sZ^*},
\]
and similarly $\|z^*\pi_\sY\|_{\sX^*} \leq \|z^*\|_{\sZ^*}$. The operator $T$ is injective since if $Tz^* = 0$, then\break\hfill $(z^*(\iota_\sX(x)),z^*(\iota_\sY(y))) = 0$ for all $(x,y) \in \sX\oplus\sY$ and so $z^*(z) = 0$ for all $z=x+y\in\sZ$ and hence $z^* = 0$. The operator $T$ is surjective since if $x^*\in\sX^*$ and $y^*\in\sY^*$ and we define $z^*(z) := x^*(x) + y^*(y)$ for $z=x+y\in\sZ$, then $z^* \in \sZ^*$. The operator $T^{-1}$ is bounded by the Open Mapping Theorem or by observing that $T^{-1}(x^*,y^*) = x^*\pi_\sX + y^*\pi_\sY \in \sZ^*$. Therefore, $T$ is an isomorphism of Banach spaces.
\end{proof}

In the proof of Lemma \ref{lem:Dual_direct_sum_Banach_spaces_is_direct_sum_dual_spaces}, we note that the adjoint map $\iota_\sX*:\sZ^*\to\sX^*$ is continuous and $(\iota_\sX^*z^*)(x) = z^*(\iota_\sX(x))$ for all $x\in\sX$, so if $x^*\in\sX^*$, then $(\iota_\sX^*x^*)(x) = x^*(\iota_\sX(x)) = x^*(x)$ for all $x\in\sX$. Thus, $\pi_{\sX^*} = \iota_\sX^*:\sZ^* \to \sX^*$ and $\pi_{\sY^*} = \iota_\sY^*:\sZ^* \to \sY^*$ are the induced projection operators. The following lemma helps motivate Definition \ref{defn:Morse-Bott_function} \eqref{item:Morse-Bott_point} but is not used elsewhere in this article.

\begin{lem}[Range of a symmetric operator whose kernel has closed complement]
\label{lem:Range_symmetric_operator_kernel_has_closed_complement}
Let $\sX$ be a Banach space over $\KK$. If $A \in \sL(\sX,\sX^*)$ is symmetric and $\Ker A$ has a closed complement $\sX_0$ in $\sX$, then the following hold:
\begin{enumerate}
  \item
  \label{item:Range_symmetric_operator_kernel_has_closed_complement}
  $\Ran A \subset \sX_0^*$;
  \item
  \label{item:Range_Fredholm_operator_index_zero}
  If $A$ is Fredholm with index zero, then $\Ran A = \sX_0^*$.
\end{enumerate}
\end{lem}

\begin{proof}
By hypothesis on $\sK:=\Ker A$, we have $\sX=\sX_0\oplus\sK$ and thus $\sX^* = \sX_0^*\oplus\sK^*$ by Lemma \ref{lem:Dual_direct_sum_Banach_spaces_is_direct_sum_dual_spaces}. Suppose $\alpha \in \sX^*$ belongs to $\sK^* \cap \Ran A$, so $\alpha = Ax$ for some $x \in \sX$. If $\xi \in \sK$, then $\langle\xi,\alpha\rangle_{\sX\times\sX^*} = \langle\xi,Ax\rangle_{\sX\times\sX^*} = \langle x,A\xi\rangle_{\sX\times\sX^*}$, since $A$ is symmetric, and $\langle x,A\xi\rangle_{\sX\times\sX^*}=0$ as $\xi\in\Ker A$. Because $\xi\in\sK$ was arbitrary, we see that $\alpha = 0$ on $\sK$ and thus $\alpha=0\in\sK^*$. Hence, $\Ran A \subset \sX_0^*$, as claimed in Item \eqref{item:Range_symmetric_operator_kernel_has_closed_complement}.

If $A$ is Fredholm, then $\sK$ is finite-dimensional and thus has a closed complement $\sX_0$ by \cite[Lemma 4.21 (a)]{Rudin}. Item \eqref{item:Range_symmetric_operator_kernel_has_closed_complement} implies that $\Ran A \subset \sX_0^*$.
Because $A$ has index zero, then $\dim\Ker A = \dim (\sX^*/\Ran A)$ and because $\sX^*/\Ran A$ is finite-dimensional, $\Ran A$ has a closed complement, say $\sM$, with $\sX^*=\Ran A\oplus\sM$ by \cite[Lemma 4.21 (b)]{Rudin}. But $\dim\sK^*=\dim\sK=\dim\sM$ and $\sX^* = (\sX_0\oplus\sK)^* = \sX_0^*\oplus\sK^*$ by Lemma \ref{lem:Dual_direct_sum_Banach_spaces_is_direct_sum_dual_spaces}, so $\Ran A=\sX_0^*$, as claimed in Item \eqref{item:Range_Fredholm_operator_index_zero}.
\end{proof}

We have the following generalization of \cite[Lemma D.3]{Feehan_Maridakis_Lojasiewicz-Simon_harmonic_maps_v5}; note that Lemma \ref{lem:Isomorphism_properties_symmetric_operator} \eqref{item:Kernel_complemented} does not directly generalize Lemma \ref{lem:Range_symmetric_operator_kernel_has_closed_complement} \eqref{item:Range_Fredholm_operator_index_zero}, since $\sX$ is assumed to be reflexive in Lemma \ref{lem:Isomorphism_properties_symmetric_operator} and while it also helps motivate Definition \ref{defn:Morse-Bott_function} \eqref{item:Morse-Bott_point}, it is not used elsewhere in this article.

\begin{lem}[Isomorphism properties of a symmetric operator]
\label{lem:Isomorphism_properties_symmetric_operator}
Let $\sX$ be a reflexive Banach space over $\KK$. If $A \in \sL(\sX,\sX^*)$ is symmetric with closed range, then the following hold:
\begin{enumerate}
  \item
  \label{item:Kernel_trivial}
  If $\Ker A=\{0\}$, then $\Ran A = \sX^*$;
  \item
  \label{item:Kernel_complemented}
  If $\Ker A$ has a closed complement $\sX_0\subset\sX$, then $\Ran A \cong \sX_0^*$.
\end{enumerate}
\end{lem}

\begin{proof}
If $M \subset \sX^*$ is a subspace, we recall from \cite[Section 4.6]{Rudin} that its annihilator is
\[
M^\perp := \{ \phi \in \sX^{**}:  \langle \alpha, \phi\rangle = 0, \text{ for all } \alpha\in M\},
\]
where $\langle\cdot,\cdot\rangle:\sX^*\times\sX^{**}\to\KK$ denotes the canonical pairing, and that by \cite[Theorem 4.12]{Rudin}
\begin{equation}
\label{eq:Rudin_theorem_4-12}
(\Ran A)^\perp = \Ker A^*,
\end{equation}
where $A^*:\sX^{**} \to \sX^*$ is the adjoint operator defined by
\[
\langle x,A^*\phi\rangle := \langle Ax,\phi\rangle,
\quad\text{for all } x\in \sX,\ \phi \in \sX^{**}.
\]
If $J:\sX \to \sX^{**}$ is the canonical map defined by $J(x)\alpha = \alpha(x)$ for all $x\in\sX$ and $\alpha\in\sX^*$, then $J$ is an isomorphism by hypothesis that $\sX$ is reflexive and thus
\[
\langle y, A^*J(x) \rangle
=
\langle Ay, J(x) \rangle
=
\langle x, A y \rangle,
\quad\text{for all } x, y\in \sX,
\]
that is,
\begin{equation}
\label{eq:AdjointA_J_equals_A}
\langle y, A^*J(x) \rangle
=
\langle x, A y \rangle,
\quad\text{for all } x, y\in \sX,
\end{equation}
where $\langle\cdot,\cdot\rangle:\sX\times\sX^*\to\KK$ also denotes the canonical pairing. Hence,
\begin{align*}
\Ker A^* &= \{\phi \in \sX^{**}:  \langle y, A^*\phi \rangle = 0 , \text{ for all } y\in \sX\}
\\
&= J\left(\{x \in \sX:  \langle y, A^*J(x) \rangle = 0 , \text{ for all } y\in \sX\}\right)
\quad \text{(by reflexivity of $\sX$)}
\\
&= J\left(\{x \in \sX:  \langle x, A y \rangle = 0 , \text{ for all } y\in \sX\}\right)
\quad \text{(by \eqref{eq:AdjointA_J_equals_A})}
\\
&= J\left(\{x \in \sX:  \langle y, A x \rangle = 0 , \text{ for all } y\in \sX\}\right) \quad \text{(by symmetry of $A$)},
\end{align*}
that is,
\begin{equation}
\label{eq:KernelA_isomorphic_KernelA}
\Ker A^* = J\left(\Ker A\right) \cong \Ker A.
\end{equation}
Consider Item \eqref{item:Kernel_trivial}. Because $\Ker A=\{0\}$ by assumption, then \eqref{eq:Rudin_theorem_4-12} and \eqref{eq:KernelA_isomorphic_KernelA} imply that
\begin{equation}
\label{eq:Annihilator_RanA_zero}
(\Ran A)^\perp = \{0\}.
\end{equation}
If $\overline{\Ran A}$ denotes the (norm) closure of $\Ran A$ in $\sX^*$, then
\begin{align*}
\Ran A &= \overline{\Ran A} \quad\text{(as $\Ran A$ closed by hypothesis)}
\\
&= {}^\perp\left((\Ran A)^\perp\right) \quad\text{(by \cite[Theorem 4.7(a)]{Rudin})}
\\
&= {}^\perp\{0\}  \quad\text{(by \eqref{eq:Annihilator_RanA_zero})}
\\
&= \sX^*,
\end{align*}
where if $N \subset X^{**}$ is a subspace, we recall from \cite[Section 4.6]{Rudin} that its annihilator is
\[
{}^\perp N := \{ \alpha \in \sX^*:  \langle \alpha, \phi\rangle = 0, \text{ for all } \phi\in N\}.
\]
This establishes Item \eqref{item:Kernel_trivial}.

Consider Item \eqref{item:Kernel_complemented}. By modifying the argument yielding Item \eqref{item:Kernel_trivial}, we now obtain
\begin{align*}
\Ran A
&= {}^\perp\left((\Ran A)^\perp\right) \quad\text{(by \cite[Theorem 4.7(a)]{Rudin} and closedness of $\Ran A$)}
\\
&= {}^\perp\left(\Ker A^*\right) \quad\text{(by \eqref{eq:Rudin_theorem_4-12})} 
\\    
&= {}^\perp\left(J\left(\Ker A\right)\right) \quad\text{(by \eqref{eq:KernelA_isomorphic_KernelA})}.
\end{align*}
But
\begin{align*}
{}^\perp\left(J\left(\Ker A\right)\right)
&=
\left\{\alpha\in\sX^*: \langle\alpha, Jx\rangle = 0,  \text{ for all } x \in \Ker A\right\}
\\
&=
\left\{\alpha\in\sX^*: \langle x, \alpha\rangle = 0,  \text{ for all } x \in \Ker A\right\}
\\
&= \left(\Ker A\right)^\perp,
\end{align*}
where if $L \subset \sX$ is a subspace, we recall from \cite[Section 4.6]{Rudin} that its annihilator is
\[
L^\perp := \{ \alpha \in \sX^*:  \langle x, \alpha\rangle = 0, \text{ for all } x\in L\},
\]
where $\langle\cdot,\cdot\rangle:\sX\times\sX^*\to\KK$ denotes the canonical pairing. Therefore, by combining the preceding identities, we obtain
\[
\Ran A = \left(\Ker A\right)^\perp.
\]
Since $\Ker A \subset \sX$ is a closed subspace, the quotient space $\sX/\Ker A$ is a Banach space (by \cite[Proposition 11.8]{Brezis}). By hypothesis, $\Ker A$ has a closed complement $\sX_0 \subset \sX$, so $\sX = \sX_0\oplus\Ker A$ and $\sX/\Ker A \cong \sX_0$. Hence, by \cite[Proposition 11.9]{Brezis} there is an isomorphism of Banach spaces,
\[
\sX_0^* \cong \left(\Ker A\right)^\perp.
\]
Consequently, we find that
\[
\Ran A \cong \sX_0^*,
\]
as claimed. This establishes Item \eqref{item:Kernel_complemented} and completes the proof of Lemma \ref{lem:Isomorphism_properties_symmetric_operator}.
\end{proof}

We have the following generalization of Lemma \ref{lem:Range_symmetric_operator_kernel_has_closed_complement} which helps motivate Definition \ref{defn:Morse-Bott_function_refined} \eqref{item:Morse-Bott_point_refined} but is not used elsewhere in this article.

\begin{lem}[Isomorphism properties of a Fredholm operator]
\label{lem:Isomorphism_properties_Fredholm_operator}
Let $\sX$ and $\tilde{\sX}$ be Banach spaces over $\KK$. If $T \in \sL(\sX,\tilde{\sX})$ is Fredholm, then $\Ker T$ has a closed complement $\sX_0$ in $\sX$ such that the following hold:
\begin{enumerate}
\item
\label{item:Fredholm}
$\Ran T \cong \sX_0$ and $\tilde{\sX} \cong \sX_0\oplus\Ker T^*$;

\item
\label{item:Fredholm_index_zero}
If $\Ind T = 0$, then $\tilde{\sX} \cong \sX_0\oplus\Ker T$.
\end{enumerate}
\end{lem}

\begin{proof}
Consider Item \eqref{item:Fredholm}. Since $T$ is Fredholm, then $\Ker T$ is finite-dimensional and thus has a closed complement $\sX_0 \subset \sX$ such that $\sX=\sX_0\oplus\Ker T$ by \cite[Lemma 4.21 (a)]{Rudin}. Similarly, because $T$ is Fredholm, $\Ran T$ is a closed subspace of $\tilde{\sX}$ and $\Coker T = \tilde{\sX}/\Ran T$ is finite-dimensional, so $\Ran T$ has a closed complement $\tilde{\sK} \subset \tilde{\sX}$ such that $\tilde{\sX}=\Ran T\oplus\tilde{\sK}$ by \cite[Lemma 4.21 (b)]{Rudin}, and $\tilde{\sX}/\Ran T = \tilde{\sK}$. Since $\Ran T$ is a Banach space and $T:\sX_0 \to \Ran T$ bijective and bounded, then $T$ is an isomorphism from $\sX_0$ onto $\Ran T$ by the Open Mapping Theorem. By \cite[Proposition 11.9]{Brezis}, we have $\tilde{\sK}^* = (\Ran T)^\perp$, and $(\Ran T)^\perp = \Ker T^*$ by \cite[Theorem 4.12]{Rudin}, and $\tilde{\sK} \cong \tilde{\sK}^*$ by finite-dimensionality. Hence, $\tilde{\sX} \cong \sX_0 \oplus \Ker T^*$, as claimed.

Consider Item \eqref{item:Fredholm_index_zero}. If $\Ind T = 0$, then $\dim\Ker T^* = \dim\Ker T$ and $\Ker T \cong \Ker T^*$ by finite-dimensionality and $\tilde{\sX}\cong \sX_0\oplus\Ker T$.
\end{proof}

\subsection{Preliminaries on nonlinear functional analysis}
\label{subsec:Nonlinear_functional_analysis_preliminaries}
In this subsection, we gather a few observations from nonlinear functional analysis.

\subsubsection{Differentiable and analytic maps on Banach spaces}
\label{subsubsec:Huang_2-1A}
We refer to Huang \cite[Section 2.1A]{Huang_2006}; see also Berger \cite[Section 2.3]{Berger_1977}. Let $\sX, \sY$ be Banach spaces over $\KK$, let $\sU\subset\sX$ be an open subset, and $\sF:\sU \to \sY$ be a map. Recall that $\sF$ is  \emph{Fr\'echet differentiable} at a point $x \in \sU$ with a  derivative, $\sF'(x) \in \sL(\sX,\sY)$, if
\[
\lim_{y\to 0} \frac{1}{\|y\|_\sX}\|\sF(x + y) - \sF(x) - \sF'(x)y\|_\sY = 0.
\]
Recall from Berger \cite[Definition 2.3.1]{Berger_1977}, Deimling \cite[Definition 15.1]{Deimling_1985}, or Zeidler \cite[Definition 8.8]{Zeidler_nfaa_v1} that $\sF$ is \emph{analytic} at $x \in \sU$ if there exists a constant $r > 0$ and a sequence of continuous symmetric linear maps $L_n:\otimes^n\sX \to \sY$ such that $\sum_{n\geq 1} \|L_n\| r^n < \infty$ and there is a positive constant $\delta = \delta(x)$ such that
\begin{equation}
\label{Taylor_expansion}
\sF(x + y) = \sF(x) + \sum_{n\geq 1} L_n(y^n), \quad\text{for all } y \in \sX \text{ with } \|y\|_\sX < \delta,
\end{equation}
where $y^n \equiv (y,\ldots,y) \in \sX \times \cdots \times \sX$ ($n$-fold product). If $\sF$ is differentiable (respectively, analytic) at every point $x \in \sU$, then $\sF$ is differentiable (respectively, analytic) on $\sU$. It is a useful observation that if $\sF$ is analytic at $x\in\sX$, then it is analytic on a ball $B_x(\varepsilon)$ (see Whittlesey \cite[p. 1078]{Whittlesey_1965}).

\subsubsection{Inverse and implicit mapping theorems for smooth and analytic maps on Banach spaces}
\label{subsubsec:Smooth_and_analytic_inverse_and_implicit_function_theorems}
Statements and proofs of the Inverse Mapping Theorem for $C^k$ maps of Banach spaces are provided by Abraham, Marsden, and Ratiu \cite[Theorem 2.5.2]{AMR}, Deimling \cite[Theorem 4.15.2]{Deimling_1985}, and Zeidler \cite[Theorem 4.F]{Zeidler_nfaa_v1}; statements and proofs of the Inverse Mapping Theorem for \emph{analytic} maps of Banach spaces are provided by Berger \cite[Corollary 3.3.2]{Berger_1977} (complex), Deimling \cite[Theorem 4.15.3]{Deimling_1985} (real or complex), and Zeidler \cite[Corollary 4.37]{Zeidler_nfaa_v1} (real or complex). The corresponding Implicit Mapping Theorems for $C^k$ or analytic maps are proved in the standard way as corollaries, for example \cite[Theorem 2.5.7]{AMR} and \cite[Theorem 4.H]{Zeidler_nfaa_v1}.

\subsubsection{Gradient maps}
\label{subsubsec:Huang_2-1B}
We recall the following basic facts concerning gradient maps.

\begin{prop}[Properties of gradient maps]
\label{prop:Huang_2-1-2}
(See Huang \cite[Proposition 2.1.2]{Huang_2006}.)
Let $\sU$ be an open subset of a Banach space, $\sX$, let $\sY$ be continuously embedded in $\sX^*$, and let $\sM:\sU \to \sY \subset \sX^*$ be a continuous map. Then the following hold.
\begin{enumerate}
\item If $\sM$ is a gradient map for $\sE$, then
\[
\sE(x_1) - \sE(x_0) = \int_0^1 \langle x_1-x_0, \sM(tx_1 + (1-t)x_0)) \rangle_{\sX\times\sX^*} \,dt, \quad\text{for all } x_0, x_1 \in \sU.
\]

\item
\label{item:Huang_2-1-2_gradient_map_iff_symmetric_derivatives}
If $\sM$ is of class $C^1$, then $\sM$ is a gradient map if and only if all of its Fr\'echet derivatives, $\sM'(x)$ for $x \in \sU$, are symmetric in the sense that
\[
\langle w,\sM'(x)v \rangle_{\sX\times\sX^*} = \langle v,\sM'(x)w \rangle_{\sX\times\sX^*}, \quad\text{for all } x \in \sU \text{ and } v,w \in \sX.
\]

\item
\label{item:Huang_2-1-2_analytic_gradient_map_implies_analytic_potential}
If $\sM$ is an analytic gradient map, then any potential $\sE:\sU\to\KK$ for $\sM$ is analytic.
\end{enumerate}
\end{prop}

\subsection{Morse Lemma for functions on Banach spaces with degenerate critical points}
\label{subsec:Morse_lemma_functions_degenerate_critical_points}
In this subsection, we prove Theorem \ref{mainthm:Hormander_C-6-1_Banach_refined} and hence Theorem \ref{mainthm:Hormander_C-6-1_Banach} by taking $\tilde{\sX} = \sX^*$.

Given Banach spaces $\sX$ and $\sZ$ over $\KK$, an open subset $\sU\subset\sX$, and a smooth map, $f:\sU \to \sZ$, and an integer $n\geq 0$, we partly follow Zeidler \cite[Sections 4.3--4.5]{Zeidler_nfaa_v1} and let $D^nf(x) = f^{(n)}(x) \in \sL^n(\sX,\sZ) = \sL(\otimes^n\sX,\sZ)$ denote the derivatives of order $n$ at a point $x \in \sU$. If $\sX = \sX_1\times\sX_2$, a product of Banach spaces $\sX_i$ over $\KK$ for $i=1,2$, we let
\[
  D_if(x_1,x_2) = f_{x_i}(x_1,x_2) \in \sL(\sX_i,\sZ)
\]
and
\[
  D_{ij}f(x_1,x_2) = f_{x_ix_j}(x_1,x_2) \in \sL(\sX_i\otimes\sX_j,\sZ)
\]
denote the first and second-order partial derivatives at a point $(x_1,x_2) \in \sU$; we may also write $D_i^2f(x_1,x_2)$ in place of $D_{ii}f(x_1,x_2)$ for $i=1,2$. We let $\GL(\sX) \subset \sL(\sX)$ denote the group of invertible operators on $\sX$.

\begin{proof}[Proof of Theorem \ref{mainthm:Hormander_C-6-1_Banach_refined}]
We generalize Ang and Tuan's proof of the \cite[Morse--Palais Lemma]{Ang_Tuan_1973}, where $f$ is $C^{p+2}$ with $p\geq 1$ and $\sY = \{0\}$ and $0\in\sX$ is a non-degenerate critical point, H\"ormander's proof of \cite[Lemma C.6.1]{Hormander_v3}, where $f$ is $C^\infty$ and $(0,0) \in \sX\times\sY$ is a degenerate critical point and $\sX=\RR^n$ and $\sY=\RR^m$, and Lang's proof of \cite[Theorem 7.5.1]{Lang_fundamentals_differential_geometry}, where $f$ is $C^{p+2}$ with $p\geq 1$ and $\sY = \{0\}$ and $0\in\sX$ is a non-degenerate critical point and $\sX$ is a real Hilbert space.

Consider the $C^{p+1}$ map,
\[
\sM:\sX\times\sY \supset \sU\times \sV \ni (x,y) \mapsto \sM(x,y) := D_1f(x,y) \in \tilde{\sX},
\]
and observe that its partial derivative with respect to $x$, that is, the $C^p$ map,
\[
D_1\sM: \sU\times \sV\times\sX \ni (x,y,\eta) \mapsto D_1\sM(x,y)\eta = D_1^2f(x,y)\eta\in \tilde{\sX},
\]
gives an isomorphism,
\[
  \sX \ni \eta \mapsto D_1\sM(0,0)\eta \in \tilde{\sX},
\]
by our hypothesis on $D_1^2f(0,0) = D_1\sM(0,0)$. By the Implicit Mapping Theorem, after possibly shrinking $\sU$ and $\sV$, there exists a $C^{p+1}$ map,
\[
\psi: \sY \supset \sV \ni y \mapsto w=\psi(y) \in \sU \subset\sX,
\]
with $\psi(0)=0$, such that $\sM(x,y) = 0$ if and only if $x = \psi(y)$, for each $y\in \sV$; moreover,
\[
  D_y\sM(\psi(y),y)) = 0 = D_1\sM(\psi(y),y)D\psi(y) + D_2\sM(\psi(y),y),
\]
where $D_y\sM(\psi(y),y)$ denotes the derivative of the one-variable map $\sM(\psi(y),y)$ with respect to $y$, and so
\[
D\psi(y) = -\left(D_1\sM(\psi(y),y)\right)^{-1}D_y\sM(\psi(y),y) \in \sL(\sY,\sX), \quad\text{for all } y \in \sV.
\]
Define a $C^{p+1}$ map,
\[
  \Psi: \sU\times \sV \ni (w,y) \mapsto (x,y) = \Psi(w,y) := (w+\psi(y),y) \in \sU\times \sV,
\]
a $C^{p+1}$ function $\tilde{f}$, and a $C^{p+1}$ map $\tilde{\sM}$ by
\begin{align*}
\tilde{f}(w,y) &:= f\circ\Psi(w,y) = f(w+\psi(y),y),
\\
\tilde{\sM}(w,y) &:= D_1\tilde{f}(w,y) = D_1 f(w+\psi(y),y)
\\
&\,= \sM(w+\psi(y),y), \quad\text{for all } (w,y) \in \sU\times \sV,
\end{align*}
noting that $D_1\Psi(w,y) = \id_\sX$, since $\Psi(w,y) = (w+\psi(y),y)$, and
\begin{multline*}
  D_1\tilde{f}(w,y) = D_1(f\circ\Psi)(w,y) = D_1f(\Psi(w,y))\circ D_1\Psi(w,y) 
  \\
  = D_1 f(w+\psi(y),y))\circ \id_\sX = D_1f(w+\psi(y),y) = D_1f(x,y),
\end{multline*}
from the Chain Rule. The map $\Psi$ has derivative
\begin{equation}
\label{eq:Derivative_Psi}
D\Psi(w,y) = \begin{pmatrix} \id_\sX & D\psi(y) \\ 0 & \id_\sY \end{pmatrix}  \in\sL(\sX\oplus\sY), \quad\text{for all } (w,y) \in \sU\times \sV,
\end{equation}
an invertible operator at each $(w,y) \in \sU\times \sV$. In particular, after possibly shrinking $\sU$ and $\sV$, the map $\Psi$ is a $C^{p+1}$ diffeomorphism of an open neighborhood of the origin in $\sX\times\sY$ by the Inverse Mapping Theorem. Therefore, the identity
\[
\sM(\psi(y),y) = 0, \quad\text{for all } y \in \sV,
\]
is equivalent to
\[
\tilde{\sM}(0,y) = 0, \quad\text{for all } y \in \sV,
\]
since $\tilde{\sM}(0,y) = \sM(\psi(y),y) \circ D\Psi(0,y)$ and $D\Psi(0,y)$ is invertible, for all $y \in \sV$. The Chain Rule and calculations similar to previous ones giving $D_1\tilde{f}(w,y) = D_1f(x,y)$ also yield
\[
D_1^2\tilde{f}(w,y) = D_1^2f(x,y) \quad\text{and thus}\quad D_1^2\tilde{f}(0,0) = D_1^2 f(0,0) = A \in \sL_\sym(\sX,\tilde{\sX}),
\]
recalling the definition of $A$ in the statement of Theorem \ref{mainthm:Hormander_C-6-1_Banach_refined}. By shrinking $\sU$ if necessary, we may assume that $\sU$ is convex and so by the second-order Taylor Formula  \cite[Section 1.4]{Lang_introduction_differential_topology} we have
\[
\tilde{f}(w,y) = \tilde{f}(0,y) + D_1\tilde{f}(0,y)w + \int_{0}^{1}(1-t)D_1^2\tilde{f}(tw,y)w^2\,dt,
\quad\text{for all } (w,y) \in \sU\times \sV,
\]
that is,
\begin{multline}
\label{eq:Hormander_C-6-1_Banach_raw}
\tilde{f}(w,y) = \tilde{f}(0,y) + \langle w,\tilde{\sM}(0,y)\rangle + \int_{0}^{1}(1-t)\left\langle w, D_1\tilde{\sM}(tw,y)w\right\rangle_{\sX\times\sX^*}\,dt,
\\
\text{for all } (w,y) \in \sU\times \sV.
\end{multline}
Therefore, using $\tilde{\sM}(0,y) = 0$ for all $y \in \sV$,
\[
\tilde{f}(w,y) = \tilde{f}(0,y) + \frac{1}{2}\langle w, B(w,y)w \rangle_{\sX\times\sX^*},
\quad\text{for all } (w,y) \in \sU\times \sV,
\]
where
\begin{equation}
\label{eq:Hormander_B_map}  
B(w,y) := 2\int_{0}^{1}(1-t)D_1\tilde{\sM}(tw,y)\,dt, \quad\text{for all } (w,y) \in \sU\times \sV.
\end{equation}
The expression \eqref{eq:Hormander_B_map} for $B$ defines a $C^p$ map,
\[
\sX\times\sY \supset \sU\times \sV \ni (w,y) \mapsto B(w,y) \in \sL_\sym(\sX,\tilde{\sX}),
\]
such that $B(0,0) = D_1\tilde{\sM}(0,0) = A$. We now generalize an argument due to Ang and Tuan (see \cite[Lemma 1]{Ang_Tuan_1973}) from the case $\tilde{\sX} = \sX^*$ to the case $\tilde{\sX} \subset \sX^*$ and make the

\begin{defn}[A closed subspace of the space of continuous, linear operators]
\label{defn:Ang_Tuan_page_642}
Let $\sL_A(\sX) \subset \sL(\sX)$ denote the closed subspace of operators $R \in \sL(\sX)$ whose adjoints $R^* \in \sL(\sX^*)$ restrict to\footnote{To avoid notational clutter, we omit explicit notation, such as $\iota:\tilde{\sX} \subset \sX^*$, for the continuous embedding.} operators $R^* \restriction \tilde{\sX} \in \sL(\tilde{\sX})$ after composition with the embedding $\tilde{\sX} \subset \sX^*$ and obey
\begin{equation}
\label{eq:Ang_Tuan_page_642}
R^*A = AR \in \sL(\sX,\tilde{\sX}).
\end{equation}
\end{defn}

We would like to write $z = R(w,y)w \in \sX$ where, after possibly further shrinking $\sU$ and $\sV$,
\[
\sX\times\sY \supset \sU\times \sV \ni (w,y) \mapsto R(w,y) \in \GL(\sX) \cap \sL_A(\sX)
\]
is a $C^p$ map such that $R(0,0)=\id_\sX$ and
\begin{equation}
\label{eq:Hormander_C-6-1}
\langle w, B(w,y)w \rangle_{\sX\times\sX^*} = \langle R(w,y)w, AR(w,y)w \rangle_{\sX\times\sX^*}, \quad\text{for all } (w,y) \in \sU\times \sV,
\end{equation}
where $B(w,y)$ is as in \eqref{eq:Hormander_B_map} and we recall that there is a continuous embedding $\tilde{\sX}\subset\sX^*$ by hypothesis. (After preparing the required foundations, the map $R$ will be contructed in the forthcoming Equation \eqref{eq:Defn_R_map}.) The identity \eqref{eq:Hormander_C-6-1} follows if we can write
\begin{equation}
\label{eq:Hormander_C-6-1_extended}
B(w,y) = R(w,y)^*AR(w,y), \quad\text{for all } (w,y) \in \sU\times \sV.
\end{equation}
Equation \eqref{eq:Hormander_C-6-1_extended} is valid at $(w,y) = (0,0)$ with $R(0,0) = \id_\sX$ and $B(0,0) = A$. We have the following generalization of Ang and Tuan \cite[Lemma 1]{Ang_Tuan_1973}.

\begin{claim}[Isomorphism onto a space of continuous, linear symmetric operators]
\label{claim:Ang_Tuan_lemma_1}
The following linear map is an isomorphism of Banach spaces,
\begin{equation}
\label{eq:Derivative_Moser_quadratic_map_extended_at_identity}
\sL_A(\sX) \ni Q \mapsto Q^*A + AQ\in \sL_\sym(\sX,\tilde{\sX}).
\end{equation}
\end{claim}

\begin{proof}
We first observe that the map \eqref{eq:Derivative_Moser_quadratic_map_extended_at_identity} is well-defined by virtue of the Definition \ref{defn:Ang_Tuan_page_642} of the subspace $\sL_A(\sX)$. Second, we show that the map \eqref{eq:Derivative_Moser_quadratic_map_extended_at_identity} is surjective. If $C\in \sL_\sym(\sX,\tilde{\sX})$, set
\[
Q := \frac{1}{2}A^{-1}C \in \sL(\sX).
\]
The adjoint of $Q$ is $Q^* = \frac{1}{2}C^*(A^{-1})^* \in \sL(\sX^*)$. Now, $A^* = A$ and $C^* = C\in \sL(\sX,\tilde{\sX})$ by our earlier discussion of properties of operators in $\sL_\sym(\sX,\tilde{\sX})$ and thus also $(A^*)^{-1} = A^{-1} \in \sL(\tilde{\sX},\sX)$. But\footnote{Because $AA^{-1}=\id_\sX=A^{-1}A$ and by \cite[Exercise 4.8]{Rudin}, one has $(A^{-1})^*A^*=\id_\sX=A^*(A^{-1})^*$, so $(A^*)^{-1}=(A^{-1})^*$.} $(A^{-1})^* = (A^*)^{-1} \in \sL(\sX^*,\tilde{\sX}^*)$ and thus $(A^{-1})^* = A^{-1} \in \sL(\tilde{\sX},\sX)$. By combining these observations, we see that $Q^* = \frac{1}{2}C^*(A^{-1})^* = \frac{1}{2}CA^{-1} \in \sL(\tilde{\sX})$, so $Q \in \sL_A(\sX)$, as required, and
\[
Q^*A + AQ = \frac{1}{2}\left(C^*(A^{-1})^*A + AA^{-1}C\right) = \frac{1}{2}\left(CA^{-1}A + AA^{-1}C\right) = C \in \sL_\sym(\sX,\tilde{\sX}),
\]
completing the proof of surjectivity. Third, we show that the map \eqref{eq:Derivative_Moser_quadratic_map_extended_at_identity} is injective. If $AQ + Q^*A = 0$, then $AQ = -Q^*A \in \sL(\sX,\tilde{\sX})$ while $AQ = Q^*A$ by \eqref{eq:Ang_Tuan_page_642} and thus $AQ = 0$ and so $Q = 0 \in \sL_A(\sX)$ since $A$ is invertible. Clearly, the map \eqref{eq:Derivative_Moser_quadratic_map_extended_at_identity} is continuous and its inverse is also continuous by the Open Mapping Theorem. This completes the proof of Claim \ref{claim:Ang_Tuan_lemma_1}.
\end{proof}

The derivative of the quadratic map,
\begin{equation}
\label{eq:Moser_quadratic_map_extended}
\sQ:\sL_A(\sX) \ni P \mapsto P^*AP \in \sL_\sym(\sX,\tilde{\sX}),
\end{equation}
at $P$ in the direction $Q$ is given by
\begin{equation}
\label{eq:Derivative_Moser_quadratic_map_extended}
D\sQ(P): \sL_A(\sX) \ni Q \mapsto Q^*AP + P^*AQ\in \sL_\sym(\sX,\tilde{\sX}).
\end{equation}
Note that the map \eqref{eq:Moser_quadratic_map_extended} is well-defined. Indeed, $(P^*AP)^* = P^*A^*P^{**} \in \sL(\sX^{**})$, where $P \in\sL(\sX)$ has adjoint operator $P^*\in\sL(\sX^*)$ and bidual operator $P^{**} \in \sL(\sX^{**})$. But $P^{**} \restriction \sX = P$ (for example, see Brezis \cite[Theorem 3.24]{Brezis} or Pietsch \cite[Chapter 0, Section A.3.6]{Pietsch_operator_ideals}) and thus $(P^*AP)^* = P^*A^*P = P^*AP \in \sL(\sX,\tilde{\sX})$ and $P^*AP$ is symmetric. When $P$ is the identity operator, we have $D\sQ(\id_\sX) = Q^*A + AQ$ and this operator is an isomorphism by Claim \ref{claim:Ang_Tuan_lemma_1}.

We now adapt the proof of Ang and Tuan \cite[Lemma 2]{Ang_Tuan_1973} and the remainder of the proof of H\"ormander \cite[Lemma C.6.1]{Hormander_v3}. The Implicit Mapping Theorem for analytic maps\footnote{Lang \cite[Theorem 5.2]{Lang_fundamentals_differential_geometry} and Palais \cite[p. 969]{Palais_1969} use a power series argument to define $F$ rather than apply the Implicit Mapping Theorem for analytic maps.} provides open neighborhoods, $\sO_\id \subset \sL_A(\sX)$ of the identity operator $\id_\sX$ and $\sO_A \subset \sL_\sym(\sX,\tilde{\sX})$ of the operator $A$, such that the restriction of the analytic map \eqref{eq:Moser_quadratic_map_extended},
\[
\sL_A(\sX) \supset \sO_\id \ni Q \mapsto Q^*AQ \in \sO_A \subset \sL_\sym(\sX,\tilde{\sX}),
\]
is an analytic diffeomorphism onto its image, with analytic inverse,
\[
F: \sL_\sym(\sX,\tilde{\sX}) \supset \sO_A \ni S \mapsto F(S) \in \sO_\id \subset \sL_A(\sX),
\]
such that $F(A) = \id_\sX$. Therefore, Equation \eqref{eq:Hormander_C-6-1} is fulfilled when we choose
\begin{equation}
\label{eq:Defn_R_map}
R = F(B),
\end{equation}
where $B$ is as in \eqref{eq:Hormander_B_map}. Substituting $z = R(w,y)w$ in Equation \eqref{eq:Hormander_C-6-1} and combining this identity with our previous expression \eqref{eq:Hormander_C-6-1_Banach_raw} for $\tilde{f}(w,y)$ yields
\begin{align*}
\tilde{f}(w,y) &= \tilde{f}(0,y) + \frac{1}{2}\langle w, B(w,y)w \rangle_{\sX\times\sX^*}
\\
&= \tilde{f}(0,y) + \frac{1}{2}\langle R(w,y)w, AR(w,y)w \rangle_{\sX\times\sX^*}
\\
&= \tilde{f}(0,y) + \frac{1}{2}\langle z, Az \rangle_{\sX\times\sX^*}, \quad\text{for all } (w,y) \in \sU\times \sV.
\end{align*}
Observe that the $C^p$ map,
\[
  \sU\times \sV \ni (w,y) \mapsto (R(w,y)w,y) \in \sX\times \sY,
\]
has derivative at the origin,
\begin{equation}
\label{eq:Derivative_Rw_idy}
\begin{pmatrix} \id_\sX & 0 \\ 0 & \id_\sY \end{pmatrix} \in\sL(\sX\oplus\sY),
\end{equation}
since $R(0,0)=\id_\sX$, and thus is locally invertible. Hence, after possibly further shrinking $\sU$ and applying the Inverse Mapping Theorem, the map
\[
\sU\times \sV \ni (w,y) \mapsto (z,y) = (R(w,y)w,y) \in \sU'\times \sV
\]
is a $C^p$ diffeomorphism onto $\sU'\times \sV$, where $\sU'$ is an open neighborhood of the origin in $\sX$. We denote its $C^p$ inverse map by
\[
\Xi: \sU'\times \sV \ni (z,y) \mapsto (w,y) \in \sU\times \sV,
\]
and note that $\Xi(0,0)=(0,0)$ with derivative at the origin,
\begin{equation}
\label{eq:Derivative_Xi}
D\Xi(0,0) = \begin{pmatrix} \id_\sX & 0 \\ 0 & \id_\sY \end{pmatrix} \in\sL(\sX\oplus\sY),
\end{equation}
by \eqref{eq:Derivative_Rw_idy}. Consequently,
\[
  \tilde{f}(\Xi(z,y)) = \tilde{f}(0,y) + \frac{1}{2}\langle z, Az \rangle_{\sX\times\sX^*},
  \quad\text{for all } (z,y) \in \sU'\times \sV.
\]
But $\tilde{f}(w,y) = f(\Psi(w,y))$ and setting $(x,y) = \Psi(w,y) = \Psi(\Xi(z,y)) =: \Phi(z,y)$, we obtain
\[
  f(\Phi(z,y)) = f(\Phi(0,y)) + \frac{1}{2}\langle z, Az \rangle_{\sX\times\sX^*},
  \quad\text{for all } (z,y) \in \sU'\times \sV,
\]
which is the desired relation \eqref{eq:Hormander_C-6-1_Banach_refined}. Equations \eqref{eq:Derivative_Psi} and \eqref{eq:Derivative_Xi} and the Chain Rule give
\[
D\Phi(0,0) = \begin{pmatrix} \id_\sX & \star \\ 0 & \id_\sY \end{pmatrix} \in\sL(\sX\oplus\sY),
\]
which is \eqref{eq:Derivative_Phi_refined}. The conclusion on analyticity of $\Phi$ follows by replacing the role of the Inverse Mapping Theorem for $C^p$ maps in the preceding arguments by its counterpart for analytic maps when $f$ is analytic (see Section \ref{subsubsec:Smooth_and_analytic_inverse_and_implicit_function_theorems}). The proof of Theorem \ref{mainthm:Hormander_C-6-1_Banach_refined} is complete.
\end{proof}

\subsection{Applications to proofs of the Morse and Morse--Bott Lemmas for functions on Banach spaces}
\label{subsec:Applications_Morse_lemma_functions_degenerate_critical_points}
We begin by recalling the

\begin{thm}[Morse Lemma for functions on Banach spaces with non-degenerate critical points]
\label{thm:Morse_Lemma_Banach}
(See Palais \cite[p. 307]{Palais_1963}, \cite[p. 968]{Palais_1969}.)
Let $\sX$ be a Banach space over $\KK$, and $\sU\subset \sX$ be an open neighborhood of the origin, and $f:\sX \supset \sU \ni x \mapsto f(x) \in \KK$ be a $C^{p+2}$ function ($p\geq 1$) such that $f(0) = 0$ and $f'(0)=0$. If $f''(0) \in \sL(\sX, \sX^*)$ is invertible\footnote{In other words, $f$ is Morse at the point $0 \in \sX$.} then there are an open neighborhood of the origin $\sV \subset \sX$ and a $C^p$ diffeomorphism, $\sV \ni y \mapsto x = \Phi(y) \in \sX$ with $\Phi(0)=0$ and $D\Phi(0)=\id_\sX$, such that
\begin{equation}
\label{eq:Morse_function_Banach}
f(\Phi(z)) = \frac{1}{2}\langle z, Az\rangle_{\sX\times\sX^*}, \quad\text{for all } z \in \sU,
\end{equation}
where
\[
  A := f''(0) = (f\circ\Phi)''(0) \in \sL_{\sym}(\sX, \sX^*).
\]
If $f$ is analytic, then $\Phi$ is analytic. If $\sX$ is a Hilbert space with norm $\|\cdot\|$, then one may further choose $\Phi$ and an orthogonal decomposition $\sX=\sX_+\oplus\sX_-$ such that
\begin{equation}
\label{eq:Morse_function_Hilbert}
f(\Phi(z)) = \frac{1}{2}\left(\|z_+\|^2 - \|z_-\|^2\right), \quad\text{for all } z = z_+ + z_- \in \sU.
\end{equation}
\end{thm}

Theorem \ref{thm:Morse_Lemma_Banach} is an immediate consequence of the more general Theorem \ref{thm:Morse-Bott_Lemma_Banach} and which is proved below (see also Lang \cite[Corollary 5.3]{Lang_fundamentals_differential_geometry}).

\begin{rmk}[Tangent space to the critical set as a subspace of the kernel of the Hessian operator]
\label{rmk:Tangent_space_critical_set_subspace_kernel_Hessian_operator}
If the critical set $\Crit f$ of a smooth function $f:\sU\to\KK$ is a smooth submanifold of $\sU$ and $x_0\in\Crit f$, then $T_{x_0}\Crit f \subset \Ker f''(x_0)$. Indeed, if $v\in T_{x_0}\Crit f$ and $\gamma(t)$ is a smooth curve in $\Crit f$ with $\gamma(0)=x_0$ and $\gamma'(0) = v$,  where $t\in(-\eps,\eps)$, then $f'(\gamma(t))=0 \in \sX^*$, since $\gamma(t) \in \Crit f$, and so the Chain Rule gives
\[
(f\circ\gamma)''(t) = f''(\gamma(t))\gamma'(t) = 0 \in \sX^*.
\]
Thus at $t=0$, we have $f(\gamma(0))''\gamma'(0) = f''(x_0)v = 0 \in \sX^*$ and hence $v \in \Ker f''(x_0)$.
\end{rmk}

We have the following generalization of Theorem \ref{thm:Morse_Lemma_Banach}.

\begin{thm}[Morse--Bott Lemma for functions on Banach spaces]
\label{thm:Morse-Bott_Lemma_Banach}
Let $\sX$ be a Banach space over $\KK$, and $\sU\subset \sX$ be an open neighborhood of the origin, and $f:\sU \to \KK$ be a $C^{p+2}$ function ($p\geq 1$) such that $f(0) = 0$. If $f$ is Morse--Bott at the origin in the sense of Definition \ref{defn:Morse-Bott_function} (\ref{item:Morse-Bott_point}), then, after possibly shrinking $\sU$, there are an open neighborhood of the origin $\sV \subset \sX$ and a $C^p$ diffeomorphism, $\sV \ni y \mapsto x = \Phi(y) \in \sU$ with $\Phi(0)=0$ and $D\Phi(0)=\id_\sX$, such that
\begin{equation}
\label{eq:Morse-Bott_function_Banach}
f(\Phi(y)) = \frac{1}{2}\langle y, Ay\rangle_{\sX\times\sX^*}, \quad\text{for all } y \in \sV,
\end{equation}
where
\[
  A := f''(0) = (f\circ\Phi)''(0) \in \sL_\sym(\sX, \sX^*).
\]
If $f$ is analytic, then $\Phi$ is analytic.
\end{thm}

\begin{rmk}[Morse--Bott Lemma for functions on Banach spaces and local coordinates]
\label{rmk:Morse-Bott_Lemma_Banach_coordinates}
By Definition \ref{defn:Morse-Bott_function} \eqref{item:Morse-Bott_point}, the closed subspace $\sK = \Ker f''(0) \subset \sX$ has a closed complement $\sX_0 \subset \sX$ such that $\sX=\sX_0\oplus\sK$ (and so $\sX^*=\sX_0^*\oplus\sK^*$ by Lemma \ref{lem:Dual_direct_sum_Banach_spaces_is_direct_sum_dual_spaces}). If $\pi \in \sL(\sX,\sX_0)$ and $\iota^* \in \sL(\sX^*,\sX_0^*)$ are the continuous projections (where $\iota:\sX_0\to\sX$ is the continuous injection), then \eqref{eq:Morse-Bott_function_Banach} becomes
\[
f(\Phi(y)) = \frac{1}{2}\langle \pi y, A_0\pi y\rangle_{\sX\times\sX^*}, \quad\text{for all } y \in \sV,
\]
where $A_0 := \iota^*A\pi \in \sL_\sym(\sX_0,\sX_0^*)$ is an isomorphism. Indeed, if we write $x = (w,\xi) \in \sX_0\oplus\sK$, then $y = \Phi(x) = (z,\xi) \in \sX_0\oplus\sK$ for all $x \in \sU$
and $A_0 = D_1^2 f(0,0) = D_1^2(f\circ\Phi)(0,0)$ and \eqref{eq:Morse-Bott_function_Banach} becomes
\[
f(\Phi(z,\xi)) = \frac{1}{2}\langle z, A_0z\rangle_{\sX\times\sX^*}, \quad\text{for all } (z,\xi) \in \sV\cap(\sX_0\oplus\sK),
\]
for coordinates adapted to the direct sum decomposition.
\end{rmk}

\begin{rmk}[Morse--Bott Lemma for functions on Hilbert spaces]
\label{rmk:Morse-Bott_Lemma_Hilbert_spaces}
Suppose now that $\sX$ is a Hilbert space over $\KK$ and identify $\sX^* \cong \sX$, so $A \in \sL(\sX)$ is self-adjoint (since $A \in \sL(\sX,\sX^*)$ is symmetric) and thus has spectrum $\sigma(A) \subset \RR$ by \cite[Theorem 12.15 (b)]{Rudin}. By the Spectral Theorem for bounded normal operators on a Hilbert space \cite[pp. 321--327]{Rudin}, there are an orthogonal decomposition into closed invariant subspaces, $\sX = \sX_{0,+}\oplus\sX_{0,-}\oplus\sK$ corresponding to the Borel subsets $(0,\infty)$, $(-\infty,0)$, and $\{0\}$ of $\sigma(A)$, continuous projections $\pi_\pm\in\sL(\sX,\sX_{0,\pm})$, and injections $\iota_\pm\in\sL(\sX_{0,\pm},\sX)$, and invertible positive operators $A^+ := \pi_+A\iota_+ \in \sL(\sX_{0,+})$ and $A^- := -\pi_-A\iota_- \in \sL(\sX_{0,-})$, such that
\begin{equation}
\label{eq:Morse-Bott_function_Hilbert}
f(\Phi(z,\xi)) = \frac{1}{2}\langle z_+, A^+z_+\rangle_\sX - \frac{1}{2}\langle z_-, A^-z_-\rangle_\sX, \quad\text{for all } (z,\xi) \in \sV\cap(\sX_0\oplus\sK),
\end{equation}
where $z_\pm = \pi_\pm z$ and $\langle \cdot,\cdot\rangle_\sX$ denotes the inner product on $\sX$. The operators $A^\pm$ have (unique) invertible positive square roots $S^\pm$ \cite[Theorem 12.33]{Rudin} and so we may define a norm on $\sX_0$ that is equivalent to $\|\cdot\|_\sX$ by setting $\|z^\pm \|_S = \|S^\pm z_\pm\|_\sX$ for all $z_\pm \in \sX_{0,\pm}$, so that \eqref{eq:Morse-Bott_function_Hilbert} becomes
\begin{equation}
\label{eq:Morse-Bott_function_Hilbert_signed}
f(\Phi(z,\xi)) = \frac{1}{2}\left(\|z^+\|_S^2 - \|z^-\|_S^2\right), \quad\text{for all } (z,\xi) \in \sV\cap(\sX_0\oplus\sK),
\end{equation}
as asserted in the special case ($\sK=0$) provided by Theorem \ref{thm:Morse_Lemma_Banach}.
\end{rmk}

\begin{rmk}[Expositions of the proofs of the Morse and Morse--Bott Lemmas for functions on Euclidean space]
\label{rmk:Morse-Bott_Lemma_proofs_Euclidean_space}
Nicolaescu provides a proof \cite[Theorem 1.12]{Nicolaescu_morse_theory} of the Morse Lemma for $C^\infty$ functions on Euclidean space (Theorem \ref{thm:Morse_Lemma_Banach} with $\sX=\RR^d$) based on that of Arnold, Gusein-Zade, and Varchenko \cite[Section 6.4]{Arnold_Gusein-Zade_Varchenko_singularities_differentiable_maps_v1} and remarks that his proof extends to yield the Morse--Bott Lemma for $C^\infty$ functions on Euclidean space (Theorem \ref{thm:Morse-Bott_Lemma_Banach} with $\sX=\RR^d$) in \cite[Proposition 2.42]{Nicolaescu_morse_theory}). See Banyaga and Hurtubise \cite[Theorem 2]{Banyaga_Hurtubise_2004} for a recent exposition of the proof of the Morse--Bott Lemma for $C^2$ functions on Euclidean space.
\end{rmk}

We turn to the more general case where the derivative of $f$ is represented by a gradient map.

\begin{thm}[Generalized Morse--Bott Lemma for functions on Banach spaces]
\label{thm:Morse-Bott_Lemma_Banach_refined}
Let $\sX$ and $\tilde{\sX}$ be Banach spaces over $\KK$ with continuous embedding $\tilde{\sX} \subset\sX^*$ and $\sU\subset \sX$ be an open neighborhood of the origin, and $f:\sU \to \KK$ be a $C^{p+1}$ function ($p\geq 1$) such that $f(0) = 0$. If $f$ is Morse--Bott at the origin in the sense of Definition \ref{defn:Morse-Bott_function_refined} (\ref{item:Morse-Bott_point_refined}), then, after possibly shrinking $\sU$, there are an open neighborhood of the origin $\sV \subset \sX$ and a $C^p$ diffeomorphism, $\sV \ni y \mapsto x = \Phi(y) \in \sU$ with $\Phi(0)=0$, such that
\begin{equation}
\label{eq:Morse-Bott_function_Banach_refined}
f(\Phi(y)) = \frac{1}{2}\langle y, Ay\rangle_{\sX\times\sX^*}, \quad\text{for all } y \in \sV,
\end{equation}
where
\[
  A := f''(0) = (f\circ\Phi)''(0) \in \sL_\sym(\sX, \tilde{\sX})
\]
and, for $\sX = \sX_0\oplus\sK$ with $\sK := \Ker f''(0)$ and a closed complement $\sX_0$,
\begin{equation}
\label{eq:Morse_Bott_DPhi_origin}
D\Phi(0,0) = \begin{pmatrix} \id_{\sX_0} & \star \\ 0 & \id_\sK \end{pmatrix} \in\sL(\sX_0\oplus\sK).
\end{equation}
If $f$ is analytic, then $\Phi$ is analytic.
\end{thm}

\begin{proof}[Proofs of Theorem \ref{thm:Morse-Bott_Lemma_Banach} and \ref{thm:Morse-Bott_Lemma_Banach_refined}]
Observe that Theorem \ref{thm:Morse-Bott_Lemma_Banach} follows immediately from Theorem  \ref{thm:Morse-Bott_Lemma_Banach_refined} by restricting to the case $\tilde{\sX} = \sX^*$, so we focus on the more general case.

Because $f$ is $C^{p+2}$ and Morse--Bott at the origin, $\Crit f \subset\sU$ is a $C^2$ submanifold by Definition \ref{defn:Morse-Bott_function_refined} \eqref{item:Morse-Bott_point_refined} and thus a $C^{p+2}$ submanifold by the Implicit Mapping Theorem. Moreover, by Definition \ref{defn:Morse-Bott_function_refined} \eqref{item:Morse-Bott_point_refined}, there is a direct sum decomposition $\sX = \sX_0\oplus\sK$, where $\sK = \Ker f''(0)$ and $\sX_0$ is a closed complement and $T_0\Crit f = \sK$. Hence, after possibly shrinking $\sU$, the Implicit Mapping Theorem provides a $C^{p+2}$ diffeomorphism $\Xi$ from an open neighborhood $\sO$ of the origin in $\sX$ onto $\sU$ such that $\Xi(0)=0$ and $D\Xi(0)=\id_\sX$ with
\[
\Crit f\circ\Xi = \sO\cap(\{0\}\oplus \sK).
\]
Therefore, we may assume without loss of generality that
\[
\Crit f = \sU\cap(\{0\}\oplus \sK).
\]
Furthermore, Definition \ref{defn:Morse-Bott_function_refined} \eqref{item:Morse-Bott_point_refined} provides that $\Ran f''(0) = \tilde{\sX}$. Hence, Theorem \ref{mainthm:Hormander_C-6-1_Banach_refined} implies that, after possibly shrinking $\sU$, there exists a $C^p$ diffeomorphism,
\[
  \Phi:\sU\cap(\sX_0\oplus\sK) \ni (z,\xi) \mapsto x = \Phi(z,\xi) \in \sX = \sX_0\oplus\sK,
\]
such that $\Phi(0,0) = 0$ and $D\Phi(0,0)$ is as in \eqref{eq:Morse_Bott_DPhi_origin} with
\[
f(\Phi(z,\xi)) = \frac{1}{2}\langle z,A_0z\rangle_{\sX\times\sX^*} + g(\xi), \quad\text{for all } (z,\xi) \in \sU\cap(\sX_0\oplus\sK),
\]
where $g(\xi) := f(\Phi(0,\xi))$, and
\[
  A_0 := D^2f(0)\restriction\sX_0 = D_1^2(f\circ\Phi)(0,0) \in \sL_\sym(\sX_0, \tilde{\sX})
\]
is an isomorphism by the Open Mapping Theorem. We observe that
\[
  D(f\circ\Phi)(z,\xi) = A_0z + Dg(\xi) \in \tilde{\sX}.
\]
Hence, $(z,\xi) \in \Crit f\circ\Phi \iff z=0$ and $Dg(\xi)=0$, that is $\xi \in \Crit g$, where $g:\sU\cap\sK\to\KK$ is a $C^p$ function with $g(0)=0$. Therefore, $\Crit f\circ\Phi = \Crit g$. In particular, $\Crit g$ is a $C^p$ submanifold of $\sU\cap\sK$, since $\Crit f\circ\Phi$ is a $C^{p+2}$ submanifold of $\sU$, and $\dim\Crit f\circ\Phi = \dim\Crit g$ with $T_0\Crit g = T_0\Crit f\circ\Phi = \sK$. Since $0 \in \Crit g$ and $g(0)=0$, there is a connected open neighborhood of the origin in $\sK$ such that $g \equiv 0$ and by shrinking $\sU$ if necessary, we may assume that $g \equiv 0$ on $\sU\cap\sK$. Hence,
\[
  D(f\circ\Phi)(z,\xi) = A_0z = A(z,\xi)
\]
by writing $A \in \sL(\sX_0\oplus\sK,\tilde{\sX})$ as
\[
A(z,\xi)  = A_0z, \quad\text{for all } (z,\xi) \in \sX_0\oplus\sK.
\]
If $f$ is analytic, then $\Phi$ is analytic by the Implicit Mapping Theorem for analytic functions. The proofs of Theorem \ref{thm:Morse-Bott_Lemma_Banach} and \ref{thm:Morse-Bott_Lemma_Banach_refined} are complete.
\end{proof}

When $\sX=\CC^d$, then Theorem \ref{thm:Morse-Bott_Lemma_Banach} yields the

\begin{cor}[Holomorphic Morse--Bott Lemma for functions on $\CC^d$]
\label{cor:Morse-Bott_Lemma_holomorphic}
(See \cite{Postnikov_Rudyak_morse_lemma} for a statement in the case $c=0$ and Petro \cite[Lemma 3.8]{PetroThesis} for a statement in the case $c\geq 0$; compare Seidel \cite[Lemma 1.6]{Seidel_2003}.)
Let $d\geq 2$ be an integer, $U\subset \CC^d$ be an open neighborhood of the origin, and $f:U \ni x \mapsto f(x) \in \CC$ be a holomorphic function such that $f(0) = 0$ and $f'(0)=0$. Assume that $\Crit f$ is a complex submanifold of $U$ with complex tangent space $T_0\Crit f = \Ker f''(0)$ of dimension $c \geq 0$ at the origin. Then, after possibly shrinking $U$, there are an open neighborhood $V\subset\CC^d$ of the origin and a complex analytic diffeomorphism,
\[
  V \ni (w_1,\ldots,w_d) \mapsto (x_1,\ldots,x_d) = \Phi(w_1,\ldots,w_d) \in \CC^d,
\]
onto an open neighborhood of the origin in $\CC^d$ such that
\[
\Phi^{-1}(U\cap \Crit\sE) = V \cap (\CC^c\cap 0) \subset \CC^c\times\CC^{d-c}
\]
with $\Phi(0)=0$ and
\[
D\Phi(0) = \begin{pmatrix} \id_{d-c} & \star \\ 0 & \id_c \end{pmatrix} \in \GL(d,\CC),
\]
where $\id_{d-c} \in \GL(d-c,\CC)$ and $\id_c \in \GL(c,\CC)$ and
\begin{equation}
\label{eq:Morse-Bott_Lemma_holomorphic}
f(\Phi(w_1,\ldots,w_d)) = w_1^2 + \cdots + w_{d-c}^2, \quad\text{for all } w = (w_1,\ldots,w_d) \in U.
\end{equation}
\end{cor}

\section{{\L}ojasiewicz gradient inequality for functions on Banach spaces}
\label{sec:Morse_lemma_degenerate_critical_points_Lojasiewicz-Simon_inequality}
In Section \ref{subsec:Proof_Lojasiewicz_gradient_inequality_Morse-Bott}, we use the Morse--Bott Lemma for $C^{p+2}$ functions ($p\geq 1$) (see Theorems \ref{thm:Morse-Bott_Lemma_Banach} and \ref{thm:Morse-Bott_Lemma_Banach_refined}) to give a concise proof of the {\L}ojasiewicz gradient inequality for $C^{p+2}$ Morse--Bott functions on Banach spaces (see Theorems \ref{mainthm:Lojasiewicz_gradient_inequality_Morse-Bott} and \ref{mainthm:Lojasiewicz_gradient_inequality_Morse-Bott_refined}); in Section \ref{subsec:Proof_Lojasiewicz_gradient_inequality_analytic}, we apply the Morse Lemma for analytic functions with degenerate critical points (see Theorems \ref{mainthm:Hormander_C-6-1_Banach} and \ref{mainthm:Hormander_C-6-1_Banach_refined}) to give an elegant proof of the {\L}ojasiewicz gradient inequality for analytic functions on Banach spaces (see Theorems \ref{mainthm:Lojasiewicz-Simon_gradient_inequality} and \ref{mainthm:Lojasiewicz-Simon_gradient_inequality_refined}).

\subsection{{\L}ojasiewicz gradient inequality for smooth Morse--Bott functions}
\label{subsec:Proof_Lojasiewicz_gradient_inequality_Morse-Bott}
In this subsection, we prove Theorem \ref{mainthm:Lojasiewicz_gradient_inequality_Morse-Bott_refined}, and hence Theorem \ref{mainthm:Lojasiewicz_gradient_inequality_Morse-Bott} upon choosing $\tilde{\sX}=\sX^*$. We begin with the

\begin{lem}[{\L}ojasiewicz gradient inequality for quadratic forms]
\label{lem:Quadratic_form_Lojasiewicz_exponent_one_half_refined}
Let $\sX$ and $\tilde{\sX}$ be Banach spaces over $\KK$ with continuous embedding $\tilde{\sX}\subset\sX^*$. If
\[
  Q:\sX \ni x \mapsto Q(x) = \frac{1}{2}\langle x,Ax\rangle_{\sX\times\sX^*}\in \KK
\]
is defined by a symmetric operator $A \in \sL_\sym(\sX,\tilde{\sX})$, whose kernel is complemented in $\sX$ and whose range is $\tilde{\sX}$, then $Q$ has {\L}ojasiewicz exponent $1/2$, that is, there is a constant $C\in(0,\infty)$ such that
\begin{equation}
\label{eq:Quadratic_form_Lojasiewicz_exponent_one_half_refined}
\|Q'(x)\|_{\tilde{\sX}} \geq CQ(x)^{1/2}, \quad\text{for all } x\in\sX.
\end{equation}
\end{lem}

\begin{proof}
The derivative of $Q:\sX\to\KK$ is given by
\[
  Q'(x)v = \frac{1}{2}\langle v,Ax \rangle_{\sX\times\sX^*} + \frac{1}{2}\langle x,Av \rangle_{\sX\times\sX^*}
  = \langle v,Ax \rangle_{\sX\times\sX^*} = Ax(v), \quad\text{for all } x, v \in \sX,
\]
so $Q'(x) = Ax \in \tilde{\sX} \subset \sX^*$. By hypothesis, $\sX = \sX_0\oplus\sK$ as a direct sum of Banach spaces, where $\sK := \Ker A$ and $\sX_0 \subset \sX$ is a closed subspace, and $\Ran A = \tilde{\sX}$, so that $A \in \sL(\sX_0,\tilde{\sX})$ is an isomorphism of Banach spaces by the Open Mapping Theorem. Note that for $x = z + \xi \in \sX_0\oplus \sK$, we have
\begin{multline*}
  Q(z + \xi) = \frac{1}{2}\langle z + \xi ,A(z+\xi)\rangle_{\sX\times\sX^*} = \frac{1}{2}\langle z + \xi ,Az\rangle_{\sX\times\sX^*}
  \\
  = \frac{1}{2}\langle z ,A(z + \xi)\rangle_{\sX\times\sX^*} = \frac{1}{2}\langle z,Az\rangle_{\sX\times\sX^*} = Q(z),
\end{multline*}
while
\[
Q'(z + \xi) = A(z+\xi) = Az = Q'(z).
\]
Hence, it suffices to prove that the {\L}ojasiewicz gradient inequality \eqref{eq:Quadratic_form_Lojasiewicz_exponent_one_half_refined} holds for all $x\in \sX_0$. For such $x \in \sX_0$, we have
\[
\|Q'(x)\|_{\tilde{\sX}} = \|Ax\|_{\tilde{\sX}} \geq \lambda\|x\|_\sX,
\]
by writing
\[
  \|x\|_\sX = \|A^{-1}Ax\|_\sX \leq \|A^{-1}\|_{\sL(\tilde{\sX},\sX_0)}\|Ax\|_{\tilde{\sX}}
\]
and denoting $\lambda := \|A^{-1}\|_{\sL(\tilde{\sX},\sX_0)} \in (0,\infty)$. On the other hand, for any $x\in\sX$,
\[
|Q(x)| \leq \frac{1}{2}\left|\langle x,Ax\rangle_{\sX\times\sX^*}\right| \leq \frac{1}{2}\|x\|_\sX\|Ax\|_{\sX^*}  \leq \frac{\kappa}{2}\|x\|_\sX\|Ax\|_{\tilde{\sX}} \leq \frac{\kappa\Lambda}{2}\|x\|_\sX^2,
\]
where we denote $\Lambda := \|A\|_{\sL(\sX,\tilde{\sX})} \in (0,\infty)$ and where $\kappa$ is the norm of the continuous embedding $\tilde{\sX} \subset \sX^*$. Therefore,
\[
\|Q'(x)\|_{\tilde{\sX}} \geq \lambda\|x\|_\sX \geq \lambda\left(2|Q(x)|/\kappa\Lambda\right)^{1/2} = \lambda\sqrt{2/\kappa\Lambda}\, |Q(x)|^{1/2},
\]
for all $x\in\sX_0$ and this yields the {\L}ojasiewicz gradient inequality \eqref{eq:Quadratic_form_Lojasiewicz_exponent_one_half_refined} for all $x\in \sX$.
\end{proof}

We have the following generalization of Lemma \ref{lem:Pushforwards_preserve_Lojasiewicz_exponents}.

\begin{lem}[{\L}ojasiewicz exponents and maps]
\label{lem:Pushforwards_preserve_Lojasiewicz_exponents_Banach_spaces}
Let $\sX$, $\tilde{\sX}$, and $\sY$ be Banach spaces over $\KK$ with continuous embedding $\tilde{\sX}\subset\sX^*$, and $\sV\subset \sY$ and $\sU\subset \sX$ be open neighborhoods of the origins, and $\Phi:\sV \to \sU$ be an open $C^1$ map such that $\Phi(0) = 0$. Let $f:\sU\to\KK$ be a $C^1$ function such that $f(0) = 0$ and $f'(x)\in\tilde{\sX}$ for all $x\in\sU$. If $f\circ\Phi$ obeys the {\L}ojasiewicz gradient inequality
\begin{equation}
  \label{eq:Lojasiewicz-Simon_gradient_inequality_analytic_functional_general_Y}
 \|(f\circ\Phi)'(y)\|_{\sY^*}
\geq C|(f\circ\Phi)(y)|^\theta, \quad\text{for all } y \in \sV,
\end{equation}
with exponent $\theta \geq 0$ then, after possibly shrinking $\sU$, the function $f(x)$ obeys the {\L}ojasiewicz gradient inequality \eqref{eq:Lojasiewicz-Simon_gradient_inequality_analytic_functional_general} with the same exponent $\theta$ and a possibly smaller constant $C\in(0,\infty)$, for all $x\in\sU$.
\end{lem}

\begin{proof}
Because $\Phi(0) = 0$ and $\Phi$ is an open map, $\Phi(\sV)$ is an open neighborhood of the origin in $\sX$ and so by shrinking $\sU$ if necessary, we may assume that $\Phi(\sV) = \sU$. Now $(f\circ\Phi)(y) = f(\Phi(y)) = f(x)$ for all $x \in \sU$ and $y \in \Phi^{-1}(x)$ and therefore the gradient inequality \eqref{eq:Lojasiewicz-Simon_gradient_inequality_analytic_functional_general_Y} yields
\begin{equation}
\label{eq:Lojasiewicz_gradient_inequality_composition_with_morphism}
\|(f\circ\Phi)'(y)\|_{\sY^*}
\geq C|f(x)|^\theta, \quad\text{for all } x \in \sU \text{ and } y \in \Phi^{-1}(x).
\end{equation}
The Chain Rule provides
\[
  (f\circ\Phi)'(y) = f'(\Phi(y))\circ\Phi'(y) \in \sY^*
\]
with $\Phi'(y) \in \sL(\sY,\sX)$ and $f'(\Phi(y)) \in \tilde{\sX} \subset \sX^*$, so
\begin{align*}
\|(f\circ\Phi)'(y)\|_{\sY^*}
&= \|f'(\Phi(y))\circ\Phi'(y)\|_{\sY^*}
\\
&\leq \|f'(\Phi(y))\|_{\sX^*} \|\Phi'(y)\|_{\sL(\sY,\sX)}
\\    
&\leq \kappa\|f'(\Phi(y))\|_{\tilde{\sX}} \|\Phi'(y)\|_{\sL(\sY,\sX)}
\\
&\leq \kappa M\|f'(\Phi(y))\|_{\tilde{\sX}}
\quad\text{for all } y \in \sV,
\end{align*}
where $M := \sup_{y \in \sV} \|\Phi'(y)\|_{\sL(\sX)}$ and $M < \infty$ (possibly after shrinking $\sV$) and $\kappa$ is the norm of the continuous embedding $\tilde{\sX} \subset \sX^*$. Because $\Phi(y)=x\in \sU$, the preceding inequality simplifies to give
\begin{equation}
\label{eq:Gradient_composition_with_morphism_lower_bound}
\|(f\circ\Phi)'(y)\|_{\sY^*}
\leq
\kappa M\|f'(x)\|_{\tilde{\sX}}, \quad\text{for all } y \in \sV.
\end{equation}
By combining the inequalities \eqref{eq:Lojasiewicz_gradient_inequality_composition_with_morphism} and \eqref{eq:Gradient_composition_with_morphism_lower_bound}, we obtain
\[
\|f'(x)\|_{\tilde{\sX}} \geq (C/(\kappa M))|f(x)|^\theta, \quad\text{for all } x \in \sU,
\]
which is \eqref{eq:Lojasiewicz-Simon_gradient_inequality_analytic_functional_general} with constant $C/(\kappa M)$, as desired.
\end{proof}

\begin{proof}[Proof of Theorem \ref{mainthm:Lojasiewicz_gradient_inequality_Morse-Bott_refined}]
By hypothesis, $f$ is a $C^{p+1}$ Morse--Bott function at the origin and so, possibly after shrinking $\sU$, Theorem \ref{thm:Morse-Bott_Lemma_Banach_refined} provides an open neighborhood $\sV$ of the origin in $\sX$ and a $C^p$ diffeomorphism $\Phi:\sV\to\sU$ such that $\Phi(0)=0$ and
\[
(f\circ\Phi)(y) = \langle y, Ay \rangle_{\sX\times\sX^*}, \quad\text{for all } y \in \sV,
\]
where $A = f''(0) = (f\circ\Phi)''(0) \in \sL_\sym(\sX,\tilde{\sX})$. By Definition \ref{defn:Morse-Bott_function_refined} \eqref{item:Morse-Bott_point_refined}, the kernel of $A$ has a closed complement in $\sX$ and the range of $A$ is $\tilde{\sX}$. Lemma \ref{lem:Quadratic_form_Lojasiewicz_exponent_one_half_refined} then asserts that the quadratic function, $Q(y) = \langle y, Ay \rangle_{\sX\times\sX^*}$ for all $y\in\sX$, has {\L}ojasiewicz exponent $1/2$, while Lemma \ref{lem:Pushforwards_preserve_Lojasiewicz_exponents_Banach_spaces} implies that the functions $f$ and $f\circ\Phi$ have the same {\L}ojasiewicz exponent, namely $1/2$. This completes the proof of Theorem \ref{mainthm:Lojasiewicz_gradient_inequality_Morse-Bott_refined}.
\end{proof}

\subsection{{\L}ojasiewicz gradient inequality for analytic functions}
\label{subsec:Proof_Lojasiewicz_gradient_inequality_analytic}
In this subsection, we apply Theorem \ref{mainthm:Hormander_C-6-1_Banach_refined} to prove Theorem \ref{mainthm:Lojasiewicz-Simon_gradient_inequality_refined}, and hence Theorem \ref{mainthm:Lojasiewicz-Simon_gradient_inequality} upon choosing $\tilde{\sX}=\sX^*$. We first have the elementary

\begin{lem}[Invariance of {\L}ojasiewicz exponent under direct sum addition or subtraction of a quadratic form]
\label{lem:Invariance_Lojasiewicz_exponent_addition_quadratic_form_refined}
Let $\sX$, $\tilde{\sX}$, $\sY$, and $\tilde{\sY}$ be Banach spaces over $\KK$ with continuous embeddings $\tilde{\sX}\subset\sX^*$ and $\tilde{\sY}\subset\sY^*$, and $\theta \in [1/2,1)$ be a constant, $\sU \subset \sX$ be an open neighborhood of the origin, $f:\sX \supset \sU \to \KK$ be a $C^2$ function with $f(0)=0\in\KK$ and $f'(0)=0$ and $f'(x)\in\tilde{\sX}$ for all $x\in\sU$, and
\[
  Q:\sY \ni y \mapsto Q(y) = \frac{1}{2}\langle y,Ay\rangle_{\sY\times\sY^*}\in \KK
\]
be defined by an operator $A \in \sL_\sym(\sY,\tilde{\sY})$ whose kernel is complemented in $\sY$ and whose range is $\tilde{\sY}$. If $f_Q:\sU\times\sY\to\KK$ is a $C^2$ function defined by $f_Q(x,y) := f(x) + Q(y)$ for $(x,y) \in \sU\times\sY$, then there are constants $C, C_0 \in (0,\infty)$ and an open neighborhood $\sV\subset \sY$ of the origin such that, after possibly shrinking $\sU$, the following holds: $f$ has {\L}ojasiewicz exponent $\theta$ on $\sU$, that is,
\begin{equation}
\label{eq:Lojasiewicz_gradient_inequality_body_refined}
\|f'(x)\|_{\tilde{\sX}} \geq C_0|f(x)|^\theta, \quad\text{for all } x \in \sU,
\end{equation}
if and only if $f_Q$ has {\L}ojasiewicz exponent $\theta$ on $\sU\times \sV$, that is,
\begin{equation}
\label{eq:Lojasiewicz_gradient_inequality_plus_quadratic_form_refined}
\|f_Q'(x,y)\|_{\tilde{\sX}\oplus\tilde{\sY}} \geq C|f_Q(x,y)|^\theta, \quad\text{for all } (x,y) \in \sU\times \sV.
\end{equation}
\end{lem}

\begin{proof}
Let $\alpha := 1/\theta \in (1,2]$ and suppose that Inequality \eqref{eq:Lojasiewicz_gradient_inequality_body_refined} holds. Since $f'(0)=0$ and $Q'(0)=0$, we may assume that $\|f'(x)\oplus Q'(y)\|_{\tilde{\sX}\oplus\tilde{\sY}} \leq 1$ for all $(x,y) \in \sU\times \sV$, for small enough $\sV$ and after possibly shrinking $\sU$. Observe that for all $(x,y) \in \sU\times \sV$,
\begin{align*}
|f_Q(x,y)| & \leq |f(x)| + |Q(y)|
\\
&\leq C_0\|f'(x)\|_{\tilde{\sX}}^\alpha + C_1\|Q'(y)\|_{\tilde{\sY}}^2
\quad\text{(by Lemma \ref{lem:Quadratic_form_Lojasiewicz_exponent_one_half_refined} and Inequality \eqref{eq:Lojasiewicz_gradient_inequality_body_refined})}
\\
&\leq C\left(\|f'(x)\|_{\tilde{\sX}}^\alpha + \|Q'(y)\|_{\tilde{\sY}}^2\right)
\\
&\leq C\left(\|f'(x)\oplus Q'(y)\|_{\tilde{\sX}\oplus\tilde{\sY}}^\alpha + \|f'(x)\oplus Q'(y)\|_{\tilde{\sX}\oplus\tilde{\sY}}^2\right)
\\
&\leq C\|f'(x)\oplus Q'(y)\|_{\tilde{\sX}\oplus\tilde{\sY}}^\alpha \quad \text{(as $\|f'(x)\oplus Q'(y)\|_{\tilde{\sX}\oplus\tilde{\sY}} \leq 1$ and $\alpha\in(1,2]$)}
\\
&= C\|f_Q'(x,y)\|_{\tilde{\sX}\oplus\tilde{\sY}}^\alpha,
\end{align*}
where $C = \max\{C_0, C_1\}$ and  $f_Q'(x,y) = f'(x)\oplus Q'(y)$. Taking the $1/\alpha$ root of the preceding inequality yields Inequality \eqref{eq:Lojasiewicz_gradient_inequality_plus_quadratic_form_refined}.

Conversely, suppose that Inequality \eqref{eq:Lojasiewicz_gradient_inequality_plus_quadratic_form_refined} holds. For all $x \in \sU$,
\begin{align*}
|f(x)| &= |f_Q(x,0)|
\\
&\leq C\|f_Q'(x,0)\|_{\tilde{\sX}\oplus\tilde{\sY}}^\alpha
\quad\text{(by Inequality \eqref{eq:Lojasiewicz_gradient_inequality_plus_quadratic_form_refined})}
\\
&= C\|f'(x)\oplus Q'(0)\|_{\tilde{\sX}\oplus\tilde{\sY}}^\alpha
\\
&= C\left(\|f'(x)\|_{\tilde{\sX}} + \|Q'(0)\|_{\tilde{\sY}}\right)^\alpha
\\
&= C\|f'(x)\|_{\tilde{\sX}}^\alpha \quad\text{(since $Q'(0)=0$)},
\end{align*}
which gives Inequality \eqref{eq:Lojasiewicz_gradient_inequality_body_refined} after taking the $1/\alpha$ root. This completes the proof of Lemma \ref{lem:Invariance_Lojasiewicz_exponent_addition_quadratic_form_refined}.
\end{proof}

We can now give the

\begin{proof}[Proof of Theorem \ref{mainthm:Lojasiewicz-Simon_gradient_inequality_refined}]
The operator $f''(0) \in \sL_\sym(\sX,\tilde{\sX})$ is Fredholm by hypothesis, with finite-dimensional kernel $\sK := \Ker f''(0)$ and closed complement $\sX_0 \subset \sX$ such that $\sX = \sX_0\oplus\sK$. Similarly, let $\tilde{\sX}_0 := \Ran f''(0) \subset \tilde{\sX}$ denote the closed range of $f''(0)$ with finite-dimensional complement $\tilde{\sK} \cong \sK = \Ker f''(0)$ and isomorphism $\tilde{\sX}_0 \cong \sX_0$ (see Lemma \ref{lem:Isomorphism_properties_Fredholm_operator}). Therefore, writing $x=(w,\xi) \in \sX=\sX_0\oplus\sK$,
\[
f''(0,0) = \begin{pmatrix} A_0 & 0 \\ 0 & 0 \end{pmatrix}:\sX_0\oplus\sK \to \tilde{\sX}_0\oplus\tilde{\sK},
\]
where $A_0 = D_1^2f(0,0) \in \sL(\sX_0,\tilde{\sX}_0)$ is symmetric with respect to the continuous embedding $\tilde{\sX}_0 \subset \sX_0^*$ and canonical pairing $\sX_0\times\sX_0^*\to\KK$. Moreover, $A_0$ is bijective and continuous by construction, so it is invertible by the Open Mapping Theorem.

By hypothesis, $f$ is analytic and so, possibly after shrinking $\sU$, Theorem \ref{mainthm:Hormander_C-6-1_Banach_refined} provides an open neighborhood $\sV$ of the origin in $\sX$ and an analytic diffeomorphism $\Phi:\sV\to\sU$ such that $\Phi(0,0)=(0,0)$ and
\[
f\circ\Phi(z,\xi) = g(\xi) + \langle z, A_0z \rangle_{\sX_0\times\sX_0^*}, \quad\text{for all } y = (z,\xi) \in \sV,
\]
where $A_0 = D_1^2(f\circ\Phi)(0,0) \in \sL_\sym(\sX_0,\tilde{\sX}_0)$ and $g(\xi) := f(\Phi(0,\xi))$ for all $\xi$ in $V$, an open neighborhood of the origin in $\sK$ defined as the image of the projection of $\sV \subset \sX = \sX_0\oplus\sK$ onto the factor $\sK$. Lemma \ref{lem:Quadratic_form_Lojasiewicz_exponent_one_half_refined} then asserts that the quadratic function, $Q(z) := \langle z, A_0z \rangle_{\sX_0\times\sX_0^*}$ for $z\in\sX_0$, has {\L}ojasiewicz exponent $1/2$ on an open neighborhood of the origin, while Lemmas \ref{lem:Pushforwards_preserve_Lojasiewicz_exponents_Banach_spaces} and \ref{lem:Invariance_Lojasiewicz_exponent_addition_quadratic_form_refined} imply that the functions $f:\sU\to\KK$ and $f\circ\Phi:\sV\to\KK$ and $g:V\to\KK$ have the same {\L}ojasiewicz exponent. But $\sK$ is a finite-dimensional vector space over $\KK$ and $g$ is analytic and thus obeys the classical {\L}ojasiewicz gradient inequality for some exponent $\theta \in [1/2,1)$ by Theorem \ref{thm:Lojasiewicz_gradient_inequality}. This completes the proof of Theorem \ref{mainthm:Lojasiewicz-Simon_gradient_inequality_refined}.
\end{proof}

\section{Analytic functions with {\L}ojasiewicz exponent one half are Morse--Bott}
\label{sec:Characterization_optimal_exponent_Morse-Bott_condition}
Our goal in this section is to complete the proof of Theorem \ref{mainthm:Analytic_function_Lojasiewicz_exponent_one-half_Morse-Bott_Banach_refined} and hence Theorem \ref{mainthm:Analytic_function_Lojasiewicz_exponent_one-half_Morse-Bott_Banach} upon choosing $\tilde{\sX}=\sX^*$.

\begin{proof}[Proof of Theorem \ref{mainthm:Analytic_function_Lojasiewicz_exponent_one-half_Morse-Bott_Banach_refined}]
Recall that $f''(0) \in \sL_\sym(\sX,\tilde{\sX})$ is a Fredholm operator by hypothesis. Let $\sK := \Ker f''(0) \subset \sX$ denote the finite-dimensional kernel with closed complement $\sX_0$, so $\sX = \sX_0\oplus\sK$, and let $\tilde{\sX}_0 := \Ran f''(0) \subset \tilde{\sX}$ denote the closed range, with finite-dimensional complement $\tilde{\sK}$, so $\tilde{\sX} = \tilde{\sX}_0\oplus\tilde{\sK}$.

We apply the Morse Lemma for functions on Banach spaces with degenerate critical points (Theorem \ref{mainthm:Hormander_C-6-1_Banach_refined}) to $f$ to produce --- after possibly shrinking the open neighborhood $\sU$ of the origin in $\sX$ --- an analytic diffeomorphism $\Phi$ from an open neighborhood $\sV$ of the origin in $\sX = \sX_0\oplus\sK$ onto $\sU$ such that
\[
\Phi^*f(z,\xi) = \frac{1}{2}\langle z,A_0z\rangle_{\sX_0\times\sX_0^*} + \Phi^*f(0,\xi), \quad\text{for all } y = (z,\xi) \in \sV \subset \sX_0\oplus\sK,
\]
where $A_0 := D_1^2(\Phi^*f)(0,0) \in \sL_\sym(\sX_0,\tilde{\sX}_0)$ and we recall that $\tilde{\sX}_0\subset\sX_0^*$ is a continuous embedding, as proved in the paragraphs immediately following the statement of Theorem \ref{mainthm:Analytic_function_Lojasiewicz_exponent_one-half_Morse-Bott_Banach_refined}. Again, the operator $A_0$ is bijective and continuous by construction, so it is invertible by the Open Mapping Theorem. The function $g := f(\Phi(0,\cdot))$ is analytic on an open neighborhood $V$ of the origin in $\sK$ defined as the image of the projection of $\sV \subset \sX = \sX_0\oplus\sK$ onto the factor $\sK$. By construction, $g(0)=0$ and $g'(0)=0$.

By shrinking $\sV$ if necessary, we may assume without loss of generality that $V$ is connected. If $g$ is identically zero on $V$, then we are done. Otherwise, if $g$ is not identically zero on $V$, our hypothesis that $f$ has {\L}ojasiewicz exponent $1/2$ and Lemmas \ref{lem:Pushforwards_preserve_Lojasiewicz_exponents_Banach_spaces} and \ref{lem:Invariance_Lojasiewicz_exponent_addition_quadratic_form_refined} imply that $g$ has {\L}ojasiewicz exponent $1/2$ as well.

If $\Ker g''(0) = \{0\}$, then $\Coker g''(0) = \{0\}$, since $g''(0) \in \sL(\sK,\sK^*)$ is symmetric\footnote{Note that $\sK$ is finite-dimensional.}, and thus $g''(0) \in \sL(\sK,\sK^*)$ is invertible. Hence, $g$ is a Morse--Bott function --- in fact a Morse function with $\Crit g = \{0\}$  --- and thus $\Phi^*f$ is a Morse--Bott function. But then $f$ itself must be a Morse--Bott function since $\Phi$ is a diffeomorphism from one open neighborhood of the origin in $\sX$ onto another and this would complete the proof of Theorem \ref{mainthm:Analytic_function_Lojasiewicz_exponent_one-half_Morse-Bott_Banach_refined}.

If $\Ker g''(0) \neq \{0\}$, then there exists $v \in \sK$ such that $\|v\|_\sK = 1$ and, since $g$ is analytic, an integer $m \geq 3$ such that $g^{(m)}(0)v^m \neq 0$. The Taylor Formula then yields
\begin{align*}
g(tv) &= \frac{1}{m!}g^{(m)}(0)v^m\, t^m + \frac{1}{(m+1)!}\int_0^t g^{(m+1)}(sv)v^{m+1}s^m\,ds,
\\
g'(tv) &= \frac{1}{(m-1)!}g^{(m)}(0)v^{m-1}\, t^{m-1} + \frac{1}{m!}\int_0^t g^{(m+1)}(sv)v^ms^{m-1}\,ds,
\end{align*}
for all $t \in \KK$ such that $tv \in V$. Therefore, after possibly further shrinking $\sV$ and hence $V$, and noting that $\|\xi\|_\sK = \|tv\|_\sK = |t|$ for all $\xi=tv \in V\cap\KK v$, we have
\begin{gather}
\label{eq:Taylor_formula_g}  
\frac{1}{C_0}\|\xi\|_\sK^m \leq |g(\xi)| \leq C_0\|\xi\|_\sK^m,
\\
\label{eq:Taylor_formula_g_prime}  
\frac{1}{C_1}\|\xi\|_\sK^{m-1} \leq \|g'(\xi)\|_{\sK^*} \leq C_1\|\xi\|_\sK^{m-1},
\end{gather}
for positive constants $C_0, C_1$ depending at most on $m$ and $|g^{(m)}(0)v^m|$ and $\sup_{\zeta\in V}\|g^{(m+1)}(\zeta)\|$, where $g^{(m+1)}(\zeta) \in \otimes^{m+1}\sK^*$. Because $g$ has {\L}ojasiewicz exponent $1/2$, then the {\L}ojasiewicz gradient inequality \eqref{eq:Lojasiewicz_gradient_inequality} yields, after possibly further shrinking $V$,
\[
\|g'(\xi)\|_{\sK^*} \geq C_2|g(\xi)|^{1/2}, \quad\text{for all } \xi \in V,
\]
for some positive constant $C_2$. By combining the preceding inequality with the inequalities \eqref{eq:Taylor_formula_g} and \eqref{eq:Taylor_formula_g_prime}, we obtain
\[
   C_1\|\xi\|_\sK^{m-1} \geq C_2\left(C_0^{-1}\|\xi\|_\sK^m\right)^{1/2}, \quad\text{for all } \xi \in V,
\]
that is,
\[
   \|\xi\|_\sK^{m-1} \geq C_0^{-1/2}C_1^{-1}C_2 \|\xi\|_\sK^{m/2}, \quad\text{for all } \xi \in V.
\]
Consequently, since $V$ is an open neighborhood of the origin in $\sK$, we must have $m-1 \leq m/2$ or equivalently, $2m-2\leq m$, that is, $m \leq 2$ and contradicting our assumption that $m \geq 3$ and $\sK\neq\{0\}$. Hence, we must have $\Ker g''(0) = \{0\}$ and this completes the proof of Theorem \ref{mainthm:Analytic_function_Lojasiewicz_exponent_one-half_Morse-Bott_Banach_refined}.
\end{proof}

\appendix

\section{Rate of convergence of a gradient flow for a function obeying a {\L}ojasiewicz gradient inequality}
\label{sec:Convergence_rate}
We recall the following enhancement of Huang \cite[Theorem 3.4.8]{Huang_2006}.

\begin{thm}[Convergence rate under the validity of a {\L}ojasiewicz gradient inequality]
\label{thm:Huang_3-4-8_introduction}
(See Feehan \cite[Theorem 3]{Feehan_yang_mills_gradient_flow_v4}.)
Let $\sU$ be an open subset of a real Banach space $\sX$ that is continuously embedded and dense in a Hilbert space $\sH$. Let $\sE:\sU\to\RR$ be an analytic function with gradient map $\sE':\sU \to \sH$ and $x_\infty \in \sU$ be a critical point, that is, $\sE'(x_\infty)=0$. Assume that there are constants $c \in (0,\infty)$, and $\sigma \in (0,1]$,
and $\theta \in [1/2,1)$ such that
\begin{equation}
\label{eq:Simon_2-2_dualspacenorm_introduction}
\|\sE'(x)\|_{\sH} \geq c|\sE(x) - \sE(x_\infty)|^\theta, \quad \text{for all } x \in \sU_\sigma,
\end{equation}
where $\sU_\sigma := \{x\in \sX: \|x-x_\infty\|_\sX < \sigma\}$. Let $u \in C^\infty([0,\infty); \sX)$ be a solution to the gradient system
\begin{equation}
\label{eq:gradient_system}
\dot u(t) = -\sE'(u(t)), \quad t \in (0,\infty),
\end{equation}
and assume that the orbit $O(u) := \{u(t): t\geq 0\} \subset \sX$ obeys $O(u) \subset \sU_\sigma$. Then there exists $u_\infty \in \sH$ such that
\begin{equation}
\label{eq:Huang_3-45_H_introduction}
\|u(t) - u_\infty\|_\sH \leq \Psi(t), \quad t\geq 0,
\end{equation}
where
\begin{equation}
\label{eq:Huang_3-45_growth_rate_introduction}
\Psi(t)
:=
\begin{cases}
\displaystyle
\frac{1}{c(1-\theta)}\left(c^2(2\theta-1)t + (\gamma-a)^{1-2\theta}\right)^{-(1-\theta)/(2\theta-1)},
& 1/2 < \theta < 1,
\\
\displaystyle
\frac{2}{c}\sqrt{\gamma-a}\exp(-c^2t/2),
&\theta = 1/2,
\end{cases}
\end{equation}
and $a, \gamma$ are constants such that $\gamma > a$ and
\[
a \leq \sE(v) \leq \gamma, \quad\text{for all } v \in \sU.
\]
If in addition $u$ obeys Hypothesis \ref{hyp:Abstract_apriori_interior_estimate_trajectory_main_introduction}, then $u_\infty \in \sX$ and
\begin{equation}
\label{eq:Huang_3-45_X_introduction}
\|u(t+1) - u_\infty\|_\sX \leq 2C_1\Psi(t), \quad t\geq 0,
\end{equation}
where $C_1 \in [1,\infty)$ is the constant in Hypothesis \ref{hyp:Abstract_apriori_interior_estimate_trajectory_main_introduction} for $\delta=1$.
\end{thm}

We recall the

\begin{hyp}[\Apriori interior estimate for a trajectory]
\label{hyp:Abstract_apriori_interior_estimate_trajectory_main_introduction}
(See Feehan \cite[Hypothesis 2.1]{Feehan_yang_mills_gradient_flow_v4}.)
Let $\sX$ be a Banach space that is continuously embedded in a Hilbert space $\sH$. If $\delta \in (0,\infty)$ is a constant, then there is a constant $C_1 = C_1(\delta) \in [1,\infty)$ with the following significance. If $S, T \in \RR$ are constants obeying $S+\delta \leq T$ and $u \in C^\infty([S,T); \sX)$, we say that $\dot u \in C^\infty([S,T); \sX)$ obeys an \apriori \emph{interior estimate on $(0, T]$} if
\begin{equation}
\label{eq:Abstract_apriori_interior_estimate_trajectory_main_introduction}
\int_{S+\delta}^T \|\dot u(t)\|_\sX\,dt \leq C_1\int_S^T \|\dot u(t)\|_\sH\,dt.
\end{equation}
\end{hyp}

In applications, $u \in C^\infty([S,T); \sX)$ in Hypothesis \ref{hyp:Abstract_apriori_interior_estimate_trajectory_main_introduction} will often be a solution to a quasi-linear parabolic partial differential system, from which an \apriori estimate \eqref{eq:Abstract_apriori_interior_estimate_trajectory_main_introduction} may be deduced. For example, Hypothesis \ref{hyp:Abstract_apriori_interior_estimate_trajectory_main_introduction} is verified by Feehan \cite[Lemma 17.12]{Feehan_yang_mills_gradient_flow_v4} for a nonlinear evolution equation on a Banach space $\calV$ of the form (see Caps \cite{Caps_evolution_equations_scales_banach_spaces}, Henry \cite{Henry_geometric_theory_semilinear_parabolic_equations}, Pazy \cite{Pazy_1983}, Sell and You \cite{Sell_You_2002}, Tanabe \cite{Tanabe_1979, Tanabe_1997} or Yagi \cite{Yagi_abstract_parabolic_evolution_equations_applications})
\begin{equation}
\label{eq:Nonlinear_evolution_equation_Banach_space}
\frac{du}{dt} + \cA u = \cF(t,u(t)), \quad t\geq 0, \quad u(0) = u_0,
\end{equation}
where $\cA$ is a positive, sectorial, unbounded operator on a Banach space $\cW$ with domain $\calV^2 \subset \cW$ and the nonlinearity $\cF$ has suitable properties.

Results on the rate of convergence of a gradient flow defined by a function obeying a {\L}ojasiewicz gradient inequality in specific examples have been proved earlier --- see Simon \cite{Simon_1983} and Adams and Simon \cite{Adams_Simon_1988} for a restricted class of analytic energy functions arising in geometric analysis and R\r{a}de \cite[Proposition 7.4]{Rade_1992} for the Yang-Mills energy function on connections on principal bundles over a closed smooth manifold of dimension two or three. For a recent example, see Carlotto, Chodosh, and Rubinstein \cite[Theorem 1]{Carlotto_Chodosh_Rubinstein_2015} for the Yamabe function on Riemannian metrics over closed smooth manifolds of dimension greater than or equal to three.

\section{Morse--Bott functions and quadratic simple normal crossing functions}
\label{sec:Blow-up_map_exceptional_divisor_polar_coordinates}
We are often asked about the relationship between Morse--Bott functions and quadratic simple normal crossing functions as in \eqref{eq:Pullback_analytic_function_simple_normal_crossing_Lojasiewicz_one_half}, so we explain the relationship in this section for $\KK=\RR$; the analogous discussion applies for $\KK=\CC$.

For an integer $p \geq 1$ and writing $\RR^* = \RR\less\{0\}$, we let $\RR^{p-1} = (\RR^p\less\{0\})/\RR^* = S^{p-1}/\{\pm 1\}$ denote real projective space, so $\RR\PP^0 \cong \{1\}$ and $\RR\PP^1 \cong S^1$ while $\RR\PP^{p-1}$ with $p\geq 3$ is obtained by identifying antipodal points of the sphere $S^{p-1}$.

\begin{defn}[Blowup at a point and exceptional divisor]
\label{defn:Blowup}
(See Krantz and Parks \cite[Definition 6.2.2]{Krantz_Parks_primer_real_analytic_functions}.) Let $p\geq 1$ be an integer and $W$ be an open neighborhood of the origin in $\RR^p$. The \emph{blowup} of $W$ at the origin is the set
\[
\widetilde{W} := \{(x,\ell) \in W \times \RR^{p-1} : x \in \ell\},
\]
where $\pi: \widetilde{W} \ni (x,\ell) \mapsto x \in W$ is the \emph{blowup map} and $E := \pi^{-1}(0) \subset \widetilde{W}$ is the \emph{exceptional divisor}.
\end{defn}

The set $\widetilde{W}$ is a real analytic manifold and the quotient map $\pi: \widetilde{W} \to W$ is real analytic and restricts to a real analytic diffeomorphism $\pi: \widetilde{W} \less E \cong W \less\{0\}$. By viewing $\RR^{p-1} = S^{p-1}/\{\pm 1\}$ and $\RR^p \less\{0\} = \RR_+\times S^{p-1}$, we may also write
\begin{align*}
\widetilde{W} &= \{(x,\ell) \in W \times \RR^{p-1} : x \in \ell\}
\\
&= \{(x,[u]) \in W \times S^{p-1}/\{\pm 1\}: x \in \RR u\}
\\
&= \{(x,[u]) \in W \times S^{p-1}/\{\pm 1\}: x = \pm |x| u\}
\\
&= \{(x,u) \in W \times S^{p-1}: (x,u) \sim (y,v) \text{ if } |x|=|y| \text{ and } u = \pm v \}
\\
&= \{(s,u) \in \RR \times S^{p-1}: su \in W \text{ and } (s,u) \sim (t,v) \text{ if } (t,v) = \pm(s,u) \},
\end{align*}
where $[u] = \{\pm u\}$ and, in the last line, the blowup map is $\pi: \widetilde{W} \ni [s,u] \mapsto su \in W$ and
\[
  \pi^{-1}(0) = \{\{0\} \times S^{p-1}\}/\{\pm 1\} \cong S^{p-1}/\{\pm 1\} = \RR^{p-1}
\]
is the exceptional divisor.

If $W$ is an open neighborhood of the origin in $\RR^p$ or the half-space $\HH^p = \{x\in\RR^p:x_p\geq 0\}$, then we could alternatively define the blowup of $W$ at the origin to be the real analytic manifold with boundary,
\[
\widehat{W} := \{(r,u) \in [0,\infty)\times S^{p-1}: ru \in W\},
\]
following the usual definition of polar coordinates on $\RR^p\less\{0\}$. The map $\pi: \widehat{W} \ni (r,u) \mapsto ru \in W$ is the blowup map and $\pi^{-1}(0) = \{0\}\times S^{p-1} \cong S^{p-1}$ is now the exceptional divisor.

Suppose now that $U\subset\RR^d$ is an open neighborhood of the origin and $f:U \to \RR$ is a $C^2$ function with $f(0)=0$ and $f'(0)=0$ and that is Morse--Bott at the origin in the sense of Definition \ref{defn:Morse-Bott_function} \eqref{item:Morse-Bott_point}. Thus, after possibly shrinking $U$, we have that $\Crit f$ is a $C^2$ submanifold of $U$ of dimension $c = \dim\Ker f''(0)$. Moreover, we may further assume that $U$ is connected and so $\Crit f \subset f^{-1}(0)$.

Theorem \ref{thm:Morse-Bott_Lemma_Banach} and Remark \ref{rmk:Morse-Bott_Lemma_Hilbert_spaces} (the Morse--Bott Lemma) imply, after possibly shrinking $U$, that one can find an neighborhood $V$ of the origin in $\RR^d$ and a $C^2$ diffeomorphism\footnote{While for clarity we have restricted our attention in this article to functions $f$ which are $C^{p+2}$ with $p\geq 1$, the Morse--Bott Lemma holds for $C^2$ functions on Euclidean space: see Banyaga and Hurtubise \cite[Theorem 2]{Banyaga_Hurtubise_2004}.} $\Phi:\RR^d \supset V \ni y\mapsto x\in U \subset\RR^d$ such that $\Phi(0)=0$ and
\[
f\circ\Phi(y) = \sum_{i=1}^{p} y_i^2 - \sum_{i=p+1}^{p+n} y_i^2, \quad\text{for all } y \in V.
\]
Note that $p+n = d-c$ and
\[
\Crit f\circ\Phi = V\cap\bigcap_{i=1}^{d-c}\{y_i=0\}.
\]
If $n=0$ and thus $1 \leq p = d-c$, we may write $(y_1,\ldots,y_p) = su$, for $s \in [0,\infty)$ and $u \in S^{p-1} \subset \RR^p$, so that
\[
f\circ\varpi(s,u,y_{p+1},\cdots,y_d) = s^2, \quad\text{for all } (su,y_{p+1},\cdots,y_d) \in U,
\]
where we define
\[
  \varpi(s,u,y_{p+1},\cdots,y_d) := \Phi(su,y_{p+1},\cdots,y_d).
\]
We see that $\varpi$ gives a $C^2$ map from an open neighborhood $V$ of the origin in $[0,\infty)\times S^{p-1}\times \RR^{d-p}$ onto $U\subset\RR^d$ such that $\varpi(0)=0$ and
\[
\varpi(\{s=0\}\cap V) = U\cap\bigcap_{i=1}^p\{y_i=0\} = U\cap\bigcap_{i=1}^{d-c}\{y_i=0\},
\]
and $\varpi$ is a diffeomorphism from $V\less\{s=0\}$ onto its image.

Similarly, if $p=0$ and thus $1 \leq n = d-c$, we may write $(y_1,\ldots,y_n) = tv$, for $t \in [0,\infty)$ and $v \in S^{n-1} \subset \RR^n$, so that
\[
f\circ\varpi(t,v,y_{n+1},\cdots,y_d) = -t^2, \quad\text{for all } (tv,y_{n+1},\cdots,y_d) \in U,
\]
where we define $\varpi(t,v,y_{n+1},\cdots,y_d) = \Phi(tv,y_{n+1},\cdots,y_d)$. We see that $\varpi$ gives a $C^2$ map from an open neighborhood of the origin in $[0,\infty)\times S^{p-1}\times \RR^{d-p}$ into $\RR^d$ such that $\varpi(0)=0$ and
\[
\varpi(\{t=0\}\cap V) = U\cap\bigcap_{i=1}^n\{y_i=0\} = U\cap\bigcap_{i=1}^{d-c}\{y_i=0\},
\]
and $\varpi$ is a diffeomorphism from $V\less\{t=0\}$ onto its image.

Finally, if $n \geq 1$ and $p \geq 1$, we may write $(y_1,\ldots,y_p) = su$ and $(y_{p+1},\ldots,y_{p+n}) = tv$, for $s, t \in [0,\infty)$ and $u \in S^{p-1}$ and $v \in S^{n-1}$, so that
\[
f\circ\varpi(s,t,u,v,y_{p+n+1},\cdots,y_d) = s^2-t^2, \quad\text{for all } (su,tv,y_{p+n+1},\cdots,y_d) \in U,
\]
where we define
\[
  \varpi(s,t,u,v,y_{p+n+1},\cdots,y_d) := \Phi(su,tv,y_{p+n+1},\cdots,y_d).
\]
We see that $\varpi$ gives a $C^2$ map from an open neighborhood of the origin in
\[
  [0,\infty)\times[0,\infty)\times S^{p-1}\times S^{n-1}\times \RR^{d-n-p}
\]
into $\RR^d$ such that $\varpi(0)=0$ and
\[
\varpi(\{s=0\}\cap \{t=0\}\cap W) = V\cap\bigcap_{i=1}^{n+p}\{y_i=0\} = V\cap\bigcap_{i=1}^{d-c}\{y_i=0\},
\]
after possibly shrinking $V$ and $\varpi$ is a diffeomorphism from $W\less (\{s=0\}\cup\{t=0\})$ onto its image.

Define a diffeomorphism of $\RR^2$ by $(t_1,t_2) \mapsto (s,t) = \varphi(t_1,t_2)$ where $t_1=s+t$ and $t_2=s-t$, so that $s = \frac{1}{2}(t_1+t_2)$ and $t = \frac{1}{2}(t_1-t_2)$. Hence, we obtain
\[
f\circ\Pi(t_1,t_2,u,v,y_{p+n+1},\cdots,y_d) = t_1t_2,
\quad\text{for all } (\varphi_1(t_1,t_2)u,\varphi_2(t_1,t_2)v,y_{p+n+1},\cdots,y_d) \in U,
\]
where we define
\[
  \Pi(t_1,t_2,u,v,y_{p+n+1},\cdots,y_d) := \Phi(\varphi_1(t_1,t_2)u,\varphi_2(t_1,t_2)v,y_{p+n+1},\cdots,y_d).
\]
We see that $\Pi$ gives a $C^2$ map from an open neighborhood of the origin in
\[
  \{(t_1,t_2)\in [0,\infty)\times\RR: |t_2|\leq t_1\} \times S^{p-1}\times S^{n-1}\times \RR^{d-n-p}
\]
into $\RR^d$ such that $\Pi(0)=0$ and
\[
\Pi(\{t_1=0\}\cap \{t_2=0\}\cap V) = U\cap\bigcap_{i=1}^{d-c}\{y_i=0\},
\]
and $\Pi$ is a diffeomorphism from $V\less (\{t_1=0\}\cup\{t_2=0\})$ onto its image.

In the preceding discussion we could have replaced the roles of the blowups $[0,\infty)\times S^{p-1}$ or $[0,\infty)\times S^{n-1}$ by $(\RR\times S^{p-1})/\{\pm 1\}$ or $(\RR\times S^{n-1})/\{\pm 1\}$ and the roles of the exceptional divisors, $S^{p-1}$ or $S^{n-1}$ by $\RR\PP^{p-1}$ or $\RR\PP^{n-1}$, the only difference being an increase in notational complexity. In summary, we have proved the

\begin{prop}[Pull-back of a Morse--Bott function to a quadratic simple normal crossing function]
\label{prop:Pull-back_Morse-Bott_function_simple_normal_crossing}
Let $d \geq 2$ be an integer, $U\subset\RR^d$ be an open neighborhood of the origin, and $f:U \to \RR$ be a $C^2$ function that is Morse--Bott at the origin and obeys $f(0)=0$. Then, after possibly shrinking $U$, there are an open neighborhood $V$ of the origin in $\RR^d$ and a $C^2$ map $\pi:V \to U$ such that $\pi$ restricts to a diffeomorphism from $V\less\{y_1=0\}$ or $V\less(\{y_1=0\}\cup\{y_2=0\})$ onto its image and
\[
\pi^*f(y) = \pm y_1^2 \quad\text{or}\quad y_1y_2, \quad\text{for all } y=(y_1,\ldots,y_d) \in V,
\]
and $\pi(\Crit f\circ\pi) = \Crit f$, where $\Crit f\circ\pi = \{y_1=0\}\cap V$ or $(\{y_1=y_2=0\})\cap V$.
\end{prop}

\section{Integrability and  Morse--Bott conditions for the harmonic map energy and the area functions}
\label{sec:Integrability_and_Morse-Bott_properties}
In Section \ref{subsec:Relationship_Adams-Simon_integrability_Morse-Bott_point_conditions} we defined the concepts of Jacobi vector, integrable Jacobi vector, and integrable critical point (see Definition \ref{defn:Jacobi_vectors_integrability}). We noted (see Lemma \ref{lem:Morse-Bott_property_implies_integrability}) that if a function is Morse--Bott at a critical point, then that critical point is integrable. Theorem \ref{mainthm:Integrability_implies_Morse-Bott_property} has been proved by Simon for a specific class of analytic functions on certain Banach spaces (given by $C^{2,\alpha}$ sections of a Riemannian vector bundle over a closed Riemannian manifold) that includes the harmonic map energy and the area functions. We shall give a proof of a more general version of Theorem \ref{mainthm:Integrability_implies_Morse-Bott_property} elsewhere \cite{Feehan_jacobi_vectors_analytic_potential_functions}, but we outline here how Theorem \ref{mainthm:Integrability_implies_Morse-Bott_property} may be proved; in addition to the references cited below, we also refer the reader to Simon \cite[Sections 3.11--3.14 and 3.13.16]{Simon_1996} for further expository details.

\begin{proof}[Outline of proof of Theorem \ref{mainthm:Integrability_implies_Morse-Bott_property}]
  In order to avoid notational conflict with the remainder of this section, we let $\sE$ denote the analytic function considered in Theorem \ref{mainthm:Integrability_implies_Morse-Bott_property}. As in Simon's proof of his infinite-dimensional version \cite[Theorem 3]{Simon_1983} of the {\L}ojasiewicz gradient inequality, one first applies Lyapunov--Schmidt reduction as in \cite[Section 2]{Simon_1983} to the function $\sE:\sU\to\KK$. This step yields an analytic embedding $\Psi:\sV\cap\sK \to \sX$ of the intersection with the kernel $\sK:=\Ker\sE''(x_0)$ and an open neighborhood $\sV$ of the origin in $\tilde{\sX}$, together with an analytic function $\Gamma = \sE\circ\Psi:\sV\cap\sK\to\KK$ (see Adams and Simon \cite[p. 230]{Adams_Simon_1988}, Feehan and Maridakis \cite[Lemmas 2.3 and 2.5]{Feehan_Maridakis_Lojasiewicz-Simon_Banach}, Simon \cite[pp. 538--539]{Simon_1983}, or Simon \cite[Part II, Section 6]{Simon_1985}). By hypothesis, $x_0$ is an integrable critical point in the sense of Definition \ref{defn:Jacobi_vectors_integrability} and so, after possibly shrinking $\sV$, the function $\Gamma$ is constant on $\sV\cap\sK$ by Adams and Simon \cite[Lemma 1, p. 231]{Adams_Simon_1988}.

One can now show that $\sE'(x)=0$ for all $x \in \Psi(\sV\cap\sK)$, essentially by reversing our proof of \cite[Lemma 2.5]{Feehan_Maridakis_Lojasiewicz-Simon_Banach} or arguing as in Simon \cite[p. 539]{Simon_1983}, and thus
\[
  \Psi(\sV\cap\sK) \subseteq \Crit\sE.
\]
By hypothesis, $\sE''(x_0) \in \sL(\sX,\tilde{\sX})$ is Fredholm with index zero and thus we have $\sX=\sX_0\oplus\sK$ and $\tilde{\sX}\cong\sX_0\oplus\sK$ (see Lemma \ref{lem:Isomorphism_properties_Fredholm_operator} \eqref{item:Fredholm_index_zero}). In particular,
\[
  \Ran\sE''(x_0) + \sK = \tilde{\sX}
\]
and so, after possibly shrinking $\sU$, the analytic gradient map $\sE':\sU\to\tilde{\sX}$ is transverse to the (linear) submanifold $\sK\subset\tilde{\sX}$ and hence the preimage $(\sE')^{-1}(\sK)$ is an open analytic submanifold of $\sX$ by the Preimage Theorem from differential topology \cite[Proposition II.2.4]{Lang_fundamentals_differential_geometry}. We thus have inclusions
\[
  \Psi(\sV\cap\sK) \subseteq \Crit\sE \subseteq (\sE')^{-1}(\sK),
\]
noting that $\Crit\sE \equiv (\sE')^{-1}(0)$. Furthermore, we have
\[
  T_{x_0}\Psi(\sV\cap\sK) = \sK = T_{x_0}(\sE')^{-1}(\sK),
\]
where the first equality follows from the construction of $\Psi$ (see \cite[Lemma 2.3]{Feehan_Maridakis_Lojasiewicz-Simon_Banach}) and the second from the observations below:
\begin{align*}
  T_{x_0}(\sE')^{-1}(\sK) &= (\sE''(x_0))^{-1}(T_{\sE'(x_0)}\sK) = (\sE''(x_0))^{-1}(\sK)
  \\
                          &= (\sE''(x_0))^{-1}(0) \quad\text{(since $\tilde{\sX}=\Ran\sE''(x_0)\oplus\sK$)}
  \\
  &= \sK \quad\text{(by definition)}.
\end{align*}
Hence, after possibly shrinking $\sU$ or $\sV$, we have $\Psi(\sV\cap\sK) = (\sE')^{-1}(\sK)$ and consequently
\[
  \Psi(\sV\cap\sK) = \Crit\sE = (\sE')^{-1}(\sK).
\]
In particular, $\Crit\sE$ is an open analytic submanifold of $\sX$ with tangent space $T_{x_0}\Crit\sE = \Ker\sE''(x_0)$ and so $\sE$ is Morse--Bott at $x_0$ in the sense of Definition \ref{defn:Morse-Bott_function_refined} \eqref{item:Morse-Bott_point_refined}.
\end{proof}  

\subsection{Integrability and  Morse--Bott conditions for the harmonic map energy function}
\label{subsec:Integrability_and_Morse-Bott_harmonic_map_energy_functional}
Following Lemaire and Wood \cite[Section 1]{Lemaire_Wood_2002}, we review the concept of \emph{integrability} of a \emph{Jacobi field} along a harmonic map and describe the relation between integrability and the Morse--Bott condition for the harmonic map energy function at a harmonic map. We then list a few examples where integrability is known for harmonic maps.\footnote{This appendix is a revised version of Feehan and Maridakis \cite[Appendix A]{Feehan_Maridakis_Lojasiewicz-Simon_harmonic_maps_v5}.}

We begin by recalling the \emph{second variation of the energy} for the harmonic energy function $\sE$ discussed in Section \ref{subsubsec:Harmonic_map_energy_function_maps_Riemann_surface_into_closed_Riemannian_manifold}. For a smooth two-parameter variation $f_{t,s}:M\to N$ of a map $f:M\to N$ with $f_{0,0}=f$ and $\partial f_{t,s}/\partial t|_{(0,0)} = v$ and $\partial^2 f_{t,s}/\partial s|_{(0,0)} = w$, the  \emph{Hessian} of $\sE$ at $f$ is defined by
\[
\sE''(f)(v,w)
:=
\left.\frac{\partial^2\sE(f_{t,s})}{\partial t\partial s}\right|_{(0,0)},
\]
where $\sE$ is as in \eqref{eq:Harmonic_map_energy_functional}. One has
\[
\sE''(f)(v,w)
=
(J_f(v),w)_{L^2(M,g)},
\]
where
\[
J_f(v) := \Delta v - \tr R^N(df,v)df
\]
is called the \emph{Jacobi operator}, a self-adjoint linear elliptic differential operator. Here, $\Delta$ denotes the Laplacian induced on $f^{-1}TN$ and the sign conventions for $\Delta$ and the curvature $R^N$ are those of Eells and Lemaire \cite{Eells_Lemaire_1983}.

Let $v$ be a \emph{vector field along $f$}, that is, a smooth section of $f^{-1}TN$, where $f:M\to N$ is a smooth map. Then $v$ is called a \emph{Jacobi field} (for the energy) if $J_f(v) = 0$. The space of Jacobi fields $\Ker J_f$ is finite-dimensional and its dimension is called the ($\sE$)-\emph{nullity} of $f$.

\begin{defn}[Integrability of a Jacobi field along a harmonic map]
\label{defn:Lemaire_Wood_2002_1-2}
(See Lemaire and Wood \cite[Definition 1.2]{Lemaire_Wood_2002}.)
A Jacobi field $v$ along a harmonic map $f_0:M\to N$ is said to be \emph{integrable} if there is a smooth family of harmonic maps, $f_t:M\to N$ for $t\in (-\eps,\eps)$, such that $f_t|_{t=0} = f_0$ and $v = \partial f_t/\partial t|_{t=0}$.
\end{defn}

The following result is stated by Kwon in her Ph.D. thesis \cite{KwonThesis} (directed by Simon); it can be deduced from Theorem \ref{mainthm:Integrability_implies_Morse-Bott_property} by applying, for example, methods of Feehan and Maridakis \cite{Feehan_Maridakis_Lojasiewicz-Simon_application_harmonic_maps}.    

\begin{thm}[Integrability of Jacobi fields and manifolds of harmonic maps]
\label{thm:Kwon_4-1_harmonic_map}
(See Kwon \cite[Proposition 4.1]{KwonThesis}.) 
Let $d\geq 2$ be an integer and $\alpha\in(0,1)$ be a constant. Let $(M,g)$ and $(N,h)$ be closed, smooth Riemannian manifolds, with $M$ of dimension $d$, and assume that there is a smooth isometric embedding $N\subset\RR^n$ for some integer $n$. If $f_0 \in C^\infty(M;N)$ is a harmonic map, so $\sE'(f_0)=0$, then the following hold:
\begin{enumerate}
\item\label{item:harmonic_maps_critical_submanifold}
If there is a constant $\delta=\delta(f_0,g,h,n,\alpha) \in (0,1]$ such that
\begin{equation}
\label{eq:harmonic_maps_critical_set}  
U_{f_0,\delta} := \left\{f \in C^{2,\alpha}(M;N): \|f-f_0\|_{C^{2,\alpha}(M;\RR^n)} < \delta \text{ and } \sE'(f)=0\right\}
\end{equation}  
is an open smooth manifold with tangent space $T_{f_0}U_{f_0,\delta} = \Ker\sE''(f_0)$ at $f_0$, then every Jacobi vector field in $\Ker\sE''(f_0)$ is integrable.
\item\label{item:harmonic_maps_unique_integrability_jacobi_fields}
If $(N,h)$ is real analytic, the isometric embedding $N\subset\RR^n$ is real analytic, and every Jacobi vector field in $\Ker\sE''(f_0)$ is integrable, then there is a constant $\delta=\delta(f_0,g,h,n,\alpha) \in (0,1]$ such that the set $U_{f_0,\delta}$ in \eqref{eq:harmonic_maps_critical_set} is an open smooth manifold with tangent space $T_{f_0}U_{f_0,\delta} = \Ker\sE''(f_0)$ at $f_0$.
\end{enumerate}
\end{thm}

It follows that for real-analytic target manifolds, all Jacobi fields along all harmonic maps are integrable if and only if the space of harmonic maps is a manifold whose tangent bundle is given by the Jacobi fields \cite[p. 470]{Lemaire_Wood_2002}. By Definition \ref{defn:Morse-Bott_function}, the conclusion of Theorem \ref{thm:Kwon_4-1_harmonic_map} \eqref{item:harmonic_maps_unique_integrability_jacobi_fields} is equivalent to the assertion that all Jacobi fields along $f_0$ are integrable if and only if the harmonic map energy function $\sE$ is Morse--Bott at $f_0$.

For a further discussion of integrability and additional references, see Adams and Simon \cite[Section 1]{Adams_Simon_1988}, Kwon \cite[Section 4.1]{KwonThesis}, and Simon \cite[pp. 270--272]{Simon_1985}.

According to \cite[Theorem 1.3]{Lemaire_Wood_2002}, any Jacobi field along a harmonic map from $S^2$ to $\CC\PP^2$ is integrable, where the two-sphere $S^2$ has its unique conformal structure and the complex projective space $\CC\PP^2$ has its standard Fubini-Study metric of holomorphic sectional curvature $1$; see Crawford \cite{Crawford_1997} for additional results.

From the list of examples provided by Lemaire and Wood \cite[p. 471]{Lemaire_Wood_2002}, there are few other examples of families of harmonic maps that are guaranteed to be integrable, with the list including harmonic maps from $S^2$ to $S^2$ but excluding harmonic maps from $S^2$ to $S^3$ or $S^4$ \cite{Lemaire_Wood_2009}.

Fern{\'a}ndez \cite{Fernandez_2012} has proved that the space $\Harm_d(S^2,S^{2n})$ of degree-$d$ harmonic maps from $S^2$ into $S^{2n}$ has dimension $2d+n^2$. However, thus far, integrability for such maps is known only when $n=1$. Bolton and Fernandez \cite{Bolton_Fernandez_2011} provide a nice survey of what is known regarding regularity of $\Harm_d(S^2,S^{2n})$: they recall that  $\Harm_d(S^2,S^2)$ is known to be a smooth manifold, outline a proof that $\Harm_d(S^2,S^6)$ is also a smooth manifold, and survey results on the structure of $\Harm_d(S^2,S^4)$ and why that space is not a smooth manifold.

\subsection{Integrability and  Morse--Bott conditions for the area function}
\label{subsec:Integrability_and_Morse-Bott_area_functional}
Suppose that $m,n \geq 1$ and $r\geq 2$ are integers and $M$ is a closed, connected, oriented, smooth manifold of dimension $m$. We let $C^{r,\alpha}(M;\RR^n)$ denote the Banach space of $C^{r,\alpha}$ maps from $M$ into $\RR^n$, where $\alpha\in[0,1]$, and let $\Imm^{r,\alpha}(M;\RR^n) \subset C^{r,\alpha}(M;\RR^n)$ denote the open subset of $C^{r,\alpha}$ immersions, and let $\Emb^{r,\alpha}(M;\RR^n) \subset C^{r,\alpha}(M;\RR^n)$ denote the open subset of $C^{r,\alpha}$ embeddings. If $\Phi \in C^{r,\alpha}(M;\RR^d)$ is an embedding, then $g_\Phi := \Phi^*g$ is a Riemannian metric on $M$ while if $\Phi$ is an immersion, then $g_\Phi$ may be singular. We now consider the \emph{area} or \emph{volume function},
\[
\Imm^{r,\alpha}(M;\RR^n) \ni \Phi \mapsto \sE(\Phi) := \Vol(M, g_\Phi) \in [0,\infty).
\]
Then $\Phi(M)$ is called a \emph{critical immersed submanifold} or (as customary) a \emph{minimal immersed submanifold} if $\sE'(\Phi)=0$, where
\[
\sE'(\Phi)\eta =  \left.\frac{d}{dt}\Vol(M, g_{\Phi+t\Phi_\eta})\right|_{t=0}
\]
for all vector fields $\eta \in C^{r,\alpha}(TM)$. One can show that
\[
\sE'(\Phi)\eta = \left(\eta,\sE'(\Phi)\right)_{L^2(M)},
\]
with an explicit expression for the gradient $\sE'(\Phi)$ provided by the \emph{first variation formula} --- see Calegari \cite[Proposition 2.1]{Calegari_2014}, Colding and Minicozzi \cite[pp. 154--155]{Colding_Minicozzi_2006}, Dajczer and Tojeiro \cite[Proposition 3.1]{Dajczer_Tojeiro_submanifold_theory}, Lawson \cite{Lawson_lectures_minimal_submanifolds}, Schoen \cite[Section 2.1]{Schoen_2014}, or Xin \cite[Theorem 1.2.2 and Remark 1.2.5]{Xin_minimal_submanifolds_related_topics}.

An explicit expression for the Hessian $\sE''(\Phi)$ at a critical point $\Phi$ is provided by the \emph{second variation formula} --- see Calegari \cite[Proposition 3.1]{Calegari_2014}, Colding and Minicozzi \cite[pp. 154--155]{Colding_Minicozzi_2006}, Lawson \cite{Lawson_lectures_minimal_submanifolds}, Schoen \cite[Section 2.1]{Schoen_2014}, and Xin \cite[Theorem 6.1.1]{Xin_minimal_submanifolds_related_topics}.

More generally, if we replace $\RR^n$ in the preceding discussion by a connected, smooth manifold $N$ without boundary, then it is known that $C^r(M;N)$ is a smooth Banach manifold --- see Abraham \cite{Abraham_1963}, Bruveris \cite{Bruveris_2016}, Eichhorn \cite{Eichhorn_1993}, Eliasson \cite{Eliasson_1967}, or Wittmann \cite{Wittmann_2019}. (It is highly likely that published proofs of this result extend to show that $C^{r,\alpha}(M;N)$ is a Banach manifold when $\alpha\in[0,1]$ and, furthermore, that $W^{k,p}(M;N)$ is Banach manifold for $k\in\NN$ and $p\in[1,\infty)$, at least for $k\geq 2$ and $kp>m$, taking note of the Sobolev Embedding Theorem \cite[Theorem 4.12]{AdamsFournier}.) We refer to Michor and Mumford \cite[Section 2.1]{Michor_Mumford_2005} for their analysis of these spaces in the $C^\infty$ category.

Recall from Dajczer and Tojeiro \cite[Corollary 3.7]{Dajczer_Tojeiro_submanifold_theory} or Xin \cite[Corollary 1.3.4]{Xin_minimal_submanifolds_related_topics} that there exists no minimal isometric immersion $\Phi: M^m \to \RR^n$ of a compact Riemannian manifold without boundary. Hence, we restrict our attention to cases where $M$ and $N$ are closed or $M$ and $N$ are complete or $M$ is compact with boundary and $N$ is complete.

One could again derive an analogue of Theorem \ref{thm:Kwon_4-1_harmonic_map} (giving the relationship between integrability of Jacobi vector fields and the Morse--Bott property of an immersed minimal submanifold) from Theorem \ref{mainthm:Integrability_implies_Morse-Bott_property} or derive an analogue of Theorem \ref{thm:Kwon_4-1_harmonic_map} for immersed minimal submanifolds from prior, more general results of Simon \cite{Simon_1983, Simon_1985} and Adams and Simon \cite{Adams_Simon_1988}.

Adams and Simon list examples of minimal submanifolds whose Jacobi vector fields are all integrable as well as examples that have nontrivial Jacobi vector fields that are not integrable \cite[pp. 249--252]{Adams_Simon_1988}. See also Allard and Almgren \cite[Section 6]{Allard_Almgren_1981}, Nagura \cite{Nagura_1981, Nagura_1982a, Nagura_1982b}, Simons \cite{Simons_1968}, and Smith \cite{Smith_1975pams, Smith_1975ajm} (via Remark \ref{rmk:Relationship_harmonic_maps_minimal_surfaces}) for related examples.

White \cite{White_1991, White_2017} has shown that for generic $C^r$ Riemannian metrics on a manifold $N$, there are no closed, immersed, minimal submanifolds $M\subset N$ with nontrivial Jacobi fields; the case of geodesics, including immersed geodesics, was proved earlier by Abraham \cite{Abraham_1970}.

\begin{rmk}[On the relationship between harmonic maps and minimal surfaces]
\label{rmk:Relationship_harmonic_maps_minimal_surfaces}
It is useful to recall the relationship between harmonic maps $f$ from a closed, smooth Riemann surface $(\Sigma,g)$ into a closed Riemannian manifold $(N,h)$, and immersed minimal surfaces in $(N,h)$, since that relationship enriches our supply of examples. Chern and Goldberg \cite[Proposition 5.1]{Chern_Goldberg_1975} show that if $\Sigma=S^2$ and $f$ is a harmonic immersion, then $f$ is a minimal immersion. More generally, though they assume $(N,h)=\RR^3$ with its standard metric and allow $(\Sigma,g)$ to be a Riemann surface with boundary, Dierkes, Hildebrandt, and Sauvigny prove \cite[Theorem 2.6.1]{Dierkes_Hildebrandt_Sauvigny_minimal_surfaces} that a conformal map $f$ is minimal if and only if it is harmonic. According to their \cite[Definition 2.6.1]{Dierkes_Hildebrandt_Sauvigny_minimal_surfaces}, they may replace $\RR^3$ by $\RR^n$ for any $n\geq 2$ and more generally, by any Riemannian manifold $(N,h)$ of dimension $n\geq 2$. Moore \cite[Theorem 4.2.2]{Moore_introduction_global_analysis} proves a similar result, namely that (the image of) a (weakly) conformal harmonic map $f:(\Sigma,g) \to (N,h)$ is a minimal surface. By restricting to $\Sigma=S^2$, Moore \cite[Proposition 4.2.3]{Moore_introduction_global_analysis} recovers the result of Chern and Goldberg: a harmonic two-sphere, $f:(S^2,g_\round) \to (N,h)$, is automatically weakly conformal and hence a parametrized minimal surface. See also \cite[pp. 36, 77, 249--250, and 309--311]{Dierkes_Hildebrandt_Sauvigny_minimal_surfaces} and their discussion of Lichtenstein's Theorem on reparameterizing maps of the disk and \cite[pp. 249--250]{Dierkes_Hildebrandt_Sauvigny_minimal_surfaces} for the relationship between area and energy integrals and the minimization problem.
\end{rmk}

\begin{rmk}[On the interpretation of mean curvature flow as a gradient flow]
\label{rmk:Interpretation_mean_curvature_flow_gradient_flow}
While there is a wealth of references on mean curvature flow, relatively few treat it as gradient flow for the area (volume) function, thus making it less accessible to gradient flow methods pioneered by Simon \cite{Simon_1983, Simon_1985}. For interpretations of mean curvature flow as a gradient system, we refer the reader to Bellettini \cite[Remark 2.8 and Section 2.3]{Bellettini_lecture_notes_on_mean_curvature_flow}, Colding, Minicozzi, and Pedersen \cite[Section 1]{Colding_Minicozzi_Pedersen_2015bams}, Ilmanen \cite{Ilmanen_1994}, Mantegazza \cite[p. 7, second paragraph]{Mantegazza_2011_lectures_mcf}, Ritor\'{e} and Sinestrari \cite[Equation (4.3)]{Ritore_Sinestrari_mean_curvature_flow_isoperimetric_inequalities}, Smoczyk \cite{Smoczyk_2012}, and Zaal \cite{Zaal_2015}. Shi and Vorotnikov \cite{Shi_Vorotnikov_2019} provide a useful recent reference, with a view to applications. 
For introductions to mean curvature flow, we refer to Ecker \cite{Ecker_2004regularity_theory_mcf}, Mantegazza \cite{Mantegazza_2011_lectures_mcf}, Ritor\'{e} and Sinestrari \cite{Ritore_Sinestrari_mean_curvature_flow_isoperimetric_inequalities}.

For applications of the DeTurck trick \cite{DeTurck_1983} to convert mean curvature flow to a nonlinear parabolic partial differential equation and establish short-time existence, we refer to Andrews and Baker \cite{Andrews_Baker_2010}, Baker \cite{BakerThesis}, and Leng, Zhao, and Zhao \cite{Leng_Zhao_Zhao_2014}.

As in the case of Ricci flow, the interpretation of mean curvature flow as a gradient system can lead to the introduction of a time-varying family of Hilbert spaces --- a family of $L^2$ spaces defined by a measure that depends on the time-varying family of immersions \cite[Section 1.2, page 7]{Mantegazza_2011_lectures_mcf}.
\end{rmk}

%
%

\bibliography{/Users/pfeehan/Dropbox/LATEX/Bibinputs/master,/Users/pfeehan/Dropbox/LATEX/Bibinputs/mfpde}
\bibliographystyle{amsplain-nodash}

\end{document}